\documentclass[12pt, requno]{amsart}
\usepackage{amsmath}
\usepackage{amssymb}
\usepackage{epsfig}
\usepackage{graphicx}
\usepackage{color}
\usepackage{fullpage}
\definecolor{shadecolor}{gray}{0.875}
\usepackage{amscd}
\usepackage{comment}
\usepackage{hyperref}
\hypersetup{
    colorlinks=true,
    linkcolor=blue,
    filecolor=magenta,      
    urlcolor=cyan,
    }
\usepackage{tikz-cd}

\numberwithin{equation}{section}

\input xy
\xyoption{all}

\calclayout
\allowdisplaybreaks[3]

\theoremstyle{plain}
\newtheorem{prop}{Proposition}[section]

\newtheorem{theo}[prop]{Theorem}
\newtheorem{coro}[prop]{Corollary}

\newtheorem{lemm}[prop]{Lemma}

\theoremstyle{definition}
\newtheorem{defi}[prop]{Definition}
\newtheorem{notation}[prop]{Notation}

\newtheorem{rema}[prop]{Remark}

\newtheorem{exam}[prop]{Example}

\newtheorem{class}[prop]{Classification}

\def\rH{{\mathrm H}}

\def\Pic{\mathrm{Pic}}

\def\Nef{\mathrm{Nef}}

\def\Spec{\mathrm{Spec}}

\def\Pic{\mathrm{Pic}}

\def\Sec{\mathrm{Sec}}

\makeatother
\makeatletter

\author{Sho Tanimoto}
\address{Graduate School of Mathematics, Nagoya University, Furocho Chikusa-ku, Nagoya, 464-8602, Japan}
\email{sho.tanimoto@math.nagoya-u.ac.jp}

\title[Manin's conjecture for quartic del Pezzo fibrations]{On Manin's conjecture \\for quartic del Pezzo fibrations}

\begin{document}

\maketitle

\begin{abstract}
We generalize the homological sieve method, developed by Das, Lehmann, Tosteson, and the author, to study certain quartic del Pezzo fibrations, and we prove a version of Manin's conjecture over global function fields in these cases. Our proofs combine the $3$-dimensional positive-characteristic minimal model program, the geometry of the space of sections and the Abel--Jacobi mapping, and the homological sieve method.
Our quartic del Pezzo surfaces have Picard rank $2$, and admit two birational morphisms to non-split quadric surfaces. In particular, they do not possess any conic fibrations.
\end{abstract}

\tableofcontents

\section{Introduction}

One of the central themes in diophantine geometry is Manin's conjecture, which predicts the asymptotic formula for the counting function of rational points of bounded height on a smooth Fano variety defined over a number field. This conjecture was originally proposed and developed by Victor Batyrev, Yuri Manin, Emmanuel Peyre, and Yuri Tschinkel, and was subsequently revised by many other mathematicians (\cite{FMT, BM, Peyre, BT, Peyre03, Peyre17, LST18,  LS24}). See \cite{BGandManin} and references therein for a modern formulation of Manin's conjecture.

One may also consider Manin's conjecture over global function fields, and there is an influential heuristic in this setting by Batyrev \cite{Bat88}. 
The idea is that Manin's conjecture over global function fields can be approached through the geometry of moduli spaces of curves/sections on a Fano variety/Fano fibration because, by the valuative criterion, rational points on a smooth Fano variety over a global function field correspond to curves/sections on those varieties defined over the base finite field. Furthermore, Ellenberg and Venkatesh strengthened this heuristic argument by suggesting that homological stability of the moduli space combined with the Grothendieck--Lefschetz trace formula would lead to a proof of the asymptotic formula (\cite{EV05, EVW}). However, such a topological proof of Manin's conjecture was not available until \cite{DLTT25}.

In \cite{DLTT25}, Ronno Das, Brian Lehmann, Philip Tosteson, and the author developed a method we call the homological sieve, and in this method, we combined several ingredients to prove a version of Manin's conjecture for rational curves on smooth split quartic del Pezzo surfaces defined over finite fields. The ingredients are: (i) algebraic geometry (birational geometry of the moduli space of rational curves on quartic del Pezzo surfaces), (ii) arithmetic geometry (simplicial schemes and the Grothendieck--Lefschetz trace formula), (iii) algebraic topology (the inclusion-exclusion principle and Vassiliev type bar complexes), and (iv) elementary analytic number theory. 

The main goal of this paper is to establish the homological sieve in a genuinely relative setting and use it to prove a version of Manin's conjecture for nontrivial quartic del Pezzo fibrations. In particular, we show that for the quartic del Pezzo fibrations considered in this paper, and over finite fields of sufficiently large cardinality relative to the truncation parameter, the number of sections of bounded anticanonical height satisfies the asymptotic predicted by Manin's conjecture, including Peyre's leading constant.

The main novelty of the paper is that the homological sieve can be made to work for nontrivial Fano fibrations. Unlike the split surfaces treated in \cite{DLTT25}, the quartic del Pezzo surfaces considered here have Picard rank $2$, possess no conic bundle structure, and arise as generic fibers of nontrivial fibrations. Their geometry is instead controlled by two birational contractions to non-split quadric surfaces. Combining these contractions with the relative minimal model program and the Abel--Jacobi mapping of sections allows us to reduce the arithmetic counting problem to a new form of the homological sieve.
Although the general framework of the homological sieve was established in \cite{DLTT25}, substantial new ingredients are required in the relative setting. In particular, the poscheme governing the inclusion-exclusion argument is different, and the most delicate point is the definition and analysis of its saturated elements.

\subsection{The main results}
Let $k = \mathbb F_q$ be a finite field of characteristic $\geq 7$. Let $B$ be a smooth geometrically integral projective curve defined over $k$. Let $K(B)$ be the function field of $B$. In this paper, we consider the following quartic del Pezzo surfaces: let $Q$ be a non-split smooth quadric surface defined over $K(B)$ such that its base change $Q_{\overline{k}}$ has Picard rank $1$. Let $Z \in Q$ be a degree $4$ integral closed point on $Q$ such that the base change $Z_{\overline{k}}$ remains integral.
Let $\phi : X \to Q$ be the blow up of $Q$ along $Z$ and we assume that $X$ is a smooth quartic del Pezzo surface defined over $K(B)$. This del Pezzo surface has Picard rank $2$ over $K(B)$, and we consider Manin's conjecture for $X$. To this end, we consider the following model of $X$:

\begin{defi}
Let $\pi : \mathcal X \to B$ be a flat projective morphism with geometrically connected fibers whose generic fiber is isomorphic to $X$. We say $\pi : \mathcal X \to B$ is a smooth wonderful model of $X$ if the following properties hold:
\begin{itemize}
\item $\mathcal X$ is smooth;
\item $-K_{\mathcal X/B}$ is relatively big and nef; and
\item every geometric fiber of $\pi$ is integral and has at worst canonical singularities.
\end{itemize}

\end{defi}
In Theorem~\ref{theo:constructionofwonderful} and Corollary~\ref{coro:examplesofwonderful}, we exhibit plenty of examples of $X$ with smooth wonderful models.

Suppose that $X$ admits a smooth wonderful model $\pi : \mathcal X \to B$.
Let $\Nef_1(\mathcal X) \subset N_1(\mathcal X)$ be the nef cone of curves on $\mathcal X$. 
We also view $N_1(\mathcal X_\eta) \subset N_1(\mathcal X)$ as a subspace
and $\Nef_1(\mathcal X_\eta) \subset \Nef_1(\mathcal X)$ as a face.
We define
\[
\Nef_1(\mathcal X)_{\mathrm{sec}} = \{ \alpha \in \Nef_1(\mathcal X) \cap N_1(\mathcal X) \, |\, \mathcal X_b.\alpha = 1\},
\]
where $\mathcal X_b$ is the class of a general geometric fiber. 
Let
\[
\ell : \Nef_1(\mathcal X) \to \mathbb R_{\geq 0}
\]
be a piecewise linear, continuous, homogeneous function.
Let $\epsilon$ be a small positive rational number.
We define
\[
\Nef_1(\mathcal X)_{\mathrm{sec}, \epsilon} = \{ \alpha \in \Nef_1(\mathcal X)_{\mathrm{sec}} \, | \, \ell(\alpha) \geq -\epsilon K_{\mathcal X/B}.\alpha\},
\]
and 
\[
\Nef_1(\mathcal X_\eta)_{\epsilon} = \{ \beta \in \Nef_1(\mathcal X_\eta) \, | \, \ell(\beta) \geq -\epsilon K_{\mathcal X/B}.\beta\},
\]
We also define
\[
r(\pi) = \min\{-K_{\mathcal X/B}.\alpha >0\, | \alpha \in N_1(\mathcal X)_{\mathbb Z}, \, \mathcal X_b.\alpha = 0\},
\]
and
\[
m(\pi) = \min\{ \deg(-\sigma^*K_{\mathcal X/B}) \, | \, \sigma : B \to \mathcal X \text{ a nef section }\}.
\]
Let $\alpha \in \Nef_1(\mathcal X)_{\mathrm{sec}, \mathbb Z}$ be a nef class and let
\[
\Sec(\mathcal X/B, \alpha)
\]
be the space of sections of class $\alpha$ on $\mathcal X$.
Then we define the counting function as
\[
\mathsf N(\mathbb F_{q}, B, \mathcal X, \ell, \epsilon, d) = \sum_{\substack{\alpha \in \Nef_1(\mathcal X)_{\mathrm{sec}, \epsilon, \mathbb Z}, \\-K_{\mathcal X/B}.\alpha \leq m(\pi) + r(\pi)d}} \#\Sec(\mathcal X/B, \alpha)(\mathbb F_{q}).
\]
Here is our main theorem which is a version of Manin's conjecture for $X$:
\begin{theo}
\label{theo:main}
Assume that the characteristic of the ground field is $>5$ and $\pi : \mathcal X \to B$ is a smooth wonderful model of the generic fiber $\mathcal X_\eta \cong X$.
There exist a constant $\mathsf C > 0$ depending only on $B_{\overline{k}}$ and $\mathcal X_{\overline{k}}$ and a piecewise linear, continuous, homogeneous function $\ell : \Nef_1(\mathcal X) \to \mathbb R_{\geq 0}$, which is uniform in $\mathcal X$ in the sense of Remark~\ref{rema:independence_intro} and is positive on a dense open cone $\mathsf U \subset \Nef_1(\mathcal X_\eta)$, such that assuming $q^\epsilon > \mathsf C$, we have
\[
\mathsf N(\mathbb F_{q}, B, \mathcal X, \ell, \epsilon, d) \sim(1-q^{-1})^{-1} \alpha(\mathcal X_\eta, \Nef_1(\mathcal X_\eta)_\epsilon) \tau_{-K_{\mathcal X/B}}(\mathcal X)q^{m(\pi) + r(\pi)d}(r(\pi)d),
\]
as $d \to \infty$. Here $\alpha(\mathcal X_\eta, \Nef_1(\mathcal X_\eta)_\epsilon)$ is the alpha constant for $\Nef_1(\mathcal X_\eta)_\epsilon$ and $\tau_{-K_{\mathcal X/B}}(\mathcal X)$ is the Tamagawa number for $X$ both introduced in \cite{Peyre}. See Section~\ref{subsec:heightzeta} for their definitions.
\end{theo}

\begin{rema}
\label{rema:independence_intro}
The $\mathbb R$-vector space $N_1(\mathcal X)$ admits a natural basis, and $\ell$ is given in a uniform way in terms of this basis. In this sense, $\ell$ is independent of $\mathcal X$.
\end{rema}

We prove Theorem~\ref{theo:main} by proving a version of Peyre's all height approach of Manin's conjecture. In particular, we prove:
\begin{theo}
\label{theo:Peyre_intro}
We have
\[
\lim_{\ell(\alpha)  \to \infty} \frac{\#\Sec(\mathcal X/B, \alpha)(\mathbb F_{q})}{q^{-K_{\mathcal X/B}.\alpha}} = \tau_{-K_{\mathcal X/B}}(\mathcal X),
\]
under the assumption that $\mathcal X$ is a blow-up model.
\end{theo}

 See Theorem~\ref{theo:allhegiht} for more details.

\begin{rema}
Our proofs of Theorem~\ref{theo:main} also apply to the following quintic del Pezzo surfaces: let $Q$ be a non-split smooth quadric surface over $K(B)$ such that $Q_{\overline{k}}$ has Picard rank $1$. Let $Z \in Q$ be a closed point of degree $3$ such that $Z_{\overline{k}}$ remains integral. Let $\phi : X \to Q$ be the blow-up of $Q$ along $Z$, and assume that $X$ is a smooth quintic del Pezzo surface over $K(B)$. Our proofs imply a similar statement. Since the proofs are similar to those of Theorem~\ref{theo:main}, we will not include a full proof of this statement in this paper.
\end{rema}

\subsection{The method of proof}

Our proofs combine the following ingredients: 
\begin{enumerate}
\item the $3$-dimensional positive-characteristic minimal model program;
\item the geometry of the space of sections and the Abel--Jacobi mapping; and
\item the homological sieve method. 
\end{enumerate}

\subsubsection*{The $3$-dimensional positive-characteristic minimal model program}

Our quartic del Pezzo surface $X$ has Picard rank $2$ and admits a birational morphism $\phi : X \to Q$ to a non-split quadric surface. The $2$-ray game gives the second birational morphism $\phi' : X \to Q'$ to another non-split quadric surface. Our counting argument exploits these two birational morphisms. To this end, we need to understand how to extend these birational morphisms to our integral model $\pi : \mathcal X \to B$. We assume that $\mathcal X$ is a smooth wonderful model of $X$. Under this assumption, we apply a relative $3$-dimensional minimal model program to $\mathcal X$ (\cite{HX15, BirkarASENS, HNT}): 
let $D$ be the flat closure of a hyperplane section from $Q$. Then, after applying a sequence of $D$-flops over $B$,
\[
\xymatrix{
 \mathcal X \ar@{-->}[rr]^\psi \ar[rd]_\pi & & \widetilde{\mathcal X} \ar[dl]^{\widetilde{\pi}} \\
 & B &
}
\]
the strict transform $\widetilde{D}$ of $D$ on $\widetilde{\mathcal X}$ becomes relatively $\widetilde{\pi}$-big and $\widetilde{\pi}$-nef and $\widetilde{\pi} : \widetilde{\mathcal X} \to B$ is a smooth wonderful model of $X$ (Proposition~\ref{prop:Dflops}). 
Note that $\psi$ is the composition of a sequence of $D$-flops, so the height functions associated to $-K_{\mathcal X/B}$ and $-K_{\widetilde{\mathcal X}/B}$ are the same and the Tamagawa numbers for $\mathcal X$ and $\widetilde{\mathcal X}$ are also the same (Lemma~\ref{lemm:flopsTamagawa}). Thus a counting problem on $\mathcal X$ is equivalent to a counting problem on $\widetilde{\mathcal X}$, so we can work with $\widetilde{\mathcal X}$ instead.

Next, in Proposition~\ref{prop:blowupquadric} and Corollary~\ref{coro:blowupofquadric}, we show that assuming $D$ is relatively $\pi$-big and $\pi$-nef, it defines a $B$-morphism
\[
\phi : \mathcal X \to \mathcal Q,
\]
where the target is a quadric surface bundle $\pi_{\mathcal Q} : \mathcal Q \to B$ whose generic fiber is $Q$ such that $\mathcal Q$ is smooth and every geometric fiber is integral. Moreover, $\phi : \mathcal X \to \mathcal Q$ is a blow-up along a geometrically integral smooth curve $\mathcal Z \subset \mathcal Q$ which is the closure of $Z \subset Q$.
Thus, it suffices to work with a blow-up model $\phi : \mathcal X \to \mathcal Q$ along a smooth geometrically integral multisection $\mathcal Z \subset \mathcal Q$ of degree $4$.

\subsubsection*{The geometry of the space of sections and the Abel--Jacobi mapping}

We need to understand the space of sections on a blow-up model $\phi : \mathcal X \to \mathcal Q$. The key to this study is the Abel--Jacobi mapping for the space of sections on $\mathcal Q$ developed in \cite{HT12} and \cite{TT25}.
Let $C$ be the discriminant curve of $\mathcal Q$. Then the intermediate Jacobian of $\mathcal Q$ is $\Pic^0(C)$, and we have the Abel--Jacobi mapping:
\[
\mathrm{AJ}: \Sec(\mathcal Q/B, \phi_*\alpha) \to \Pic^d(C),
\]
for some $d$. When the height of $\phi_*\alpha$ is sufficiently large, \cite{HT12} shows that this mapping is a Zariski open subset of a projective bundle over the image.

On the other hand, let $k(\alpha) = E.\alpha$ where $E$ is the exceptional divisor of $\phi$. Then we have the following morphism
\[
\Phi_\alpha : \Sec(\mathcal X/B, \alpha) \to \mathrm{Hilb}^{[k(\alpha)]}(\mathcal Z), \quad [\sigma: B \to \mathcal X] \mapsto \phi_*([\sigma(B)\cap E]).
\]
Hence, we obtain the morphism:
\[
\Phi_\alpha \times \mathrm{AJ} : \Sec(\mathcal X/B, \alpha) \to \mathrm{Hilb}^{[k(\alpha)]}(\mathcal Z) \times \Pic^d(C).
\]
Regarding this morphism, in Proposition~\ref{prop:dominantsmooth} and Theorem~\ref{theo:spaceofsections_structure}, we show that $\Phi_\alpha \times \mathrm{AJ}$ is a smooth morphism and the structure morphism
\[
\Phi_\alpha \times \mathrm{AJ} : \mathrm{Sec}(\mathcal X/B, \alpha) \to \mathrm{Hilb}^{[k(\alpha)]}(\mathcal Z) \times \Pic^d(C)
\]
 is a Zariski open subset of a projective bundle over the image.
 In particular, over $\mathbb C$, the composition
 \[
 \mathrm{Sec}(\mathcal X/B, \alpha) \to \mathrm{Hilb}^{[k(\alpha)]}(\mathcal Z) \times \Pic^d(C)
 \to \Pic^{k(\alpha)} (\mathcal Z) \times \Pic^d(C)
 \]
 coincides with the Abel--Jacobi mapping to the intermediate Jacobian of $\mathcal X$, and it is a MRC fibration (Corollary~\ref{coro:MRC}).
Thus, a key to our counting problem is to understand the number of $\mathbb F_q$-points on this open subset and we can appeal to the inclusion-exclusion principle and the homological sieve method to achieve such a goal.

\subsubsection*{Homological sieve}

Our geometric understanding of the space of sections as above enables us to apply the homological sieve method, a formalized inclusion-exclusion principle developed in \cite{DLTT25}.
We define the Zariski open subset
\[
U_{k(\alpha)} = \{ [w] \in \mathrm{Hilb}^{[k(\alpha)]}(\mathcal Z) \, |\, \text{ $w$ is admissible for $\mathcal Q$} \} \subset \mathrm{Hilb}^{[k(\alpha)]}(\mathcal Z).
\]
Note that we realize
 \[
 \mathrm{Sec}(\mathcal X/B, \alpha) \to U_{k(\alpha)} \times \Pic^d(C),
 \]
 as a Zariski open subset of a (generically) projective bundle
 \[
 \mathbb P(E) \to U_{k(\alpha)} \times \Pic^d(C).
 \]
 Then we consider the corresponding $\mathbb G_m$-torsor above $ \mathrm{Sec}(\mathcal X/B, \alpha)$ which is a Zariski open subset
 \[
 \widetilde{M}_\alpha \subset E.
 \]
 To count $\mathbb F_q$-points on $\widetilde{M}_\alpha$, we analyze its complement $E \setminus  \widetilde{M}_\alpha$.
 To this end we define the poscheme (a scheme with a poset structure)
\[
\mathfrak H =\{x = (x_B, x_{\mathcal Z}, x_C, \hat{1}) \in \mathrm{Hilb}^{\mathrm{fin}}(B) \times \mathrm{Hilb}^{\mathrm{fin}}(\mathcal Z)\times \mathrm{Hilb}^{\mathrm{fin}}(C) \times\{ \hat{1}\} \, | \, \text{$x$ satisfies $(*)$}\},   
\]
where $(*)$ means that 
\[
p_{\mathcal Z \to B}^*x_B \subset x_\mathcal Z, \quad p_{C \to B}^*x_B \subset x_C.
\]
Here the poset structure is given by the inclusion of subschemes component-wise.
Then $E \setminus  \widetilde{M}_\alpha$ is stratified by a stratification
\[
\mathsf Z \subset (U_{k(\alpha)} < \mathfrak H)\times_{U_{k(\alpha)}} E,
\]
where for $(w< x, L) \in (U_{k(\alpha)} < \mathfrak H)(\overline{k})\times \Pic^d(C)(\overline{k})$, the fiber $\mathsf Z_{w < x, L}$ is a linear subspace of $E_{w, L}$.
We then follow the strategy of \cite{DLTT25}. 
We consider a truncation $P\subset (U_{k(\alpha)} < \mathfrak H)$
and the bar complex $B(P, \mathsf Z_{P})$, which is a certain simplicial scheme, i.e., a contravariant functor from $\Delta$ to the category of separated schemes, and it assigns $[n] = \{ 0, \cdots, n\}$ to the scheme
\[
\{ (w, L), w < x_0 \leq \cdots \leq x_n \in \mathfrak H, z \in \mathsf Z_{w < x_n, L} \, | \, (w < x_i) \in P\}.
\]
\cite[Theorem 7.5]{DLTT25} shows that \'etale cohomology of $E \setminus  \widetilde{M}_\alpha$ is isomorphic to \'etale cohomology of the bar complex $B(P, \mathsf Z_{P})$ in a certain range.
Then, using the spectral sequence associated to a stratification of $B(P, \mathsf Z_{P})$ and the Grothendieck--Lefschetz trace formula for simplicial schemes, we conclude that up to an error term, we have
\begin{align*}
\#\widetilde{M}_\alpha(k) \sim 
\#\Pic^0(C)(k)q^{h_{\mathcal Q}(\alpha) +3-2g(B)-g(C)} \sum_{(w \leq x) \in (U_{k(\alpha)}(k) \leq \mathfrak H^\circ(k))} \mu_k(w, x)q^{-\gamma(x)},
\end{align*}
as $\ell(\alpha)  \to \infty$,
where $\mu_k$ is the M\"obius function for the poset $\mathfrak H^\circ(k) \subset \mathfrak H(k)$ which is the set of saturated elements.
Then, we consider the following virtual height zeta function:
\[
\mathsf Z(t) = \sum_{\mathsf k = 0}^\infty q^{\mathsf k}\left(\sum_{(w \leq x) \in (U_{\mathsf k}(k) \leq \mathfrak H^\circ(k))} \mu_k(w, x)q^{-\gamma(x)}\right)t^{\mathsf k}.
\]
Since the M\"obius function is multiplicative, the above zeta function becomes the Euler product.
We compute this Euler product explicitly and we confirm
\[
\frac{\#\Pic^0(C)(k)q^{\mathsf k+2-2g(B)-g(C)}}{1-q^{-1}}\left(\sum_{(w \leq x) \in (U_{\mathsf k}(k) \leq \mathfrak H^\circ(k))} \mu_k(w, x)q^{-\gamma(x)}\right) \sim  \tau_{-K_{\mathcal X/B}}(\mathcal X),
\]
as $\mathsf k \to \infty$. Altogether we conclude
Theorem~\ref{theo:Peyre_intro}.

\subsection{Related work}

\subsubsection*{Manin's conjecture for del Pezzo surfaces}

Manin's conjecture for smooth del Pezzo surfaces defined over number fields has been extensively studied. When the degree of a del Pezzo surface is $\geq 6$, the conjecture follows from \cite{BTtoric96, BTtoric98, CLT01} because such a surface is a toric surface. Manin's conjecture for split quintic del Pezzo surfaces was proved over $\mathbb Q$ in \cite{delaBre} and over arbitrary number fields in \cite{BD24}. Some non-split quintic del Pezzo surfaces with a conic bundle structure are handled in \cite{HBL25}. For del Pezzo surfaces of degree $\leq 4$, the only proved case is the example treated in \cite{DBB11} which settles one example of a non-split quartic del Pezzo surface over $\mathbb Q$. It has Picard rank $5$, is $\mathbb Q$-rational, and it comes with a conic bundle structure. This conic bundle structure was a key to their counting argument. We should also note that \cite{BS19} establishes upper and lower bounds of the correct magnitude for quartic del Pezzo surfaces with conic bundle structures.
In particular, our result is new in the sense that our quartic del Pezzo surface has Picard rank $2$ and it does not come with any conic bundle structure. Instead, it admits two birational morphisms to non-split quadric surfaces. It would be interesting to explore whether analytic methods can be applied to such surfaces.

\subsubsection*{Manin's conjecture over global function fields}

There are fewer studies on Manin's conjecture over global function fields. As mentioned before, there is an influential heuristic for this conjecture by Batyrev and Ellenberg--Venkatesh (\cite{Bat88} and \cite{EV05, EVW}). The study of Manin's conjecture over global function fields was further developed by David Bourqui (e.g., \cite{Bou03, Bou09, Bou16}). 

The first implementation of this strategy of Batyrev--Ellenberg--Venkatesh was obtained in \cite{DLTT25} for quartic del Pezzo surfaces over finite fields, building on the study of the Cohen--Jones--Segal conjecture in \cite{DT24}.
Moreover, \cite{FHL26} generalized \cite{DLTT25} to higher genus curves by using the homological sieve method developed in \cite{DLTT25} combined with Betti number bounds for the space of curves, generalizing \cite[Theorem B.1]{DLTT25} by Sawin--Shusterman. These Betti number bounds play a crucial role in the current paper when we deal with the error term.

The follow-up \cite{Tani26} to \cite{DLTT25} applied the homological sieve method to quintic del Pezzo surfaces. In particular, this paper settles a version of Peyre's all height approach to Manin's conjecture for rational curves on split quintic del Pezzo surfaces. A further strengthened theorem for quintic del Pezzo surfaces was obtained in \cite{BFG26}. Upper bounds were obtained in \cite{GH24, Glas25} for low degree del Pezzo surfaces. For a modern formulation of Manin's conjecture over global function fields, see \cite{LT26}.

\subsubsection*{Sections of del Pezzo fibrations}

The geometry of the space of sections on a Fano fibration is a fascinating subject. One of the pioneering papers on this is \cite{HT14} which explored the geometry of the spaces of sections of quartic del Pezzo fibrations and its interaction with the Abel--Jacobi mapping. Our method also relies on the Abel--Jacobi mapping and \cite{HT14} is one source of inspiration for our work. The case of quadric surface bundles was handled in \cite{HT12}, and our result is based on the geometric description of the space of sections of a quadric surface bundle. Geometric Manin's conjecture for del Pezzo fibrations, which predicts properties of the space of sections such as dimension and irreducibility, was first explored in \cite{LT19} and \cite{LT21}. This was extended in part to higher dimension in \cite{LRT25}. Homological stability for the space of sections was first explored in \cite{TT25}, and a geometric construction in \cite{TT25} is also key to our understanding of the space of sections of quartic del Pezzo fibrations.

\bigskip

\noindent
{\bf Acknowledgements:}
The author thanks Hiroto Akaike, Enhao Feng, Matthew Hase-Liu, Brendan Hassett, Natsume Kitagawa, Brian Lehmann, Yusuke Nakamura, and Hiromu Tanaka for answering the author's questions.
In particular, the author would like to thank Enhao, Matthew, Brendan, and Brian for various discussions. The author also thanks Tim Browning, Enhao Feng, Jakob Glas, Matthew Hase-Liu, Natsume Kitagawa, Dan Loughran, and Hiromu Tanaka for comments on an early draft. In particular, Enhao and Matthew pointed out some issues with our previous proofs in Section~\ref{sec:geometryofthespace}.

The author was partially supported by JST FOREST program Grant number JPMJFR212Z and by JSPS KAKENHI Grant-in-Aid (B) 23K25764.

ChatGPT was used to proofread this paper and to revise the introduction.

\section{Preliminaries}
\label{sec:preliminaries}

First, we record some preliminaries for this paper.

\noindent
{\bf Notation:}
We assume that our ground field $k$ is a field of arbitrary characteristic.
Let $X$ be a scheme over $k$ and $k \subset k'$ be an extension. Then we denote its base change by $X_{k'}$. We denote linear equivalence by $\sim$ and $\mathbb Q$-linear equivalence by $\sim_{\mathbb Q}$. Let $f : X \to Y$ be a $k$-projective morphism. We denote the relative real N\'eron-Severi space by $N^1(X/Y)$ and its dual space of $1$-cycles by $N_1(X/Y)$.
When $Y = \Spec (k)$ we denote them by $N^1(X)$ and $N_1(X)$ respectively.
We denote the pseudo-effective cones of divisors and curves by
\[
\overline{\mathrm{Eff}}^1(X), \quad \overline{\mathrm{Eff}}_1(X),
\]
respectively, and we also denote the nef cones of divisors and curves by
\[
\mathrm{Nef}^1(X), \quad \mathrm{Nef}_1(X),
\]
respectively. We denote the lattice generated by integral curves by $N_1(X)_{\mathbb Z} \subset N_1(X)$ and for any subset $\mathsf C \subset N_1(X)$, we define $$\mathsf C_{\mathbb Z} := \mathsf C\cap N_1(X)_{\mathbb Z}.$$

For any smooth geometrically integral projective curve $B$ over $k$, we denote the set of closed points of $B$ by $|B|$. For any $b \in |B|$, we set $|b| = [k(b):k]$.

\bigskip

\subsection{Preliminaries on singularities}

The following lemma is well-known to experts. We record it for completeness:

\begin{lemm}
\label{lemm:bignefterminal}
We assume that $k$ is an algebraically closed field.
Let $X$ be a normal projective variety, with only terminal singularities, defined over $k$. When the characteristic of $k$ is positive, we may assume that the dimension of $X$ is $\leq 3$. Let $D$ be a big and nef $\mathbb Q$-Cartier $\mathbb Q$-divisor on $X$. Then there exists an effective $\mathbb Q$-divisor $0 \leq D'\sim_{\mathbb Q} D$ such that $(X, D')$ is a terminal pair.
\end{lemm}

\begin{proof}
This is well-known in characteristic $0$. For a reference, see the proof of \cite[Theorem 2.3]{LTT14} or \cite[Remark 3.4.3]{BGandManin}.
Thus we may assume that the characteristic of $k$ is positive and the dimension of $X$ is at most $3$.
By Wilson's theorem, there exists an effective Cartier divisor $N \geq 0$ such that for a sufficiently small rational number $t > 0$, $D-tN$ is ample. For Wilson's theorem, see \cite[Lemma 3.4.2]{BGandManin} which is stated in characteristic $0$, but its proof works in positive characteristic.
Let $\rho : \widetilde{X} \to X$ be a log resolution of the pair $(X, N)$. Its existence in dimension $3$ is guaranteed by \cite[Propositions 4.1 and 4.2]{ResolutionI} and \cite[Main theorem]{ResolutionII}.
Since $X$ has only terminal singularities, there exists a $\rho$-exceptional effective $\mathbb Q$-divisor $E \geq 0$ such that
\[
K_{\widetilde{X}} = \rho^*K_{X} + E,
\]
and the support of $E$ coincides with the exceptional locus of $\rho$.
Write $A_t = D - tN$ which is an ample $\mathbb Q$-divisor.
We can write
\[
\rho^*D = \rho^*A_t + t\rho^*N = \rho^*A_t + t\rho^{-1}_*N + tE',
\]
where $\rho^{-1}_*N$ is the strict transform of $N$ on $\widetilde{X}$ and $E'\geq 0$ is an effective and $\rho$-exceptional divisor on $\widetilde{X}$.
Let $E''$ be an effective and $\rho$-exceptional divisor such that $-E''$ is $\rho$-ample.
Let $t'$ be a sufficiently small positive rational number. Then $\rho^*A_t -t'E''$ is ample on $\widetilde{X}$. Let $m$ be a sufficiently large positive integer such that $m(\rho^*A_t -t'E'')$ is very ample.
Then, by Bertini's theorem, there exists $0 \leq A'  \sim m(\rho^*A_t -t'E'')$ such that $A' + \rho^{-1}_*N + E$ is log smooth. 
Put $$0 \leq D' = \rho_*\left(\frac{1}{m}A' + t\rho^{-1}_*N + tE' + t'E''\right) \sim_{\mathbb Q} D.$$
We claim that $(X, D')$ is a terminal pair.
Indeed, since we have $$\rho^*D' = \rho^{-1}_*D' + tE' + t'E'',$$ by the negativity lemma (\cite[Section 2.3]{BirkarASENS}), we obtain
\[
K_{\widetilde{X}} + \rho^{-1}_*D' = \rho^*(K_X + D') -tE' - t'E'' + E.
\]
Since $t, t'$ are sufficiently small and $m$ is also sufficiently large, our assertion follows from \cite[Corollary 2.11]{Kollar13}.
\end{proof}

Next we record the following lemma:

\begin{lemm}
\label{lemm:factorialsing}
Let $k$ be an algebraically closed field and $X$ be a Gorenstein $\mathbb Q$-factorial terminal threefold defined over $k$.
Then $X$ has only factorial singularities.
\end{lemm}

\begin{proof}
In characteristic $0$, this follows from a result of Reid and Ue. (See \cite[Lemma 5.1]{Kawamata88}.) In positive characteristic, it follows from \cite[Corollary 2.30 (1)]{Kollar13} that $X$ has only isolated singularities. Moreover \cite[Theorem 6.4]{Tanaka25} implies that for any singularity $x \in X$, the completion $\widehat{\mathcal O}_{X, x}$ of the local ring $\mathcal O_{X, x}$ is a hypersurface singularity. By \cite[Corollary 2.10]{DLM10}, the class group $\mathrm{Cl}(\widehat{\mathcal O}_{X, x})$ is torsion free. Since we have a natural injection $\mathrm{Cl}(\mathcal O_{X, x})\hookrightarrow \mathrm{Cl}(\widehat{\mathcal O}_{X, x})$, our assertion follows.

\end{proof}

\subsection{$\mathsf L^1$-traces}

First we recall the following definition:

\begin{defi}[{\cite[Definition 2.2]{DLTT25}}]
Assume that $K$ is a field equipped with a norm $|\cdot| : \overline{K} \to \mathbb R$.
Let $V$ be a finite dimensional $K$-vector space with an endomorphism $F$.
The $\mathsf L^1$-trace of $F$ on $V$ is
\[
|F, V| = \sum_{\lambda \in \overline{K}} |\lambda| \dim (V\otimes \overline{K})_\lambda,
\]
where $(V\otimes \overline{K})_\lambda$ is the generalized eigenspace of $F$ with an eigenvalue $\lambda$.

\end{defi}
Note that we have $|\mathrm{Tr}(F\curvearrowright V)| \leq |F, V|$ and if $W$ is a $F$-subquotient of $V$, then we have $|F, W|\leq |F, V|$. See \cite[Section 2.1]{DLTT25} for more properties of $\mathsf L^1$-traces.

\subsubsection*{Cohomology of configuration spaces}

Let $P$ be a finite poset with the bottom element $\hat{0}$ and the top element $\hat{1}$.
We let $\rH^*(P)$ be the reduced cohomology with $\mathbb Q_\ell$-coefficients of the nerve of the open interval $(\hat{0}, \hat{1})$ in $P$. Let $\Pi$ be a finite set and let $\mathrm P(\Pi)$ be the poset of set partitions of $\Pi$. This admits the bottom and the top elements. The following theorem follows from \cite[Theorem 3.3 (ii)]{Petersen}. The normalization of homological degrees used here is based on \cite[Theorem 2.7]{DLTT25}:

\begin{theo}[{\cite[Theorem 3.3 (ii)]{Petersen}}]
\label{theo:Petersen}
Let $Y$ be a quasi-projective variety over a field $k$ and $\mathcal F$ be an $\ell$-adic constructible sheaf on $Y$ in the \'etale topology. Then there is a Galois and $\mathfrak S_m$-equivariant spectral sequence converging to 
\[
\rH_{\textnormal{\'et}, c}^*(\mathrm{Conf}_m(Y)_{\overline{k}}, \mathcal F_{\overline{k}}^{\boxtimes m}),
\]
with the total $E_1$-page given by
\[
\bigoplus_{\substack{\textnormal{$p$: a set partition} \\ \textnormal{ of degree $m$}}} \bigotimes_{\lambda \in p} \rH_{\textnormal{\'et}, c}^*(Y_{\overline{k}}, \mathcal F_{\overline{k}})\otimes \rH^{|\lambda|-3}(\mathrm P(\lambda))[-|\lambda| + 1].
\]

\end{theo}

Here is a key estimate for our proofs:

\begin{prop}
\label{prop:keyestimate}
Assume that our ground field $k$ is $\mathbb F_q$. Let $\mathfrak L = \oplus_{n, i} \mathfrak L_{n, i}$ be a bigraded sheaf of $\ell$-adic constructible sheaves on $Y$ with $\mathbb Q_\ell$-coefficients in the \'etale topology. Suppose that there is also a bigraded vector space $\mathfrak L' = \oplus_{n, i} \mathfrak L'_{n, i}$ with $\mathbb Q_\ell$-coefficients such that each $\mathfrak L'_{n, i}$ is finite dimensional.
Assume that both are concentrated in cohomological degrees $i \leq -2$ and in the first grading $n \geq 1$. 
Suppose also that there is a constant $E \geq 1$ such that 
\begin{itemize}
\item we have
\[
|\mathrm{Frob}, \mathfrak L'_{n, i}| \leq Eq^{i/2}.
\]
\item we have
\[
|\mathrm{Frob}, \rH_{\textnormal{\'et}, c}^j(Y_{\overline{k}},  \mathfrak L_{\overline{k}})_{n, i})| \leq Eq^{j/2+i/2} .
\]
\end{itemize}
We fix integers $I, \mathsf c_1>0$, $\mathsf k\geq0$.
Then let $W_{n, i, m}$ be 
\[
\bigoplus_{j_1 + j_2 = i}\left(\rH^{j_1}_{\textnormal{\'et}, c}(\mathrm{Conf}_m(Y)_{\overline{k}}, \mathfrak L_{\overline{k}}^{\boxtimes m} )\otimes \mathfrak L'\right)_{n, j_2}.
\]
Assume that $q > 2^{22}E^2$. Then we have
\[
\left|\mathrm{Frob},\bigoplus_{m \geq 0} \bigoplus_{m \leq n \leq m+\mathsf k} \bigoplus_{ i \leq -I - m + \mathsf c_1} (W_{n, i, m})_{\mathfrak S_m}\right|  = O((\mathsf k + 1)(4E)^{\mathsf k}4^{I - \mathsf c_1}q^{-I/2 + \mathsf c_1/2})
\]
where the implied constant is independent of $q, I, \mathsf k, \mathsf c_1$.

\end{prop}

\begin{proof}
By Theorem~\ref{theo:Petersen} and \cite[Proposition 2.8]{DLTT25}, it suffices to bound
\begin{equation}
\label{equation:L1trace}
\left|\mathrm{Frob}, \bigoplus_m \bigoplus_{m \leq n \leq m+\mathsf k}\bigoplus_{ i \leq -I -m + \mathsf c_1}\bigoplus_{j_1 + j_2 = i}\bigoplus_{p': \text{partition of $m$} }\left (\left( \bigotimes_{\lambda' \in p'} \rH_{\textnormal{\'et}, c}^{*}(Y_{\overline{k}}, \mathfrak L_{\overline{k}})[-\lambda'+1]\right) \otimes \mathfrak L'\right)_{n, j_1, j_2} \right|.
\end{equation}
Note that $j_1$ must satisfy $0 \leq j_1 - m + |p'|\leq 2|p'|$, so we conclude that $0 \leq j_1 \leq 2m$.
Using the proofs of \cite[Lemma 2.4 and Lemma 2.5]{DLTT25}, the $\mathsf L^1$-trace of 
\[
\left(\bigotimes_{\lambda' \in p'} \rH_{\textnormal{\'et}, c}^{*}(Y_{\overline{k}}, \mathfrak L_{\overline{k}})[-\lambda'+1]\right)_{n, j_1, j_2}
\]
is bounded by 
\[
(2E)^n2^{j_1 + |p'|}q^{j_1/2} (2/\sqrt{q})^{-j_2} \leq (2E)^n2^{3m}q^{j_1/2} (2/\sqrt{q})^{-j_2}.
\]
Thus the $\mathsf L^1$-trace of 
\[
\left(\left( \bigotimes_{\lambda' \in p'} \rH_{\textnormal{\'et}, c}^{*}(Y_{\overline{k}}, \mathfrak L_{\overline{k}})[-\lambda'+1]\right) \otimes \mathfrak L'\right)_{n, j_1, j_2} 
\]
is bounded by
\[
2^{3m}(4E)^{n}q^{j_1/2}(4/\sqrt{q})^{-j_2} \leq 2^{7m}(4E)^{n}(4/\sqrt{q})^{-i}.
\]
This implies that (\ref{equation:L1trace}) is bounded by
\[
O\left(\sum_{m}\sum_{i \leq -I-m+\mathsf c_1}\mathsf (\mathsf k+1)(2m+1)2^{7m}(4E)^{m + \mathsf k}p(m)(4/\sqrt{q})^{-i} \right),
\]
where $p(m)$ is the number of partitions of $m$. 
This is bounded by
\[
O\left(\sum_{m}\mathsf (\mathsf k+1)(4E)^{m+\mathsf k}(2m+1)2^{7m}p(m)(4/\sqrt{q})^{I +m- \mathsf c_1} \right).
\]
Using the asymptotic formula for $p(m)$, one can deduce that this is bounded by
\[
O((\mathsf k+1)(4E)^{\mathsf k}4^{I - \mathsf c_1}q^{-I/2 + \mathsf c_1/2}).
\]
Thus our assertion follows.

\end{proof}

\subsection{Betti number bounds for the space of sections}

First, we recall the following theorem which is a generalization of \cite[Theorem B.1]{DLTT25} by Sawin--Shusterman:

\begin{theo}[Feng--Hase-Liu, {\cite[Theorem A]{FHL26}}]
\label{theo:FengHaseLiu}
Let $k$ be a separably  closed field and $B$ be a smooth projective curve over $k$.
Let $X$ be a projective variety over $k$ and $H$ be an ample divisor on $X$.
Then there exists a constant $\mathsf C > 0$ depending only on $B, X, H$ such that for any $e  \in \mathbb Z_{>0}$, the morphism scheme
\[
\mathrm{Mor}(B, X, \leq e)
\]
parametrizing morphisms $s : B \to X$ of $H$-degree $\leq e$ satisfies 
\[
\sum_i \dim \rH^i_{\textnormal{\'et}, c}(\mathrm{Mor}(B, X, \leq e), \mathbb Q_\ell) \leq \mathsf C^{e + 1}.
\]
\end{theo}

As a corollary, we have the following statement:

\begin{coro}
\label{coro:FengHaseLiu}
Let $k$ be a separably  closed field and $B$ be a smooth projective curve over $k$.
Let $\pi : \mathcal X \to B$ be a flat projective morphism from a 
projective variety $\mathcal X$.
Let $H$ be a Cartier divisor on $\mathcal X$ whose restriction to the generic fiber $\mathcal X_\eta$ is ample.
Then there exists a constant $\mathsf C$ depending only on $B, \mathcal X, H$ such that for any $e \in \mathbb Z_{>0}$, the space of sections
\[
\Sec(\mathcal X/B, \leq e)
\]
parametrizing sections of $H$-degree $\leq e$ satisfies
\[
\sum_i \dim \rH^i_{\textnormal{\'et}, c}(\mathrm{Sec}(\mathcal X/B, \leq e), \mathbb Q_\ell) \leq \mathsf C^{e + 1}.
\]
\end{coro}

\begin{proof}
Since $H$ is relatively big, one can find a positive integer $m$, a vertical Cartier divisor $N$, and an ample divisor $A$ on $\mathcal X$ such that
\[
A \leq mH + N.
\]
Since $N$ is vertical, there exists a constant $\mathsf c \geq 1$ such that for any class $\alpha$ of sections, we have
\[
 N.\alpha \leq \mathsf c.
\]
Thus we have
\[
H.\alpha \leq e \implies A.\alpha  \leq me + \mathsf c.
\]
We apply Theorem~\ref{theo:FengHaseLiu} to $(X, A)$ and we obtain a constant $\mathsf C_A$. 
Now note that for any class $\alpha$ of sections, we have
\[
\mathrm{Mor}(B, \mathcal X, \alpha) \cong \Sec(\mathcal X/B, \alpha) \times \mathrm{Aut}(B).
\]
Using the K\"unneth formula as in \cite[Chapter VI Corollary 8.23]{Milne}, $\mathsf C = \mathsf C_A^{\max\{m, \mathsf c\}}$ witnesses our assertion.
\end{proof}

\section{Quartic del Pezzo fibrations}
\label{sec:geometry}

In this section, we study the geometry of quartic del Pezzo fibrations. Our ground field $k$ is a perfect field of arbitrary characteristic.

\subsection{Models of quartic del Pezzo fibrations}

 The main protagonist of this paper is the following:

\begin{defi}
Let $X$ be a smooth geometrically integral projective surface defined over a field $K$. We say $X$ is a {\it del Pezzo surface} if the anticanonical divisor $-K_X$ is ample. The {\it degree} of $X$ is given by $(-K_X)^2$. We say $X$ is a {\it quartic del Pezzo surface} if $X$ is a del Pezzo surface such that $(-K_X)^2 = 4$.
\end{defi}

Let $B$ be a smooth geometrically integral projective curve defined over a perfect field $k$ and $K = k(B)$ be its function field.
Next we discuss models of a smooth del Pezzo surface $X$ defined over $K$:

\begin{defi}
Let $X$ be a smooth del Pezzo surface defined over $K$. An {\it integral model of $X$} is a flat projective $k$-morphism $\pi : \mathcal X \to B$ with $\mathcal X$ normal, geometrically integral, and projective such that the generic fiber $\mathcal X_\eta$ is isomorphic to $X$.
Such an integral model always exists by taking the flat closure in an integral model of projective space and applying the normalization.

We say $\pi$ is a {\it smooth model} if $\mathcal X$ is regular. As $k$ is a perfect field, a smooth model also exists by applying resolution of singularities which was established in dimension $3$ by \cite{ResolutionI, ResolutionII}.
\end{defi}

\begin{defi}
Let $X$ be a quartic del Pezzo surface defined over $K$. An integral model $\pi : \mathcal X \to B$ of $X$ is called a {\it standard model} if the following properties hold:
\begin{itemize}
\item $\mathcal X$ has only Gorenstein terminal singularities;
\item every fiber of $\pi_{\overline{k}}$ is integral; and
\item $-K_{\mathcal X}$ is relatively $\pi$-ample.
\end{itemize}
\end{defi}

Such a model always exists when $k$ is an algebraically closed and the characteristic of $k$ is not equal to $2$:
\begin{theo}[{\cite[1. 10 Theorem]{Corti96} and \cite[Theorem 1.5]{Kitagawa}}]
\label{theo:standardmodel}
Assume that $k$ is an algebraically closed field and the characteristic of $k$ is not equal to $2$.
Let $B$ be a smooth integral projective curve defined over $k$ and $K = k(B)$ be its function field. Let $X$ be a quartic del Pezzo surface defined over $K$. Then a standard model of $X$ always exists.
\end{theo}

\begin{proof}
Over $\mathbb C$, this was achieved in \cite[1. 10 Theorem]{Corti96}. Over $k$, this was established by \cite[Theorem 1.5]{Kitagawa}.
\end{proof}

For actual applications, it is useful to have the following models:
\begin{defi}
Let $X$ be a quartic del Pezzo surface defined over $K$. An integral model $\pi : \mathcal X \to B$ of $X$ is called a {\it factorial model} if the following properties hold:
\begin{itemize}
\item $\mathcal X$ has only factorial terminal singularities;
\item every fiber of $\pi_{\overline{k}}$ is integral; and
\item $-K_{\mathcal X}$ is relatively $\pi$-big and $\pi$-nef.
\end{itemize}
\end{defi}

As a corollary of Theorem~\ref{theo:standardmodel}, we have the following:

\begin{coro}
\label{coro:facotorialmodel}
Assume that $k$ is an algebraically closed field. When the characteristic of $k$ is positive, assume that $\mathrm{char}(k) = p > 5$. Let $X$ be a quartic del Pezzo surface defined over $K$.
Then a factorial model of $X$ exists.
\end{coro}

\begin{proof}
By Theorem~\ref{theo:standardmodel}, a standard model $\pi : \mathcal X \to B$ exists and we fix one such model.
Let $\rho : \widetilde{\mathcal X} \to \mathcal X$ be a $\mathbb Q$-factorization, i.e., $\widetilde{\mathcal X}$ is $\mathbb Q$-factorial and $\rho$ is projective and isomorphic in codimension $1$. (This exists in characteristic $0$ by \cite[Corollary 4.5]{Kawamata88}. See also \cite[Theorem 6.25]{KM98}. In positive characteristic, this is a theorem of Birkar (\cite[Theorem 1.6]{BirkarASENS}) assuming that $\mathrm{char}(k) = p > 5$.)
Let $\widetilde{\pi} = \pi \circ \rho$. Since $\rho$ is crepant, we have
\begin{itemize}
\item $\widetilde{\mathcal X}$ has only Gorenstein $\mathbb Q$-factorial terminal singularities;
\item every fiber of $\widetilde{\pi}$ is integral; and
\item $-K_{\widetilde{\mathcal X}}$ is relatively $\pi$-big and $\pi$-nef.
\end{itemize}
Thus it suffices to show that $\widetilde{\mathcal X}$ has only factorial singularities.
This follows from Lemma~\ref{lemm:factorialsing}.
\end{proof}

\subsection{Auxiliary results on quartic del Pezzo surfaces}

Let $k$ be a field of arbitrary characteristic and $X$ be a quartic weak del Pezzo surface, with at worst canonical singularities geometrically, defined over $k$. In this section, we collect some auxiliary results on $X$. First we record the following lemma:

\begin{lemm}
\label{lemm:conic_nef}
Let $X$ be a quartic del Pezzo surface, with at worst canonical singularities geometrically, defined over $k$. Let $D$ be a Cartier divisor on $X$ such that $D^2 = 2$ and $-K_X.D = 4$. Then $D$ is a big and nef divisor.
\end{lemm}

\begin{proof}
We may assume that our ground field is algebraically closed.
Let $\rho : \widetilde{X} \to X$ be the minimal resolution.
If $D$ is nef, then $D$ is big because $D^2 >0$. So it suffices to show that $D$ is nef. Suppose not. 
First note that by the Riemann--Roch theorem, we may assume that $D$ is effective. Then there exists a $(-1)$-curve $E$ on $\widetilde{X}$ such that $\mathsf c = \rho^*D.E <0$.
Applying the Hodge index theorem to the intersection matrix involving $-K_{\widetilde{X}}, \rho^*D, E$, we conclude that
\[
-4\mathsf c^2 + 8\mathsf c + 6 \geq 0.
\]
However, this is impossible when $\mathsf c \leq -1$. Thus our assertion follows.
\end{proof}

\begin{coro}
\label{coro:floppingcurve}
Let $X$ be a quartic weak del Pezzo surface, with at worst canonical singularities geometrically, defined over $k$. Let $D$ be a Cartier divisor on $X$ such that $D^2 = 2$ and $-K_X.D = 4$. 
Suppose that there exists an integral curve $C$ such that $D.C < 0$. 
Then we have $-K_X.C = 0$.

\end{coro}

\begin{proof}
We may assume that our ground field is algebraically closed.
Let $\rho : \widetilde{X} \to X$ be the minimal resolution.
The proof of Lemma~\ref{lemm:conic_nef} shows that the strict transform $\widetilde{C}$ of $C$ must be a $(-2)$-curve. Thus our assertion follows.
\end{proof}

Next we record the following lemma:

\begin{lemm}
\label{lemm:birationaltoquadric}
Assume that the characteristic of $k$ is not equal to $2$.
Let $X$ be a quartic weak del Pezzo surface with at worst canonical singularities geometrically. Let $D$ be a nef Cartier divisor on $X$ such that $D^2 = 2$ and $-K_X.D = 4$. Then the complete linear system $|D|$ defines a birational morphism $\phi : X \to Q$ to a quadric surface in $\mathbb P^3$.
\end{lemm}

\begin{proof}
We may assume that our ground field is algebraically closed.
Let $\rho : \widetilde{X} \to X$ be the minimal resolution. Our assumptions imply that 
\[
-K_{\widetilde{X}} \sim - \rho^*K_X
\]
is a big and nef divisor on $\widetilde{X}$. Then $\widetilde{D} = \rho^*D$ is a nef Cartier divisor such that $\widetilde{D}^2 = 2$ and $-K_{\widetilde{X}}.\widetilde{D} = 4$. It follows from \cite[Proposition 5.2.2.4]{ADHL} that $|\widetilde{D}|$ is base-point free. Note that \cite[Proposition 5.2.2.4]{ADHL} is stated when the ground field has characteristic $0$, but the proof of this proposition works in every characteristic different from $2$. A key to this argument is the Kawamata--Viehweg vanishing theorem which is valid in our case. Indeed, under our assumptions, \cite[Theorem A]{KT25} shows that $\widetilde{X}_{\overline{k}}$ is globally $F$-regular. Thus the Kawamata--Viehweg vanishing theorem holds by \cite[Theorem 6.8]{SS10}.
Finally, the Riemann--Roch formula and the Kawamata--Viehweg vanishing theorem imply that $|\widetilde{D}|$ is $3$-dimensional. Thus our assertion follows from the fact that $\widetilde{D}^2 = 2$. 
\end{proof}

Let $Q \subset \mathbb P^3$ be a non-split smooth quadric surface defined over $k$, i.e., the Picard rank of $Q$ is $1$. Let $Z \subset Q$ be a closed point of degree $4$.
Let $\phi : X \to Q$ be the blow-up of $Q$ along $Z$. Let $H$ be the pullback of the hyperplane class on $Q$ and $E$ be the exceptional divisor of $\phi$.
Then $\Pic(X)$ is generated by $H$ and $E$. Moreover we have:
\begin{lemm}
\label{lemm:anotherbirational}
Assume that the characteristic of $k$ is not equal to $2$.
Assume that $X$ is a smooth quartic del Pezzo surface.
Then $H' = 3H - 2E$ is the pullback of the hyperplane class via a birational morphism $\phi' : X \to Q'$ to another smooth quadric surface $Q'$, of Picard rank $1$, defined over $k$.
\end{lemm}

\begin{proof}
We have $-K_X.H' = 4$ and $H'^2 = 2$. By Lemma~\ref{lemm:conic_nef}, $H'$ is big and nef.
Hence it follows from Lemma~\ref{lemm:birationaltoquadric} that $H'$ is the pullback of the hyperplane class via a birational morphism $\phi' : X \to Q'$ to a quadric surface.
The exceptional divisor of $\phi'$ is linearly equivalent to $4H-3E$, and it is geometrically the disjoint union of four $(-1)$-curves because $-K_X$ is ample. Hence $Q'$ is smooth. Finally the Picard rank of $Q'$ is $1$ because the Picard rank of $X$ is $2$. Thus our assertion follows.
\end{proof}

\subsection{Wonderful models}
\label{subsec:wonderful}

Let $B$ be a geometrically integral smooth projective curve defined over a perfect field $k$,
and let $K = k(B)$ be its function field.
The following models are quite useful in practice:

\begin{defi}
Let $X$ be a quartic del Pezzo surface defined over $K$. An integral model $\pi : \mathcal X \to B$ of $X$ is called a {\it wonderful model} if the following properties hold:
\begin{itemize}
\item $\mathcal X$ has only geometrically factorial terminal singularities;
\item every fiber of $\pi_{\overline{k}}$ is integral and has only canonical singularities; and
\item $-K_{\mathcal X}$ is relatively $\pi$-big and $\pi$-nef.
\end{itemize}
\end{defi}

\noindent
{\bf Setup:}
Throughout this section, we consider the following quartic del Pezzo surfaces:
let $Q$ be a non-split smooth quadric surface defined over $K$ such that its base change $Q_{\overline{k}(B)}$ has Picard rank $1$.
Let $Z \subset Q$ be a closed point of degree $4$ such that its base change $Z_{\overline{k}(B)}$ is integral. Let $\phi : X \to Q$ be the blow-up of $Q$ along $Z$. Throughout this section, we assume that $X$ is a smooth quartic del Pezzo surface defined over $K$. 

The following proposition shows that the birational morphism $\phi$ extends to a wonderful model after applying a sequence of flops:

\begin{prop}
\label{prop:Dflops}
We assume that $k$ is an algebraically closed field. When the characteristic of $k$ is positive, we also assume that the characteristic of $k$ is $> 5$. Let $\pi : \mathcal X \to B$ be a wonderful model of $X$. Let $D$ be the flat closure of a hyperplane section from $Q$. Then, after applying a sequence of $D$-flops over $B$,
\[
\xymatrix{
 \mathcal X \ar@{-->}[rr]^\psi \ar[rd]_\pi & & \widetilde{\mathcal X} \ar[dl]^{\widetilde{\pi}} \\
 & B &
}
\]
the strict transform $\widetilde{D}$ of $D$ becomes relatively $\widetilde{\pi}$-big and $\widetilde{\pi}$-nef and $\widetilde{\pi} : \widetilde{\mathcal X} \to B$ is a wonderful model of $X$. Moreover, when $\mathcal X$ is smooth, $\widetilde{\mathcal X}$ is also smooth.
\end{prop}

\begin{proof}
Let $H$ be $-K_{\mathcal X}$, which is relatively $\pi$-big and $\pi$-nef.
Let $F$ be the class of a general fiber of $\pi$. 
By the relative version of Wilson's theorem, there exists an effective divisor $N$ such that for a sufficiently small rational $t >0$, $-K_{\mathcal X} - tN$ is relatively $\pi$-ample and $(\mathcal X, tN)$ is a terminal pair. 
Let $m > 0 $ be a sufficiently large integer such that $-K_{\mathcal X} - tN + mF$ is ample.
By arguing as in Lemma~\ref{lemm:bignefterminal}, we can find an effective $\mathbb Q$-divisor $H' \sim H+mF$ such that $(\mathcal X, H')$ is a terminal pair. Then for a sufficiently small rational $\epsilon > 0$, $(\mathcal X, \epsilon D + H')$ is also a terminal pair.
Thus by \cite[Corollary 1.4.2]{BCHM} and \cite[Theorem 1.1]{HNT}, we can run a relative $(\epsilon D + H' + K_{\mathcal X})$-MMP over $B$. Note that since we have $\epsilon D + H' + K_{\mathcal X} \sim_{\mathbb Q} \epsilon D + mF$, this MMP is a relative $D$-MMP over $B$. As stated above, let $\psi : \mathcal X \dashrightarrow \widetilde{\mathcal X}$ be the outcome of this MMP. Then the strict transform $\widetilde{D}$ of $D$ is relatively $\widetilde{\pi}$-big and $\widetilde{\pi}$-nef.

It follows from Lemma~\ref{lemm:conic_nef} that $D|_{\mathcal X_b}$ is nef whenever $\mathcal X_b$ is a (possibly singular) del Pezzo surface. This shows that this MMP does not modify the generic fiber. In particular, $\widetilde{\pi} : \widetilde{\mathcal X} \to B$ is an integral model of $X$.
Moreover, when $\mathcal X_b$ is a weak del Pezzo surface and $D|_{\mathcal X_b}$ is not nef, Corollary~\ref{coro:floppingcurve} shows that a negative curve in $\mathcal X_b$ is a $(-2)$-curve on the minimal resolution of $\mathcal X_b$. We inductively conclude that each step of this MMP is a $D$-flop. This implies that $\widetilde{\mathcal X}$ has only Gorenstein terminal singularities and $-K_{\widetilde{\mathcal X}}$ is relatively $\widetilde{\pi}$-big and $\widetilde{\pi}$-nef. Since $\widetilde{\mathcal X}$ is $\mathbb Q$-factorial, Lemma~\ref{lemm:factorialsing} shows that $\widetilde{\pi} : \widetilde{\mathcal X} \to B$ is a factorial model of $X$.

Thus the only thing we need to show is that every fiber of $\widetilde{\pi}$ is normal and has only canonical singularities when $\psi$ is a sequence of $D$-flops. First let us prove that every fiber is normal. We assume that the characteristic of $k$ is $0$.
Then by the inversion of adjunction (\cite[Theorem 5.50]{KM98}), $(\mathcal X, \mathcal X_b)$ is plt. Since we are applying $D$-flops over $B$, we can conclude that $(\widetilde{\mathcal X}, \widetilde{\mathcal X}_b)$ is also plt. Hence by \cite[Corollary 5.52]{KM98}, $ \widetilde{\mathcal X}_b$ is normal.

In positive characteristic, as we explained before, $\mathcal X_b$ is normal and globally $F$-regular, so by \cite[Theorem A]{Das15} $(\mathcal X, \mathcal X_b)$ is purely $F$-regular (or divisorial $F$-regular in other terminology) in a neighborhood of $\mathcal X_b$. Thus by \cite[Theorem 3.3]{HW02}, $(\mathcal X, \mathcal X_b)$ is plt. Then $(\widetilde{\mathcal X}, \widetilde{\mathcal X}_b)$ is also plt.
Let
\[
\xymatrix{
& Y\ar[ld]_\rho \ar[rd]^{\widetilde{\rho}}& \\
\mathcal X_b \ar@{-->}[rr]_\psi& & \widetilde{\mathcal X}_b
}
\]
be a common resolution. Then since $\psi$ is a sequence of $D$-flops, we have
\[
\rho^*K_{\mathcal X_b} \sim \widetilde{\rho}^*K_{\widetilde{\mathcal X}_b}.
\]
Let $\widetilde{\mathcal X}_b^n$ be the normalization of $\widetilde{\mathcal X}_b$
and let $(K_{\widetilde{\mathcal X}} + \widetilde{\mathcal X}_b)|_{\widetilde{\mathcal X}_b^n} = K_{\widetilde{\mathcal X}_b^n} + B^n$. Then we have
\[
\rho^*K_{\mathcal X_b} \sim \widetilde{\rho}^*K_{\widetilde{\mathcal X}_b} \sim \widetilde{\rho}_n^*(K_{\widetilde{\mathcal X}_b^n} + B^n),
\]
where $\widetilde{\rho}_n : Y \to \widetilde{\mathcal X}_b^n$ is the natural morphism.
In particular, $(\widetilde{\mathcal X}_b^n, B^n)$ is a weak del Pezzo pair with only canonical singularities. The computation in \cite[Theorem 3.36]{Kollar13} shows that $B^n$ has only standard coefficients. Thus by \cite[Theorem 3.1]{HX15}, $(\widetilde{\mathcal X}_b^n, B^n)$ is strongly $F$-regular. Therefore it follows from \cite[Proposition 4.1]{HX15} that $\widetilde{\mathcal X}_b$ is normal.
Now, since $\mathcal X_b$ has only canonical singularities, the same holds for $\widetilde{\mathcal X}_b$. 

Finally the last statement in characteristic $0$ follows from \cite[Theorem 2.4]{Kollar89}. In positive characteristic, this is established in \cite[Theorem 1.1]{Tanaka25}.
\end{proof}

\begin{coro}
\label{coro:Dflopsoverk}
Let $k$ be a perfect field. We assume that when $k$ has positive characteristic, it is $>5$.
Let $\pi : \mathcal X \to B$ be a wonderful model of $X$.
Let $D$ be the flat closure of a hyperplane section from $Q$. Let 
\[
\xymatrix{
 \mathcal X_{\overline{k}} \ar@{-->}[rr]^{\psi_{\overline{k}}} \ar[rd]_{\pi_{\overline{k}}} & & \widetilde{\mathcal X}_{\overline{k}} \ar[dl]^{\widetilde{\pi}_{\overline{k}}} \\
 & B &
}
\]
be a sequence of $D_{\overline{k}}$-flops obtained in Proposition~\ref{prop:Dflops}. Then this diagram descends to $k$.
\end{coro}
\begin{proof}
Recall that $\psi_{\overline{k}}$ is the outcome of a relative $D_{\overline{k}}$-MMP over $B_{\overline{k}}$.
Let $A$ be a general ample $\mathbb Q$-divisor on $\mathcal X$. 
Then for a sufficiently small rational $\epsilon > 0$, this MMP is still a relative $D_{\overline{k}} + \epsilon A_{\overline{k}}$-MMP over $B_{\overline{k}}$. Moreover since the base change $X_{\overline{k}(B)}$ has Picard rank $2$, all fibers of $\pi_{\overline{k}}$ are integral, and 
\[
0 \to N^1(B_{\overline{k}}) \to N^1(\mathcal X_{\overline{k}}) \to N^1(\mathcal X_{\overline{k}}/B_{\overline{k}}) \to 0,
\]
is exact,
we conclude that $N^1(\mathcal X_{\overline{k}}/B_{\overline{k}})$ has rank $2$. This implies that the strict transform $\widetilde{D}_{\overline{k}} + \epsilon \widetilde{A}_{\overline{k}}$ is $\widetilde{\pi}_{\overline{k}}$-ample for a sufficiently small rational $\epsilon > 0$.
Thus the $\mathcal O_{B_{\overline{k}}}$-algebra
\[
\oplus_m \pi_*\mathcal O(\lfloor m(D_{\overline{k}} + \epsilon A_{\overline{k}})\rfloor),
\]
is finitely generated, and by finite Galois descent, the $\mathcal O_B$-algebra
\[
\oplus_m \pi_*\mathcal O(\lfloor m(D + \epsilon A)\rfloor),
\]
is finitely generated too. Hence the proj construction
\[
\widetilde{\pi} : \widetilde{\mathcal X} = \mathrm{Proj}_B(\oplus_m \pi_*\mathcal O(\lfloor m(D + \epsilon A)\rfloor)) \to B ,
\]
descends $\widetilde{\pi}_{\overline{k}}$ to $k$. Thus our assertion follows.
\end{proof}

Finally, when $D$ is relatively big and nef, it defines a morphism to a quadric surface bundle:

\begin{prop}
\label{prop:blowupquadric}
Let $k$ be a perfect field of characteristic $\neq 2$.
Let $\pi : \mathcal X \to B$ be a wonderful model of $X$.
Let $D$ be the flat closure of a hyperplane section from $Q$.
Assume that $D$ is relatively $\pi$-big and $\pi$-nef.
Then the natural morphism
\[
\pi^*\pi_*\mathcal O(D) \to \mathcal O(D),
\]
is surjective, and $\pi_*\mathcal O(D)$ is a locally free sheaf of rank $4$.
In particular, it defines a $B$-morphism
\[
\phi : \mathcal X \to \mathbb P_B(\pi_*\mathcal O(D)),
\]
whose image is a quadric surface bundle $\pi_{\mathcal Q} : \mathcal Q \to B$ such that every geometric fiber is integral.

Moreover, $\phi : \mathcal X \to \mathcal Q$ is a geometrically $K_{\mathcal X}$-negative extremal birational contraction whose exceptional locus is a geometrically integral surface $E$ mapping to a geometrically integral curve $\mathcal Z \subset \mathcal Q$ which is the closure of $Z \subset Q$. Furthermore, $\mathcal Q$ has only geometrically factorial terminal singularities.
\end{prop}

\begin{proof}
Under our assumption, any geometric fiber $\mathcal X_{\overline{b}}$ is a weak del Pezzo surface with only canonical singularities.
It follows from Lemma~\ref{lemm:birationaltoquadric} that $|D|_{\mathcal X_{\overline{b}}}|$ is base-point free and it defines a birational morphism to an integral quadric surface $Q_{\overline{b}}$ in $\mathbb P^3$.
So, by Grauert's theorem, $\pi_*\mathcal O(D)$ is a locally free sheaf of rank $4$, and the natural morphism
\[
\pi^*\pi_*\mathcal O(D) \to \mathcal O(D),
\]
is surjective. Thus
 it defines a $B$-morphism
\[
\phi : \mathcal X \to \mathbb P_B(\pi_*\mathcal O(D)),
\]
whose image is a quadric surface bundle $\pi_{\mathcal Q} : \mathcal Q \to B$.

Since $N^1(\mathcal X_{\overline{k}}/B_{\overline{k}})$ has rank $2$, $\phi : \mathcal X \to \mathcal Q$ is a geometrically $K_{\mathcal X}$-negative extremal birational contraction. In particular, the exceptional locus of $\phi$ is a geometrically integral divisor $E$ contracting to a geometrically integral curve $\mathcal Z$ which is the flat closure of $Z$.
Finally $\mathcal Q$ has geometrically $\mathbb Q$-factorial terminal singularities. Since we have
\[
K_{\mathcal X} \sim \phi^*K_{\mathcal Q} + E,
\]
it follows that $K_{\mathcal Q}$ is Cartier. Thus Lemma~\ref{lemm:factorialsing} shows that $\mathcal Q$ has only geometrically factorial terminal singularities. Thus our assertion follows.
\end{proof}

\begin{coro}
\label{coro:blowupofquadric}
In the setting of Proposition~\ref{prop:blowupquadric}, let us assume that $\mathcal X$ is smooth. 
Then $\mathcal Q$ is smooth and $\phi : \mathcal X \to \mathcal Q$ is the blow-up of $\mathcal Q$ along a geometrically integral smooth curve $\mathcal Z$.
\end{coro}
\begin{proof}
In characteristic $0$, this follows from a famous classification by Mori on extremal contractions on smooth threefolds (\cite[Theorem 3.3]{Mori82}). In arbitrary characteristic, such a classification was obtained by Koll\'ar (\cite[(1,1) Main theorem]{Kollar91}).
\end{proof}

\subsection{Cone theorem}

Let us work in the setup established in Section~\ref{subsec:wonderful}.
In particular, $k$ is a perfect field and $X$ is a blow-up of a non-split quadric surface over $K(B)$.
Let $\pi : \mathcal X \to B$ be a factorial model of $X$. 
We consider $N_1(\mathcal X)$ and we view $N_1(\mathcal X_\eta)$ as a subspace of $N_1(\mathcal X)$. Let $V = N_1(\mathcal X_\eta)$, but the lattice structure $V_{\mathbb Z}$ is given by $N_1(\mathcal X_\eta) \cap N_1(\mathcal X)_{\mathbb Z}$. Note that this contains $N_1(\mathcal X_\eta)_{\mathbb Z}$, but might be different.

The following proposition is a version of \cite[Theorem 5.7]{LRT25} in our setting:
\begin{prop}
\label{prop:conetheorem}
Assume that the characteristic of $k$ is not equal to $2$.
Then there exist finitely many $\alpha_1, \cdots, \alpha_t \in \Nef_1(\mathcal X)_{\mathrm{sec}, \mathbb Z}$ such that we have
\[
\Nef_1(\mathcal X)_{\mathrm{sec}, \mathbb Z} = \bigcup_{i = 1}^t (\alpha_i + \Nef_1(\mathcal X_\eta)\cap V_{\mathbb Z}).
\]
\end{prop}
\begin{proof}
Let $A$ be the affine space
\[
\{ \alpha \in N_1(\mathcal X) \, | \, \mathcal X_b.\alpha = 1 \}.
\]
Then one can view this as a translate of $N_1(\mathcal X_\eta)$.

Under our setup, we have a birational morphism $\phi_1 : X \to Q_1$ and we denote the closure of the exceptional divisor by $E_1 \subset \mathcal X$. By Lemma~\ref{lemm:anotherbirational}, we have another birational morphism $\phi_2 : X \to Q_2$ and we denote the closure of the exceptional divisor by $E_2 \subset \mathcal X$. Let $\alpha_0 \in A$ be a unique rational point defined by 
\[
E_i.\alpha_0 = 0 \, (i = 1, 2).
\]
Then we have
\[
\Nef_1(\mathcal X)_{\mathrm{sec}, \mathbb Z}  \subset (\alpha_0 + \Nef_1(\mathcal X_\eta)).
\]
Thus our assertion follows from \cite[Theorem 1.1(a)]{HW07}.
\end{proof}

\section{Blow-up models}

\subsection{Quadric surface bundles}
\label{subsec:quadric}

Let $k$ be a perfect field of characteristic not equal to $2$. In this section, we recall the geometry of quadric surface bundles over curves, which was developed in \cite[Sections 2 and 3]{HT12}.

\subsubsection*{Quadric surface bundles}

Let $B$ be a smooth, geometrically integral, projective curve defined over $k$.
Let $\mathcal E$ be a rank $4$ locally free sheaf over $B$.
We consider the associated projective bundle $\pi_{\mathcal E} : \mathbb P(\mathcal E) \to B$. In this paper, we employ the Grothendieck notation for projective bundles so that the fiber of $\pi_{\mathcal E}$ at $\overline{b} : \Spec (\overline{k}) \to B$ parametrizes codimension-one subspaces in $\mathcal E\otimes k(\overline{b})$.

A {\it quadric surface bundle over $B$} is a geometrically integral subvariety $\mathcal Q \subset \mathbb P(\mathcal E)$ with the natural flat projection $\pi_{\mathcal Q} : \mathcal Q \to B$ such that every geometric fiber of $\pi_{\mathcal Q}$ is a quadric surface in the corresponding geometric fiber of $\pi_{\mathcal E}$.
In this paper, we only consider bundles for which (i) every geometric fiber of $\pi_{\mathcal Q}$ is integral so that a singular geometric fiber has at most one $\mathbf A_1$-singularity and (ii) the relative Picard rank of $\pi_{\mathcal Q_{\overline{k}}} : \mathcal Q_{\overline{k}} \to B_{\overline{k}}$ is $1$. Such a quadric surface bundle is defined by a section
\[
q \in \rH^0(B, \mathrm{Sym}^2\mathcal E \otimes \mathcal I),
\]
where $\mathcal I$ is an invertible sheaf on $B$.
We normalize $\mathcal E$ so that we have
\[
\epsilon(\pi_{\mathcal Q}) := \deg (\mathcal I) = 0 \text{ or } 1.
\]
We may interpret $q$ as a homomorphism
\[
\mathcal E^\vee \to \mathcal E \otimes \mathcal I.
\]
Then the discriminant $\mathfrak d$ is an effective divisor on $B$ which records points where $q$ drops the rank, and its degree is given by
\[
\Delta := \deg(\mathfrak d) = 2 \deg (\mathcal E) + 4\deg (\mathcal I).
\]
So in particular, we have
\[
\Delta = 
\begin{cases}
2 \deg (\mathcal E) & \text{ if $\epsilon(\pi_{\mathcal Q}) = 0$,}\\
2 \deg (\mathcal E) + 4 & \text{ if $\epsilon(\pi_{\mathcal Q}) = 1$.}
\end{cases}
\]
A quadric surface bundle $\pi_{\mathcal Q} : \mathcal Q \to B$ is {\it square free} if $\mathfrak d$ is reduced. This is equivalent to $\mathcal Q$ being smooth. We will assume this condition from now on.

\subsubsection*{The relative Fano variety of lines}
Let $\pi_{\mathcal Q} : \mathcal Q \to B$ be a quadric surface bundle over $B$ such that the total space $\mathcal Q$ is smooth, every geometric fiber of $\pi_{\mathcal Q}$ is integral, and the geometric relative Picard rank is $1$.
We also assume that the generic fiber of $\pi_{\mathcal Q}$ is smooth so that a general fiber is smooth too.

Let $F_1(\mathcal Q)\to B$ be the relative Fano variety of lines, parametrizing lines in fibers of $\pi_{\mathcal Q}$. Let 
\[
F_1(\mathcal Q) \to C \to B,
\]
be the Stein factorization. Then $C \to B$ is a degree $2$ finite morphism ramified along $\mathfrak d$.
In particular, $C$ is a geometrically integral smooth projective curve of genus 
\[
g(C) = 2g(B)-1 + \frac{\Delta}{2}.
\]
Then the projection $\pi_{F_1(\mathcal Q)} : F_1(\mathcal Q) \to C$ is a $\mathbb P^1$-bundle.
When $k$ is an algebraically closed field, the Brauer group of $K$ is trivial. Thus this $\mathbb P^1$-bundle admits a section. When $k$ is a finite field, this bundle also admits a section due to Lang's theorem and the Hochschild--Serre spectral sequence.
So from now on we assume that there is a rank $2$ bundle $\mathcal F$ on $C$ such that we have a $C$-isomorphism:
\[
F_1(\mathcal Q) \cong \mathbb P_C(\mathcal F).
\]

Conversely, suppose that we have a degree $2$ finite morphism $g : C \to B$ from a smooth geometrically integral projective curve to $B$. Since we assume that the characteristic of $k$ is different from $2$, $g$ is separable, so it is tamely ramified along a reduced divisor $\mathfrak d \subset B$. Let $\mathcal F$ be a rank $2$ locally free sheaf over $C$, and consider the associated projective bundle
\[
\pi_{\mathbb P_C(\mathcal F)} : \mathbb P_C(\mathcal F) \to C.
\]
Then, we apply the Weil restriction of scalars:
\[
\pi_{\mathcal Q'} : \mathcal Q' := \mathrm{Res}_{C/B}(\mathbb P_C(\mathcal F)) \to B.
\]
For $b \in B(\overline{k}) \setminus \mathfrak d$, let $g^{-1}(b) = \{ c_1, c_2\}$. Then the fiber of $\pi_{\mathcal Q'}$ at $b$ is isomorphic to
\[
\mathbb P(\mathcal F\otimes k(c_1)) \times \mathbb P(\mathcal F\otimes k(c_2)) \cong \mathbb P^1 \times \mathbb P^1.
\]
When $b \in \mathfrak d(\overline{k})$ with the preimage $c \in C(\overline{k})$, the fiber of $\pi_{\mathcal Q'}$ is set-theoretically
\[
\mathrm{Sym}^2(\mathbb P(\mathcal F\otimes k(c))) \cong \mathbb P^2,
\]
but its multiplicity is $2$. The threefold $\mathcal Q'$ is singular along the diagonal of the reduced scheme of each multiple fiber.

If we blow up the diagonal of $\mathrm{Sym}^2(\mathbb P(\mathcal F\otimes k(c)))$ for each multiple fiber and contract all strict transforms of the reduced schemes of multiple fibers, then we obtain a quadric surface bundle over $B$
\[
\pi_{\mathcal Q} : \mathcal Q \to B,
\]
which is smooth, every geometric fiber is integral, and the geometric relative Picard rank is $1$. Note that this is defined over $k$ because we perform this operation for all singular fibers simultaneously.  This birational model fits into the diagram
\[
\xymatrix{
& \widetilde{\mathcal Q} \ar[dl] \ar[dr]& \\
\mathcal Q' \ar[dr]_{\pi_{\mathcal Q'} }& & \mathcal Q \ar[dl]^{\pi_{\mathcal Q}} \\
& B &
}
\]
where $\widetilde{\mathcal Q}$ is the blow-up of $\mathcal Q'$ mentioned above.
Moreover, one can prove that $\widetilde{\mathcal Q} \to \mathcal Q'$ is a crepant resolution.
In particular, $\mathcal Q'$ is Gorenstein and has only canonical singularities.

Assuming $k$ is an algebraically closed field or a finite field, this establishes the following one-to-one correspondence:
\[
\{ \pi_{\mathcal Q} : \mathcal Q \to B\} \leftrightarrow \{ g : C \to B, \, \pi_{\mathbb P_C(\mathcal F)} : \mathbb P_C(\mathcal F) \to C\}.
\]

\subsubsection*{Correspondence of sections}

Let $\pi_{\mathcal Q} : \mathcal Q \to B$ be a quadric surface bundle over $B$ such that the total space $\mathcal Q$ is smooth, the generic fiber is also smooth, every geometric fiber of $\pi_{\mathcal Q}$ is integral, and the geometric relative Picard rank is $1$ as in the previous section. Let
\[
F_1(\mathcal Q) \to C \to B,
\]
be the corresponding relative Fano variety of lines and its Stein factorization $g : C \to B$.
We assume that the $\mathbb P^1$-bundle $\pi_{F_1(\mathcal Q)} : F_1(\mathcal Q) \to C$ is the projectivization of a rank $2$ locally free sheaf $\mathcal F$ over $C$.
We denote $\mathcal O_{F_1(\mathcal Q)/C}(1)$ by $\xi$.

Let $\sigma : B \to \mathcal Q$ be a section of $\pi_{\mathcal Q}$. Its height is defined by
\[
h_{\mathcal Q} (\sigma) := \deg (-\sigma^*K_{\mathcal Q/B}),
\]
where $-K_{\mathcal Q/B}$ is the relative anticanonical divisor.
Our object of interest is the space of sections of height $h_{\mathcal Q}$ for $\pi_{\mathcal Q}$
\[
\mathrm{Sec}(\mathcal Q/B, h_{\mathcal Q}),
\]
which is a quasi-projective scheme over $k$.

On the other hand, we also consider a section $\tau : C \to F_1(\mathcal Q)$ of $\pi_{F_1(\mathcal Q)}$. We define its degree $d_{F_1(\mathcal Q)}$ by
\[
d_{F_1(\mathcal Q)}(\tau) = \deg(\tau^*\xi).
\]
We have the space of sections of degree $d_{F_1(\mathcal Q)}$ for $\pi_{F_1(\mathcal Q)}$:
\[
\mathrm{Sec}(F_1(\mathcal Q)/C, d_{F_1(\mathcal Q)}).
\]
We explain the relation between these two spaces following \cite[Section 3]{HT12} and \cite[Section 6]{TT25}.
Let 
\[
\mathcal U \subset \mathcal Q \times_B F_1(\mathcal Q),
\]
be the universal family of lines. Let $\sigma : B \to \mathcal Q$ be a section of height $h_{\mathcal Q}$.
We consider the base change:
\[
\xymatrix{
\mathcal U_\sigma \ar@{^{(}->}[r] \ar[d] & B \times_B F_1(\mathcal Q) \ar[d]^{\sigma \times \mathrm{id}}  \\
\mathcal U \ar@{^{(}->}[r]& \mathcal Q \times_B F_1(\mathcal Q)
}
\]
Then the map 
\[
\mathcal U_\sigma \to F_1(\mathcal Q) \to B,
\]
is a degree $2$ separable map, so we conclude that $\mathcal U_\sigma \to F_1(\mathcal Q) \to C$ is an isomorphism. Thus we obtain a section $\tau : C \to F_1(\mathcal Q)$.
Moreover this construction makes sense for a family of sections $\sigma : B \times T \to \mathcal Q$, so we obtain a morphism
\[
\mathrm{Sec}(\mathcal Q/B, h_{\mathcal Q}) \to \mathrm{Sec}(F_1(\mathcal Q)/C, d_{F_1(\mathcal Q)}),
\]
where the relation between $h_{\mathcal Q}$ and $d_{F_1(\mathcal Q)}$ is given by
\[
h_{\mathcal Q} = 2d_{F_1(\mathcal Q)} - \deg (\mathcal F) - \frac{\Delta}{2}.
\]
See \cite[After Proposition 2]{HT12} for an explanation of this formula. Note that in \cite{HT12}, the authors adopted a classical notation for projectivizations of locally free sheaves which explains the sign difference.

Conversely, suppose that we have a section $\tau : C \to F_1(\mathcal Q)$.
By the Weil restriction of scalars, this corresponds to a section $\sigma' : B \to \mathcal Q'$.
By taking the strict transforms, we obtain a section $\sigma : B \to \mathcal Q$. Note that since $\widetilde{\mathcal Q} \to \mathcal Q'$ is a crepant resolution, we have
\[
h_{\mathcal Q'}(\sigma') = \deg(-\sigma'^*K_{\mathcal Q'/B}) = \deg(-\widetilde{\sigma}^*K_{\widetilde{Q}/B}),
\]
where $\widetilde{\sigma} : B \to \widetilde{\mathcal Q}$ is the strict transform of $\sigma'$. Then since we have
\[
K_{\widetilde{\mathcal Q}} = \psi^*K_{\mathcal Q} + 2E,
\]
where $E$ is the exceptional divisor for a birational morphism $\psi : \widetilde{\mathcal Q} \to \mathcal Q$. Note that $E$ has multiplicity $2$ in the fiber of $\pi_{\widetilde{\mathcal Q}} : \widetilde{\mathcal Q} \to B$, and any section $\widetilde{\sigma} : B \to \widetilde{\mathcal Q}$ does not intersect with $E$. Hence we conclude that
\[
h_{\mathcal Q'}(\sigma') = h_{\mathcal Q}(\sigma).
\]
Thus this construction works in any family of sections $\tau : C \times T \to F_1(\mathcal Q)$ and we obtain the inverse map 
\[
\mathrm{Sec}(F_1(\mathcal Q)/C, d_{F_1(\mathcal Q)}) \to \mathrm{Sec}(\mathcal Q/B, h_{\mathcal Q}).
\]
The upshot is the following proposition:
\begin{prop}[{\cite{HT12}}]
We have an isomorphism
\[
\mathrm{Sec}(\mathcal Q/B, h_{\mathcal Q}) \cong \mathrm{Sec}(F_1(\mathcal Q)/C, d_{F_1(\mathcal Q)}),
\]
with 
\[
h_{\mathcal Q} = 2d_{F_1(\mathcal Q)} - \deg (\mathcal F) - \frac{\Delta}{2}.
\]
\end{prop}

\subsection{Blow-ups of a quadric surface}

Assume that $k$ is an algebraically closed field. Let $Q$ be a possibly singular integral quadric surface defined over $k$. Let $\phi : X \to Q$ be the blow-up of $Q$ along a length $4$ subscheme $Z$ on $Q$. We study when $X$ is a weak del Pezzo surface with mild singularities.

\subsubsection*{Case 1: $Q$ is smooth and $Z$ consists of distinct $4$ points}

In this situation, $X$ is always smooth and we have the following proposition:

\begin{prop}
\label{prop:whenweakdPI}
Suppose that $Q$ is smooth and $Z$ consists of distinct $4$ points. Then $X$ is a smooth weak del Pezzo surface if and only if there is no line on $Q$ containing three points of $Z$.
Moreover, it is a smooth del Pezzo surface if and only if there is no line on $Q$ containing two points of $Z$ and there is no hyperplane section containing $Z$.
\end{prop}

\begin{proof}
We prove the first statement. The second statement is similar.
When there is a line $\ell$ on $Q$ such that three points of $Z$ are contained in $\ell$, then its strict transform $\widetilde{\ell}$ of $\ell$ on $X$ satisfies 
\[
-K_X.\widetilde{\ell} \leq -1,
\]
showing that $-K_X$ is not nef.

Conversely, suppose that $-K_X$ is not nef. This means that there is an integral curve $C$ such that
\[
-K_X.C \leq -1.
\]
Let $H$ be the pullback of the hyperplane class on $Q$.
We define constants $\mathsf C_1, \mathsf C_2, \mathsf C_3$ by
\[
\mathsf C_1 = H.C, \quad \mathsf C_2 = -K_X.C, \quad \mathsf C_3 = C^2.
\]
The intersection matrix involving $-K_X, H, C$ is given by
\[
\begin{pmatrix}
4 & 4 & \mathsf C_2 \\
4 & 2 & \mathsf C_1 \\
\mathsf C_2 & \mathsf C_1 & \mathsf C_3
\end{pmatrix}.
\]
Thus, the Hodge index theorem tells us that
\[
 8 \mathsf C_1\mathsf C_2 - 4 \mathsf C_1^2 - 2\mathsf C_2^2 -8 \mathsf C_3 \geq 0.
\]
On the other hand, since $C$ is integral, the adjunction tells us that
\[
\mathsf C_3 - \mathsf C_2 \geq -2.
\]
Combining these results, we conclude that
\begin{align*}
4 \mathsf C_1^2 & \leq 16 + 8 (\mathsf C_1-1)\mathsf C_2,
\end{align*}
which implies that $\mathsf C_1 \leq 2$.
Thus $C$ is a line or a hyperplane section. However, the strict transform of a hyperplane section cannot be a negative curve, proving the claim.
\end{proof}

\subsubsection*{Case 2: $Q$ is singular and $Z$ consists of distinct $4$ points which are not the vertex}

In this situation, $Q$ is a quadric cone with one $\mathbf A_1$-singularity which is the vertex:
\begin{prop}
\label{prop:whenweakdPII}
Suppose that $Q$ is singular and $Z$ consists of distinct $4$ points which are not the vertex. 
Then $X$ is a weak del Pezzo surface with one $\mathbf A_1$-singularity if and only if there is no line on $Q$ containing three points of $Z$.
\end{prop}

\begin{proof}
One may apply the discussion of Proposition~\ref{prop:whenweakdPI} to the minimal resolution of $X$.
\end{proof}

\subsubsection*{Case 3: $Q$ is smooth and the support of $Z$ consists of distinct $3$ points}

In this situation, there is a length two subscheme $Z'$ of $Z$ which is supported at a point.
Its blow-up $X$ admits one $\mathbf A_1$-singularity on the exceptional divisor $E$ such that the pullback of $Z'$ is $2E$. Let us denote other exceptional divisors by $E_1, E_2$.
We have the following equalities:
\[
K_X \sim \phi^*K_Q + E_1 + E_2 + 2E, \quad E^2 = -1/2, \quad K_X.E = -1,
\]
where $\phi : X \to Q$ is the blow-up map.
We denote the minimal resolution by $\rho : \widetilde{X} \to X$ with the unique exceptional divisor $E_3$ and we denote the strict transform of $E$ by $\widetilde{E}$.
The numerical pullback of $E$ is $\widetilde{E} + \frac{1}{2}E_3$ and we have
\[
K_{\widetilde{X}} \sim \rho^*K_X, \quad \widetilde{E}^2 = -1, \quad K_{\widetilde{X}}.\widetilde{E} = -1.
\]
Using these we prove the following proposition:

\begin{prop}
\label{prop:whenweakdPIII}
Suppose that $Q$ is smooth and the support of $Z$ consists of distinct $3$ points.
Then $X$ is a weak del Pezzo surface with one $\mathbf A_1$-singularity if and only if there is no line on $Q$ containing a length $3$ subscheme of $Z$.
\end{prop}

\begin{proof}
First, let us assume that we have a line $\ell$ on $Q$ such that a length $3$ subscheme $Z'$ of $Z$ is contained in $\ell$, but not $Z$ itself. Suppose that $Z'$ is reduced. In this case, we have 
\[
\rho^*\phi^*\ell = \widetilde{\ell} + E_1 + E_2 + \widetilde{E} + E_3,
\]
where $\widetilde{\ell}$ is the strict transform of $\ell$. This implies that we have $ \widetilde{\ell}^2 = -3$.

Next, suppose that $Z'$ is non-reduced. In this situation, we have
\[
\rho^*\phi^*\ell = \widetilde{\ell} + E_i + 2\widetilde{E} + E_3,
\]
for some $i = 1, 2$, which implies that $\widetilde{\ell}^2 = -3$.

Now assume that $\ell$ contains $Z$. Then we have
\[
\rho^*\phi^*\ell = \widetilde{\ell} + E_1 + E_2 + 2\widetilde{E} + E_3.
\]
This implies that $\widetilde{\ell}^2 = -4$. Thus in all cases, $-K_X$ cannot be nef.

For the converse, we may argue as in Proposition~\ref{prop:whenweakdPI}.
\end{proof}

Finally, we prove the following proposition which explains why certain fibers are not allowed in smooth wonderful models:

\begin{prop}
\label{prop:nonvertex}
Let $Q$ be a singular quadric cone and let $$Z = \Spec\, (k[t]/(t^{N+1})),$$ with $N \geq 1$ which is embedded into $Q$ by $\hat{\sigma} : Z \to Q$. If $Z$ is supported at the vertex of $Q$, then the blow-up $$\phi : X = \mathrm{Bl}_Z(Q) \to Q$$ is not a weak del Pezzo surface with canonical singularities.
\end{prop}

\begin{proof}
Let $\rho : \widetilde{X} \to X$ be the minimal resolution. Our assertion follows if $\widetilde{X}$ is not a weak del Pezzo surface. 
Suppose that $Z$ is supported at the vertex. When $Z$ has length $2$ and the direction of $Z$ is not tangent to any line on $Q$, $X$ is not normal. When $Z$ has length $3$, the first order direction of $Z$ is tangent to a line on $Q$. So we may assume that $Z$ is tangent to a line.
Let $\psi : \mathbb F_2 \to Q$ be the blow-up of the vertex. Since the pullback of the vertex by $\phi \circ \rho : \widetilde{X} \to X \to Q$ is an honest Cartier divisor, $\phi\circ\rho$ has to factor through $\psi$. The pullback of $Z$ by $\psi$ comes with an embedded point, so to get $\widetilde{X}$, a further blow-up is needed, which means that we must blow up the $(-2)$-curve. So $\widetilde{X}$ cannot be a weak del Pezzo surface. Thus our assertion follows.
\end{proof}

\subsection{Blow-ups of quadric surface bundles}

Let $k$ be an algebraically closed field of characteristic not equal to $2$.
Let $B$ be a smooth projective curve defined over $k$.
Let $\pi_{\mathcal Q} : \mathcal Q \to B$ be a quadric surface bundle such that (i) the total space $\mathcal Q$ is smooth, (ii) every geometric fiber of $\pi_{\mathcal Q}$ is integral, (iii) the relative Picard rank of $\pi_{\mathcal Q}$ is $1$, and (iv) the generic fiber of $\pi_{\mathcal Q}$ is a smooth quadric surface.

In this section, we introduce the construction of a blow-up $\mathcal X$ of $\mathcal Q$ along a degree $4$ multisection $\mathcal Z \subset \mathcal Q$ such that $\mathcal X$ is a smooth wonderful model of the generic fiber $\mathcal X_\eta$ which is a smooth quartic del Pezzo surface over $K = k(B)$.

To this end, we fix a smooth integral projective curve $\mathcal Z$ of genus $g(\mathcal Z)$ and a degree $4$ map $g : \mathcal Z \to B$ such that 
\begin{itemize}
\item $g$ is unramified over the discriminant divisor $\mathfrak d \subset B$; and
\item $g$ is geometrically simply branched.
\end{itemize}
We consider the following base change diagram:
\[
\xymatrix{
g^*\mathcal Q \ar[r]^{\widetilde{g}} \ar[d]_{g^*\pi_{\mathcal Q}} & \mathcal Q \ar[d]^{\pi_{\mathcal Q}} \\
\mathcal Z \ar[r]_g & B
}
\]
Then $g^*\pi_{\mathcal Q} : g^*\mathcal Q \to \mathcal Z$ is a quadric surface bundle such that (i) the total space $g^*\mathcal Q$ is smooth, (ii) every geometric fiber of $g^*\pi_{\mathcal Q}$ is integral, (iii) the discriminant divisor is $g^*\mathfrak d$ which is reduced, and (iv) the generic fiber is smooth.
Since any smooth quadric surface is separably rationally connected, it follows from \cite{GHS03} (in characteristic $0$) and \cite{dejongstarr} (in positive characteristic) that $g^*\pi_{\mathcal Q}$ admits a section.

\begin{lemm}
\label{lemm:gluingarguments}
For any $m \geq 0$, there exists a section $\sigma : \mathcal Z \to g^*\mathcal Q$ such that its normal bundle
$
N_{\sigma}
$
satisfies 
\[
\mu^{\min}(N_{\sigma}) \geq m + 2g(\mathcal Z),
\]
where $\mu^{\min}(N_{\sigma})$ is the minimal slope of a locally free sheaf $N_\sigma$.
\end{lemm}

\begin{proof}
This is well-known to experts. We will briefly sketch a proof.
Let $s = m + 2g(\mathcal Z)$.
As explained above, we have a section $\sigma' : \mathcal Z \to g^*\mathcal Q$.
We fix a set of general points $z_1, \cdots, z_s$ on $\mathcal Z$.
We glue a number of very free vertical rational curves to $\mathcal Z$ and obtain a stable map $f : \mathcal Z' \to g^*\mathcal Q$. 
Then using \cite[Lemma 2.6]{GHS03} one can make the twisted normal sheaf $N_f(-z_1 - \cdots -z_s)$ globally generated. (Note that \cite{GHS03} assumes the ground field is $\mathbb C$, however, \cite[Lemma 2.6]{GHS03} works in positive characteristic too. See also \cite[Section 2.2]{BLRT21}.) Let $\sigma : \mathcal Z \to g^*\mathcal Q$ be a general smoothing of $\sigma'$. Then $N_\sigma(-z_1 - \cdots - z_s)$ is still globally generated. It follows from \cite[Lemma 2.7]{LRT25} that we have
\[
\mu^{\min}(N_{\sigma}(-z_1 - \cdots -z_s)) \geq 0.
\]
Thus our assertion follows.
\end{proof}

Here is the construction of wonderful models of smooth quartic del Pezzo fibrations:

\begin{theo}
\label{theo:constructionofwonderful}
Let $k, B, \pi_{\mathcal Q} : \mathcal Q \to B, g: \mathcal Z \to B$ as above.
Suppose that there is a section $\sigma' : \mathcal Z \to g^*\mathcal Q$ such that 
\[
\mu^{\min}(N_{\sigma'}) \geq 3 + 2g(\mathcal Z).
\]
For a general deformation $\sigma : \mathcal Z \to g^*\mathcal Q$ of $\sigma'$, the composition $\widetilde{g}\circ \sigma : \mathcal Z \to \mathcal Q$ is an embedding which is a degree $4$ multisection over $B$. Let $\mathcal X$ be the blow-up of $\mathcal Q$ along $\mathcal Z$. Then, $\mathcal X$ is a smooth wonderful model of a smooth quartic del Pezzo surface $\mathcal X_\eta$.
\end{theo}

\begin{proof}
Using \cite[Corollary 2.9]{LRT25}, for a general deformation $\sigma : \mathcal Z \to g^*\mathcal Q$ of $\sigma'$ and any four points $z_1, z_2, z_3, z_4$ on $\mathcal Z$, the twisted normal bundle
\[
N_\sigma(-z_1 - z_2-z_3),
\]
is globally generated and $H^1(\mathcal Z, N_\sigma(-z_1 - z_2-z_3-z_4)) = 0$. Note that the ground field in \cite{LRT25} is $\mathbb C$, but this corollary is valid over any algebraically closed field. By \cite[Proposition 3.3]{LRT25}, for any $r \leq 4$, $\sigma$ passes through $r$ general points on $g^*\mathcal Q$, and the locus of $\Sec(g^*\mathcal Q/\mathcal Z)$ parametrizing general sections passing through $r$ general points is smooth and has expected dimension. 
This shows that the composition $\widetilde{g}\circ \sigma : \mathcal Z \to \mathcal Q$ is an embedding. Indeed, we first prove that this is injective. Let $b \in B$ be a point and $z_1, \cdots, z_r$ be its preimages in $\mathcal Z$ with $r = 3$ or $4$. Then the codimension of the locus parametrizing general sections $\sigma$ with $\widetilde{g}\circ \sigma(z_1) = \widetilde{g}\circ \sigma(z_2)$ is $2$. Varying $b \in B$ is $1$-dimensional, so this shows that for a general member $\sigma$, $\widetilde{g}\circ \sigma$ is injective. The proof that it is an embedding is similar. When $g$ is unramified over $b \in B$, the composition $\widetilde{g}\circ \sigma(z_1)$ is automatically immersive at preimages of $b$ once it is injective. When $g$ is simply branched over $b \in B$,  $\widetilde{g}\circ \sigma$ will be \'etale to the image as long as $\sigma$ avoids a certain jet, and this is a codimension $2$ condition. This shows that $\widetilde{g}\circ \sigma$ is an embedding as long as $\sigma$ is general. We conclude that $\mathcal X = \mathrm{Bl}_{\mathcal Z}(\mathcal Q)$ is smooth.
Other properties follow from similar dimension-count arguments combined with Propositions~\ref{prop:whenweakdPI}, \ref{prop:whenweakdPII}, and \ref{prop:whenweakdPIII}.
Thus our assertion follows.
\end{proof}

\begin{coro}
\label{coro:examplesofwonderful}
Let $k, B, \pi_{\mathcal Q} : \mathcal Q \to B, g: \mathcal Z \to B$ as above.
Then, there exists a section $\sigma' : \mathcal Z \to g^*\mathcal Q$ of arbitrarily large height such that for a general deformation $\sigma : \mathcal Z \to g^*\mathcal Q$ of $\sigma'$, the statement of Theorem~\ref{theo:constructionofwonderful} holds.
\end{coro}

\begin{proof}
This follows from Theorem~\ref{theo:constructionofwonderful} and Lemma~\ref{lemm:gluingarguments}.
\end{proof}

\section{Geometry of the space of sections}
\label{sec:geometryofthespace}

The goal of this section is to understand the geometry of the space of sections for quartic del Pezzo fibrations.

\subsection{Revisiting sections on quadric surface bundles}

In this section, we recall some geometric constructions of the space of sections of quadric surface bundles as developed in \cite{HT12} and \cite{TT25}.

\subsubsection*{Sections of quadric surface bundles}
Let $k$ be a perfect field of characteristic not equal to $2$ and $B$ be a smooth projective geometrically integral curve defined over $k$. Let $$\pi_{\mathcal Q} : \mathcal Q  \subset \mathbb P(\mathcal E) \to B,$$ be a quadric surface bundle as defined in Section~\ref{subsec:quadric}.
In \cite{TT25}, Yuri Tschinkel and the author considered the space of sections inducing jets.
Let $\Sigma_B = \{ \hat{\sigma}_j\}_{j \in J}$ be the set of finitely many admissible jets $\hat{\sigma}_j$ supported at $b_j \in B$. (For the definition of admissible jets, see \cite[Definition 9]{HT06}.)
We consider the space of sections inducing $\Sigma_B$ which is a closed subscheme
\[
\Sec(\mathcal Q/B, h_{\mathcal Q}, \Sigma_B) \subset \Sec(\mathcal Q/B, h_{\mathcal Q}),
\]
parametrizing sections $\sigma : B \to \mathcal Q$ of height $h_{\mathcal Q}$ inducing $\Sigma_B$.
As discussed in Section~\ref{subsec:quadric}, we consider the relative family of lines
\[
F_1(\mathcal Q) \to C \to B,
\]
with the universal family $\mathcal U \subset\mathcal Q \times_B F_1(\mathcal Q) $.
As observed in \cite[Section 6]{TT25}, for each admissible jet $\hat{\sigma} : \Spec \, (k(b) [t]/(t^{N + 1}) )\to \mathcal Q$, we consider the Cartesian diagram
\[
\xymatrix{
\mathcal U_{\hat{\sigma}} \ar[r] \ar[d] &   \Spec \, (k(b) [t]/(t^{N + 1})) \times_B  F_1(\mathcal Q)\ar[d]^{\hat{\sigma} \times \mathrm{id}} \\
\mathcal U \ar[r] &   \mathcal Q \times_B F_1(\mathcal Q).
}
\]
When $\hat{\sigma}$ is supported at a point on $B$ such that $C \to B$ is not branched, the closed subscheme $\mathcal U_{\hat{\sigma}} \subset F_1(\mathcal Q)$ is geometrically the union of two copies of $$\Spec \, (\overline{k} [t]/(t^{N + 1})),$$ which define admissible $N$-th jets on $F_1(\mathcal Q)$. When $\hat{\sigma}$ is supported at a point on $B$ where $C \to B$ is branched, the closed subscheme $$\mathcal U_{\hat{\sigma}} \subset F_1(\mathcal Q),$$ is one copy of $\Spec \, (k(b) [t]/(t^{2N + 2}))$ which defines an admissible $(2N+1)$-th jet on $F_1(\mathcal Q)$.
Thus if $\Sigma_B$ is a set of admissible jets of $\pi_{\mathcal Q}$ of degree $$\deg \Sigma_B = \sum_{j \in J} [k(b_j):k]\mathrm{length}(\hat{\sigma}_j),$$ then this induces the set $\Sigma_C$ of admissible jets for $\pi_{F_1(\mathcal Q)}$ of degree $2\deg \Sigma_B$.
Hence the correspondence established in Section~\ref{subsec:quadric} induces the following isomorphism:
\begin{prop}[{\cite[Lemma 18]{TT25}}]
\label{prop:lemma18}
We have an isomorphism:
\[
\Sec(\mathcal Q/B, h_{\mathcal Q}, \Sigma_B)  \cong \Sec(F_1(\mathcal Q)/C, d_{F_1(\mathcal Q)}, \Sigma_C).
\]
\end{prop}

\subsubsection*{Abel--Jacobi mapping}

Next we analyze the geometry of the space of sections:
$$\Sec(F_1(\mathcal Q)/C, d_{F_1(\mathcal Q)}, \Sigma_C),$$ following \cite[Proposition 2]{HT12} and \cite[Section 4]{TT25}. Let $k$ be an algebraically closed field or a finite field so that $F_1(\mathcal Q) \to C$ is the projectivization of a rank $2$ locally free sheaf $\mathcal F$ on $C$, i.e., 
\[
F_1(\mathcal Q) \cong \mathbb P_C(\mathcal F).
\]
We have the well-known bijection between the set of sections and the set of surjections:
\[
\{ \tau : C \to \mathbb P_C(\mathcal F)\} \leftrightarrow \{ \mathcal F \twoheadrightarrow L\},
\]
which is given by 
\[
\tau \mapsto L := \tau^*\mathcal O_{\mathbb P_C(\mathcal F)/C}(1).
\]
We consider two projections
\[
\xymatrix{
& C \times \mathrm{Pic}^{d_{\mathbb P_C(\mathcal F)}}(C) \ar[ld]_{\varpi} \ar[rd]^{\varpi_{d_{\mathbb P_C(\mathcal F)}}}& \\
C &  & \mathrm{Pic}^{d_{\mathbb P_C(\mathcal F)}}(C) 
}
\]
and a universal family 
\[
\mathcal L_{d_{\mathbb P_C(\mathcal F)}} \to C \times \mathrm{Pic}^{d_{\mathbb P_C(\mathcal F)}}(C).
\]
There exists a constant $\mathsf c_1(\mathcal F)$ depending only on $\mathcal F$ such that assuming $d_{\mathbb P_C(\mathcal F)} \geq \mathsf c_1(\mathcal F)$, the sheaf
\[
\mathcal V_{d_{\mathbb P_C(\mathcal F)}} = \varpi_{d_{\mathbb P_C(\mathcal F)}*}(\varpi^*\mathcal F^\vee \otimes \mathcal L_{d_{\mathbb P_C(\mathcal F)}}),
\]
is a locally free sheaf of rank 
\[
2d_{\mathbb P_C(\mathcal F)} - \deg(\mathcal F) + 2(1-g(C)),
\]
and for any $L \in \Pic^{d_{\mathbb P_C(\mathcal F)}}(B_{\overline{k}})$, the sheaf 
\[
\mathcal F^\vee_{\overline{k}} \otimes L,
\]
is globally generated.
The above bijection realizes 
\[
\xymatrix{
\Sec(F_1(\mathcal Q)/C, d_{\mathbb P_C(\mathcal F)}) \ar[rd]_{\mathrm{AJ}} \ar@{^{(}->}[r] & \mathbb P(\mathcal V_{d_{\mathbb P_C(\mathcal F)}}^\vee) \ar[d] \\
& \Pic^{d_{\mathbb P_C(\mathcal F)}}(C)
}
\]
as a Zariski open subset of a projective bundle over $\Pic^{d_{\mathbb P_C(\mathcal F)}}(C)$
where $\mathrm{AJ}$ is the Abel--Jacobi mapping.
Moreover, since $\mathcal F^\vee_{\overline{k}} \otimes L$ is globally generated and $\mathcal F$ has rank $2$, we conclude that the Abel--Jacobi mapping
\[
\mathrm{AJ} : \Sec(F_1(\mathcal Q)/C, d_{\mathbb P_C(\mathcal F)}) \to \Pic^{d_{\mathbb P_C(\mathcal F)}}(C),
\]
is surjective.

Next suppose that we have a section $\tau : C \to \mathbb P_C(\mathcal F)$ of degree $d_{\mathbb P_C(\mathcal F)}$ inducing $\Sigma_C$.
The deformation theory of $\Sec(F_1(\mathcal Q)/C, d_{\mathbb P_C(\mathcal F)}, \Sigma_C)$ is controlled by the twisted normal bundle $N_\tau(-\Sigma_C)$ which has degree
\[
2d_{\mathbb P_C(\mathcal F)} - \deg(\mathcal F) - \deg (\Sigma_C).
\]
These observations imply the following:
\begin{lemm}
\label{lemm:spaceofsectionsP}
Assume that $2d_{\mathbb P_C(\mathcal F)} - \deg(\mathcal F) - \deg (\Sigma_C) \geq 2g(C)-1$. Then, we have $H^1(C, N_\tau(-\Sigma_C)) = 0$ so that $\Sec(F_1(\mathcal Q)/C, d_{\mathbb P_C(\mathcal F)}, \Sigma_C)$ is smooth and has expected dimension
\[
2d_{\mathbb P_C(\mathcal F)} - \deg(\mathcal F) - \deg (\Sigma_C) + 1 - g(C).
\]
\end{lemm}
We have the exact sequence:
\[
0 \to \mathcal O \to \mathcal F^\vee \otimes L \to T_{\mathbb P_C(\mathcal F)/C}|_C \cong N_\tau \to 0.
\]
Let
\[
 \mathcal F^\vee(-\Sigma_C) \otimes L,
\]
be the kernel of
\[
 \mathcal F^\vee \otimes L \to N_\tau/N_\tau(-\Sigma_C) \to 0.
\]
Note that  $\mathcal F^\vee(-\Sigma_C)$ does not depend on the choice of $L$ and $\tau$.
A nowhere vanishing section of
\[
\rH^0(C,  \mathcal F^\vee(-\Sigma_C) \otimes L) \subset \rH^0(C,  \mathcal F^\vee \otimes L)
\]
corresponds to  $\tau : C \to \mathbb P_C(\mathcal F)$ inducing $\Sigma_C$ with $\tau^*\mathcal O_{\mathbb P_C(\mathcal F)}(1) \cong L$.

A key to our argument is the following lemma:

\begin{lemm}
\label{lemm:slope}
Assume that our ground field is an algebraically closed field of characteristic $\neq 2$.
Assume the existence of $\tau$ as above.
Let $\mathsf c_2(\mathcal F) =  \deg(\mathcal F) + 2g(C)$. If $$2d_{\mathbb P_C(\mathcal F)} -\deg(\Sigma_C)\geq \mathsf c_2(\mathcal F),$$
 then the minimal slope $\mu^{\min}(\mathcal F^\vee(-\Sigma_C) \otimes L)$ satisfies
\[
\mu^{\min}(\mathcal F^\vee(-\Sigma_C) \otimes L) \geq 0.
\]
\end{lemm}

\begin{proof}
When $\mathcal F^\vee(-\Sigma_C) \otimes L$ is semistable, our assertion is obvious because the degree of this bundle is 
\[
2d_{\mathbb P_C(\mathcal F)} - \deg (\mathcal F) - \deg(\Sigma_C).
\]

Suppose that $\mathcal F^\vee(-\Sigma_C) \otimes L$ is unstable.
Let $$\mathcal L_1 \subset \mathcal F^\vee(-\Sigma_C) \otimes L,$$ be the maximal destabilizing invertible sheaf and we denote its quotient by $\mathcal L_2$. Then the minimal slope $\mu^{\min}(\mathcal F^\vee(-\Sigma_C) \otimes L)$ is equal to $\deg (\mathcal L_2)$.

We have the exact sequence
\[
0 \to \mathcal O \to \mathcal F^\vee(-\Sigma_C) \otimes L \to N_\tau(-\Sigma_C) \to 0.
\]
If $\mathcal L_1 \subset \mathcal O$, then the degree of $\mathcal L_2$ is greater than or equal to 
\[
2d_{\mathbb P_C(\mathcal F)} - \deg (\mathcal F) - \deg(\Sigma_C).
\]
Thus we may assume that the homomorphism $\mathcal L_1 \to N_\tau(-\Sigma_C)$ is non-zero.
This means that 
\[
\deg(\mathcal L_1) \leq 2d_{\mathbb P_C(\mathcal F)} - \deg (\mathcal F) - \deg(\Sigma_C).
\]
Equivalently we have $\deg(\mathcal L_2 ) \geq 0$.
Thus our assertion follows.

\end{proof}


Assume that 
\[
\mu^{\min}(\mathcal F^\vee(-\Sigma_C) \otimes L) \geq 2g(C),
\]
for every $L \in \Pic^{d_{\mathbb P_C(\mathcal F)}}(C)(\overline{k}).$
By \cite[Corollary 2.9]{LRT25}, 
\begin{align*}
\rH^1(C,  \mathcal F^\vee(-\Sigma_C) \otimes L) =0,
\end{align*}
and $\mathcal F^\vee(-\Sigma_C) \otimes L$ is globally generated.
Finally, it follows from the Riemann--Roch theorem that we have
\[
\dim \rH^0(C,  \mathcal F^\vee(-\Sigma_C) \otimes L)  = 2d_{\mathbb P_C(\mathcal F)} - \deg(\mathcal F)- \deg(\Sigma_C) + 2(1-g(C)).
\]
These show that  the fibers of the Abel--Jacobi mapping
\[
\mathrm{AJ}: \Sec(F_1(\mathcal Q)/C, d_{\mathbb P_C(\mathcal F)}, \Sigma_C) \to \mathrm{Pic}^{d_{\mathbb P_C(\mathcal F)}}(C), \, \tau \mapsto L,
\]
are smooth, equidimensional, and surjective. Thus, it follows from \cite[Lemma 2.1]{DLTT25} that $\mathrm{AJ}$ is a smooth morphism and we have $$\mathrm{AJ}(\Sec(F_1(\mathcal Q)/C, d_{\mathbb P_C(\mathcal F)}, \Sigma_C)) = \mathrm{Pic}^{d_{\mathbb P_C(\mathcal F)}}(C).$$
Moreover, the sheaf
\[
\mathcal V_{d_{\mathbb P_C(\mathcal F)}, \Sigma_C} = \varpi_{d_{\mathbb P_C(\mathcal F)}*}(\varpi^*\mathcal F^\vee(-\Sigma_C) \otimes \mathcal L_{d_{\mathbb P_C(\mathcal F)}}),
\]
is a locally free sheaf of rank 
\[
2d_{\mathbb P_C(\mathcal F)} - \deg(\mathcal F)- \deg(\Sigma_C) + 2(1-g(C)),
\]
over the image $\mathrm{AJ}(\Sec(F_1(\mathcal Q)/C, d_{\mathbb P_C(\mathcal F)}, \Sigma_C))$.
Moreover, it follows from our assumption that 
\[
\mathrm{AJ}(\Sec(F_1(\mathcal Q)/C, d_{\mathbb P_C(\mathcal F)}, \Sigma_C)) = \Pic^{d_{\mathbb P_C(\mathcal F)}}(C).
\]
Combining everything together, we have
\begin{prop} 
\label{prop:spaceofsectionsforQ}
Assume that 
\[
\mu^{\min}(\mathcal F^\vee(-\Sigma_C) \otimes L) \geq 2g(C)
\] holds for every $L \in \Pic^{d_{\mathbb P_C(\mathcal F)}}(C)(\overline{k})$.
Then we have an open embedding
\[
\xymatrix{
\Sec(F_1(\mathcal Q)/C, d_{\mathbb P_C(\mathcal F)}, \Sigma_C) \ar[rd]_{\mathrm{AJ}} \ar@{^{(}->}[r] & \mathbb P(\mathcal V_{d_{\mathbb P_C(\mathcal F)}, \Sigma_C}^\vee) \ar[d] \\
&\Pic^{d_{\mathbb P_C(\mathcal F)}}(C)
}
\]
which realizes $\Sec(F_1(\mathcal Q)/C, d_{\mathbb P_C(\mathcal F)}, \Sigma_C)$ as a Zariski open subset of a projective bundle over $\Pic^{d_{\mathbb P_C(\mathcal F)}}(C)$.
\end{prop}

\subsection{The space of sections on blow-up models}
\label{subsec:spaceofsectionsonblowup}

First we state our setup in this section:

\noindent
{\bf Setup}:
Let $k$ be an algebraically closed field or a finite field. We assume that the characteristic of $k$ is not equal to $2$.
Let $B$ be a geometrically integral smooth projective curve over $k$.
Let $\pi_{\mathcal Q} : \mathcal Q \to B$ be a quadric surface bundle over $B$ as in Section~\ref{subsec:quadric}.
We fix a geometrically integral smooth projective degree $4$ multisection $\mathcal Z \subset \mathcal Q$,
and let $\phi : \mathcal X \to \mathcal Q$ be the blow-up of $\mathcal Q$ along $\mathcal Z$.
We assume that the generic fiber is a smooth quartic del Pezzo surface.
This assumption implies that $\mathcal Z \to B$ is separable.
Under this setup, we study the space of sections of $\pi : \mathcal X \to B$.

\subsubsection*{The geometry of the space of sections of $\pi$}

First, we compute a basis for $N^1(\mathcal X)_{\mathbb Z}$. Let $F_{\mathcal Q}$ be a class of a general geometric fiber of $\pi_{\mathcal Q} : \mathcal Q \to B$. Let $\xi_{\mathcal Q}$ be the class of  $\mathcal O_{\mathbb P(\mathcal E)}(1)|_{\mathcal Q}$ where $\mathcal Q \subset \mathbb P(\mathcal E)$. Then, a basis for $N^1(\mathcal Q)_{\mathbb Z}$ is given by
\[
F_{\mathcal Q}, \xi_{\mathcal Q}.
\]
We denote the pullbacks of these classes to $\mathcal X$ by $F, \xi$, and let $E$ be the exceptional divisor for $\phi : \mathcal X \to \mathcal Q$. Then a basis for $N^1(\mathcal X)_{\mathbb Z}$ is given by
\[
F, \xi, E.
\]
The blow-up formula shows that the relative anticanonical divisor $-K_{\mathcal X/B}$ satisfies
\[
-K_{\mathcal X/B} \sim -K_{\mathcal Q/B} - E.
\]

Let $\alpha$ be a class of a section of $\pi : \mathcal X \to B$. This satisfies $F.\alpha = 1$ and it is specified by two numbers:
\[
h_{\mathcal Q}(\alpha) = -K_{\mathcal Q/B}.\alpha, \quad k(\alpha) = E.\alpha.
\]
These satisfy 
\[
h(\alpha) = -K_{\mathcal X/B}.\alpha = h_{\mathcal Q}(\alpha) - k(\alpha).
\]
We are interested in the space of sections of class $\alpha$:
\[
\mathrm{Sec}(\mathcal X/B, \alpha).
\]

We consider the following structure map:
let $$\mathrm{Hilb}^{[k(\alpha)]}(\mathcal Z),$$ be the Hilbert scheme of length $k(\alpha)$ zero-dimensional subschemes on $\mathcal Z$. Let $w\in \mathrm{Hilb}^{[k(\alpha)]}(\mathcal Z)$ which is supported at $\{z_j\} \subset \mathcal Z$, and assume that the multiplicity at $z_j$ is given by $m_j$. Since $\mathcal Z$ is smooth, it induces an $N_j$-th jet $\hat{\sigma}_j$ on $\mathcal Q$ where $N_j = m_j - 1$.
We denote the set of jets
\[
\{ \hat{\sigma}_j \}
\]
by $\Sigma_B(w)$. We define the Zariski open subset
\[
U_{k(\alpha)} = \{ [w] \in \mathrm{Hilb}^{[k(\alpha)]}(\mathcal Z) \, |\, \text{ $\Sigma_B(w)$ is admissible for $\pi_{\mathcal Q}$} \} \subset \mathrm{Hilb}^{[k(\alpha)]}(\mathcal Z).
\]
Then we obtain the following structure morphism:
\[
\Phi_\alpha : \mathrm{Sec}(\mathcal X/B, \alpha) \to U_{k(\alpha)}, [\sigma : B \to \mathcal X] \mapsto \phi_*(\sigma(B)\cap E).
\]
Regarding this structure morphism, we have the following result:
\begin{theo}
\label{theo:spaceofsectionsforX}
Let $C, \mathcal F, \Delta$ be as in Section~\ref{subsec:quadric}. 
Assume that we have
\[
h(\alpha) - k(\alpha) \geq  - \frac{\Delta}{2} +2g(C).
\]
Then, the space 
\[
\mathrm{Sec}(\mathcal X/B, \alpha)
\]
is of expected dimension $$h(\alpha) + 2(1-g(B)).$$ Moreover it is smooth and the structure morphism $\Phi_\alpha$ is dominant and smooth.
\end{theo}

\begin{proof}
We may assume that our ground field is an algebraically closed field.
Let $w \in U_{k(\alpha)}$ be an image of $\Phi_\alpha$. Then, the fiber $\Phi^{-1}(w)$ is isomorphic to
\[
\mathrm{Sec}(\mathcal Q/B, h_{\mathcal Q}(\alpha), \Sigma_B(w)).
\]
Indeed, the pushforward map gives
\[
\Phi^{-1}(w) \to \mathrm{Sec}(\mathcal Q/B, h_{\mathcal Q}(\alpha), \Sigma_B(w)), \sigma  \mapsto \phi\circ\sigma,
\]
and the inverse map is given by taking the strict transforms.
Thus, it follows from Proposition~\ref{prop:lemma18} and Lemma~\ref{lemm:spaceofsectionsP} that this space is smooth, and has expected dimension
\begin{align*}
&2d_{F_1(\mathcal Q)}(\alpha) - \deg(\mathcal F) - 2 \deg(\Sigma_B) + 1 - g(C)\\ &= h_{\mathcal Q}(\alpha) -2\deg(\Sigma_B) + 2 - 2g(B) = h(\alpha) - k(\alpha) + 2(1-g(B)).
\end{align*}
Here we use the equalities
\begin{align*}
&h_{\mathcal Q}(\alpha) = 2d_{F_1(\mathcal Q)}(\alpha) - \deg(\mathcal F) - \frac{\Delta}{2}, \quad h(\alpha) = h_{\mathcal Q}(\alpha) - k(\alpha)\\
& \deg(\Sigma_B(w)) = k(\alpha), \quad g(C) = 2g(B) -1 + \frac{\Delta}{2}.
\end{align*}
Since every component of $\mathrm{Sec}(\mathcal X/B, \alpha)$ has dimension $$\geq h(\alpha) + 2(1-g(B)) = h(\alpha) - k(\alpha) + 2(1-g(B)) + k(\alpha),$$ we conclude that every component of this scheme has expected dimension. 
It follows from \cite[Lemma 2.1]{DLTT25} that $\mathrm{Sec}(\mathcal X/B, \alpha)$ is smooth and $\Phi_\alpha$ is smooth. 
\end{proof}
\begin{defi}
We define
\[
\mathrm{neg}(\mathbb P(\mathcal F)/C) := \min \{ d_{\mathbb P_C(\mathcal F)} (\tau)\, | \, \tau \in \Sec(\mathbb P(\mathcal F)/C)(\overline{k})\}.
\]

\end{defi}

The following lemma is inspired by \cite[Lemma 4.11]{DLTT25}, and it is a key to our homological sieve arguments developed in Section~\ref{subsec:sieve}:
\begin{lemm}
\label{lemm:dimensionestimates}
Let $k$ be an algebraic closed field or a finite field.
We assume that $2d_{\mathbb P_C(\mathcal F)}  - 2k(\alpha) \geq \mathsf c_2(\mathcal F)$.
Then the complement
\[
U_{k(\alpha)} \setminus \Phi_\alpha (\mathrm{Sec}(\mathcal X/B, \alpha))
\]
has dimension 
\[
\leq \max\{2k(\alpha) + \deg(\mathcal F)-1 + 5g(C) - d_{\mathbb P_C(\mathcal F)} - \mathrm{neg}(\mathbb P(\mathcal F)/C), k(\alpha) - (2d_{\mathbb P_C(\mathcal F)}-2k(\alpha)) + \deg(\mathcal F)-1 + 8g(C)\}.
\]
 In particular, the codimension of this complement in $U_{k(\alpha)}$ is greater than or equal to 
 $$\min\left\{\frac{1}{2}(h(\alpha) - k(\alpha))-\frac{\deg(\mathcal F)}{2} + \frac{\Delta}{4}+1 - 5g(C) + \mathrm{neg}(\mathbb P(\mathcal F)/C), h(\alpha) - k(\alpha)  + \frac{\Delta}{2} -8g(C)+ 1\right\}.$$
\end{lemm}

\begin{proof}
We may assume that our ground field is an algebraically closed field.
Let $w \in U_{k(\alpha)}(k)$ and assume that it is not contained in $\Phi_\alpha (\mathrm{Sec}(\mathcal X/B, \alpha))$.
Under this assumption, for any $L \in \Pic^{d_{\mathbb P_C(\mathcal F)}}(C)(k)$, sections in 
\[
\rH^0(C,  \mathcal F^\vee(-\Sigma_C) \otimes L) 
\]
correspond to reducible divisors on 
\[
\pi_{\mathbb P_C(\mathcal F)} : \mathbb P_C(\mathcal F) \to C,
\]
which is the union of a section and vertical curves.
We consider the union 
\begin{align*}
&\mathsf K = \bigcup_L \mathbb P(\rH^0(C,  \mathcal F^\vee(-\Sigma_C) \otimes L) ) \subset  
\mathcal P_{d_{\mathbb P_C(\mathcal F)}},
\end{align*}
where $\mathcal P_{d_{\mathbb P_C(\mathcal F)}}$ is the space of effective Cartier divisors on $\mathbb P_C(\mathcal F)$ for the class of sections of degree $d_{\mathbb P_C(\mathcal F)}$ and $L$ runs over $\Pic^{d_{\mathbb P_C(\mathcal F)}}(C)(k)$.

Let $D = D_0 + \sum_i m_iP_i$ be a general divisor in such a locus where $D_0$ is a section and $P_i$'s are distinct vertical curves.
Let $\ell_1$ be the number of $\pi_{\mathbb P_C(\mathcal F)}$-vertical curves of $D$ not meeting $\Sigma_C$ and $\ell_2$ be the number of $\pi_{\mathbb P_C(\mathcal F)}$-vertical curves of $D$  meeting $\Sigma_C$, counted with multiplicities. 
It follows from the Riemann-Roch theorem applying to $N_{D_0/\mathbb P_C(\mathcal F)}$ that the dimension of $\mathsf K$
is at most
\[
\max\{2g(C), 2d_{\mathbb P_C(\mathcal F)} - 2\ell_1 - \ell_2 -\deg(\mathcal F) -\deg(\Sigma_C) + 1-g(C)\} + \ell_1.
\]
Now the above estimate must be greater than or equal to 
\[
2d_{\mathbb P_C(\mathcal F)} -\deg(\mathcal F) -\deg(\Sigma_C) + 1-g(C).
\]
This is only possible when
\[
2d_{\mathbb P_C(\mathcal F)} - 2\ell_1 - \ell_2 -\deg(\mathcal F) -\deg(\Sigma_C) + 1-g(C) \leq 2g(C),
\]
which implies that
\begin{equation}
\label{equation:ell1}
\ell_1 \geq 2d_{\mathbb P_C(\mathcal F)} -\deg(\mathcal F) -\deg(\Sigma_C) + 1-4g(C).
\end{equation}
We also have
\[
d_{\mathbb P_C(\mathcal F)} - \ell_1 - \ell_2 \geq \mathrm{neg}(\mathbb P(\mathcal F)/C). 
\]
Thus we obtain
\begin{equation}
\label{equation:ell2}
2k(\alpha) + \deg(\mathcal F)-1 + 4g(C) - d_{\mathbb P_C(\mathcal F)} - \mathrm{neg}(\mathbb P(\mathcal F)/C) \geq \ell_2.
\end{equation}

On the other hand, the dimension of the locus in
\begin{align*}
&\bigcup_{(w, L) \in U_{k(\alpha)} \times \Pic^{d_{\mathbb P_C(\mathcal F)}}(C)} \mathbb P(\rH^0(C,  \mathcal F^\vee(-\Sigma_C) \otimes L)),
\end{align*}
parametrizing divisors such that $\pi_{\mathbb P_C(\mathcal F)}$-vertical components have degree $\ell_1 + \ell_2$ and containing $w$ is at most
\[
2d_{\mathbb P_C(\mathcal F)} - \ell_1 - \ell_2 -\deg(\mathcal F) -2k(\alpha) + 1.
\]
because of \cite[Lemma 2.10]{LRT25} unless the section component is rigid, in the sense that $\mathrm{H}^0$ of the normal bundle is zero. 
Note that  \cite[Lemma 2.10]{LRT25} is stated in characteristic $0$, but its proof is valid in positive characteristic as well. Indeed, a key to this is Clifford's theorem for semistable bundles which can be found in, e.g., \cite{LN08} for a positive characteristic version.

Now there are only finitely many rigid sections of $\mathbb P_C(\mathcal F)$ and they intersect the image $\iota(\mathcal Z \times_BC)$ in finitely many points. Since the fiber above $w$ in the complement has dimension at least $\ell_1$, we conclude that
\begin{align*}
&\dim \,(U_{k(\alpha)} \setminus \Phi_\alpha(\mathrm{Sec}(\mathcal X/B, \alpha)) \\ &\leq \max\{\max_{\ell_2}\{g(C)+\ell_2\}, \max_{\ell_1, \ell_2}\{2d_{\mathbb P_C(\mathcal F)} - 2\ell_1 - \ell_2 -\deg(\mathcal F) -k(\alpha)  + 1\}\}.
\end{align*}
Combining with (\ref{equation:ell1}) and (\ref{equation:ell2}), we have
\begin{align*}
&\dim \, (U_{k(\alpha)}  \setminus \Phi_\alpha (\mathrm{Sec}(\mathcal X/B, \alpha))\\
&\leq \max\{ 2k(\alpha) + \deg(\mathcal F)-1 + 5g(C) - d_{\mathbb P_C(\mathcal F)} - \mathrm{neg}(\mathbb P(\mathcal F)/C), 3k(\alpha) - 2d_{\mathbb P_C(\mathcal F)} + \deg(\mathcal F)-1 + 8g(C)\}\\
&= \max\Big\{ k(\alpha)  - (d_{\mathbb P_C(\mathcal F)}-k(\alpha)-\deg(\mathcal F)+1 - 5g(C) + \mathrm{neg}(\mathbb P(\mathcal F)/C)),\\ & \qquad\qquad\qquad\qquad\qquad\qquad\qquad\qquad\qquad\qquad\qquad\qquad k(\alpha) - (h(\alpha) - k(\alpha)  + \frac{\Delta}{2} -8g(C)+ 1) \Big\}\\
&\leq  \max\Big\{ k(\alpha)  - \left(\frac{1}{2}(h(\alpha) - k(\alpha))-\frac{\deg(\mathcal F)}{2} + \frac{\Delta}{4}+1 - 5g(C) + \mathrm{neg}(\mathbb P(\mathcal F)/C)\right),\\ & \qquad\qquad\qquad\qquad\qquad\qquad\qquad\qquad\qquad\qquad\qquad\qquad k(\alpha) - (h(\alpha) - k(\alpha)  + \frac{\Delta}{2} -8g(C)+ 1) \Big\}.
\end{align*}

\end{proof}

Let $\alpha$ be a nef class of sections such that 
\[
h(\alpha) - k(\alpha) \geq -\deg (\mathcal F) -\frac{\Delta}{2} + \mathsf c_2(\mathcal F),
\]
and consider the space of sections of the class $\alpha$
\[
\mathrm{Sec}(\mathcal X/B, \alpha).
\]
We have the structure morphism
\[
\Phi_\alpha \times \mathrm{AJ} : \mathrm{Sec}(\mathcal X/B, \alpha) \to U_{k(\alpha)} \times \Pic^{d_{\mathbb P_C(\mathcal F)}}(C).
\]
\begin{prop}
\label{prop:dominantsmooth}
Assume that 
$$2d_{\mathbb P_C(\mathcal F)} - 2k(\alpha) \geq \mathsf c_2(\mathcal F) + 4g(C).$$
Then, there exists a closed subset $F \subset U_{k(\alpha)}$ such that
the image of $\Phi_\alpha\times \mathrm{AJ}$ is a Zariski open subset in $(U_{k(\alpha)}\setminus F) \times \Pic^{d_{\mathbb P_C(\mathcal F)}}(C)$, and $\Phi_\alpha \times \mathrm{AJ}$ is dominant and smooth over this image in $(U_{k(\alpha)}\setminus F) \times \Pic^{d_{\mathbb P_C(\mathcal F)}}(C)$.
Moreover one can find $F$ so that the codimension of $F$ in $U_{k(\alpha)}$ is greater than or equal to 
 $$\min\left\{\frac{1}{2}(h(\alpha) - k(\alpha))-\frac{\deg(\mathcal F)}{2} + \frac{\Delta}{4}+1 - 
 7g(C) + \mathrm{neg}(\mathbb P(\mathcal F)/C), h(\alpha) - k(\alpha)  + \frac{\Delta}{2} -12g(C)+ 1\right\}.$$
\end{prop}
\begin{proof}
We may assume that our ground field is an algebraically closed field.
First we consider the class $\alpha'$ such that $d_{\mathbb P_C(\mathcal F)}(\alpha')  = d_{\mathbb P_C(\mathcal F)} - 2g(C)$ and $k(\alpha') = k(\alpha)$. Let $F$ be
\[
U_{k(\alpha)} \setminus \Phi_{\alpha'}(\mathrm{Sec}(\mathcal X/B, \alpha')).
\]
It follows from Lemma~\ref{lemm:dimensionestimates} that the codimension of $F$ satisfies the inequality in the statement.

For any $w \in (U_{k(\alpha)} \setminus F)(k)$, it follows from Lemma~\ref{lemm:slope} that one can find $L' \in \Pic^{d_{\mathbb P_C(\mathcal F)} - 2g(C)}(C)$ such that
\[
\mu^{\min}(\mathcal F^\vee(-\Sigma_C) \otimes L') \geq 0.
\]
This means that for any $L \in \Pic^{d_{\mathbb P_C(\mathcal F)}}(C)$, we have
\[
\mu^{\min}(\mathcal F^\vee(-\Sigma_C) \otimes L) \geq 2g(C).
\]
Thus our assertion follows from Proposition~\ref{prop:spaceofsectionsforQ} and the proof of Theorem~\ref{theo:spaceofsectionsforX}.
\end{proof}

Next, we consider the universal family
\[
\mathfrak W \subset U_{k(\alpha)} \times \mathcal Z \to U_{k(\alpha)},
\]
which parametrizes $w = \Sigma_B$.
Because of admissible conditions, the pushforward to $U_{k(\alpha)} \times B$ realizes $\mathfrak W$ as a subscheme
\[
\mathfrak W \subset U_{k(\alpha)} \times B.
\]
We consider the Cartesian product
\[
\xymatrix{
\widetilde{\mathfrak W} \ar[r] \ar[d] &  \mathfrak W\times_B   \mathbb P_C(\mathcal F)\ar[d]^{\hat{\sigma} \times \mathrm{id}} \\
U_{k(\alpha)}  \times \mathcal U \ar[r] &  U_{k(\alpha)}  \times  \mathcal Q \times_B  \mathbb P_C(\mathcal F),
}
\]
and the corresponding subscheme
\[
\widetilde{\mathfrak W} \subset  U_{k(\alpha)} \times \mathbb P_C(\mathcal F) \to U_{k(\alpha)} \times C,
\]
which parametrizes $\Sigma_C$.
Moreover on $U_{k(\alpha)} \times C $, we have a homomorphism
\[
\widetilde{\varpi}^*\mathcal F^\vee  \to\widetilde{\varpi}^*\mathcal F^\vee|_{\widetilde{\mathfrak W}} \to N|_{\widetilde{\mathfrak W}},
\]
where
$
\widetilde{\varpi} : U_{k(\alpha)} \times C    \to C
$
is the projection and $N|_{\widetilde{\mathfrak W}}$ is the quotient of $\widetilde{\varpi}^*\mathcal F^\vee|_{\widetilde{\mathfrak W}}$ by the subbundle $\widetilde{\mathcal W} \subset \widetilde{\varpi}^*\mathcal F^\vee|_{\widetilde{\mathfrak W}}$ corresponding to the section $$\widetilde{\mathfrak W} \subset \widetilde{\mathfrak W} \times_C \mathbb P_C(\mathcal F), $$ via
the $\mathbb G_m$-torsor 
$$
\widetilde{\varpi}^*\mathcal F^\vee|_{\widetilde{\mathfrak W}} \setminus (\text{zero section}) \to   \widetilde{\mathfrak W} \times_C \mathbb P_C(\mathcal F).
$$
We denote the kernel of $\widetilde{\varpi}^*\mathcal F^\vee  \to N|_{\widetilde{\mathfrak W}}$ by 
\[
\widetilde{\varpi}^*\mathcal F^\vee(-\widetilde{\mathfrak W}).
\]

Here is a structure theorem for $\mathrm{Sec}(\mathcal X/B, \alpha)$:
\begin{theo}
\label{theo:spaceofsections_structure}
Let $\alpha$ be a nef class of sections such that  
\[
h(\alpha) - k(\alpha) \geq -\frac{\Delta}{2}  + 6g(C).
\]
Let $F \subset U_{k(\alpha)}$ be a closed subset from Proposition~\ref{prop:dominantsmooth}.
Then the structure morphism
\[
\Phi_\alpha \times \mathrm{AJ}|_{(\Phi_\alpha \times \mathrm{AJ}(\mathrm{Sec}(\mathcal X/B, \alpha)))\cap ((U_{k(\alpha)} \setminus F) \times \Pic^{d_{\mathbb P_C(\mathcal F)}}(C))}
\]
 is a Zariski open subset of the projective bundle 
 \[
 \mathbb P(((\widetilde{\varpi}_{d_{\mathbb P_C(\mathcal F)}})_*(q^*\widetilde{\varpi}^*\mathcal F^\vee(-\widetilde{\mathfrak W}) \otimes p^*\mathcal L_{d_{\mathbb P_C(\mathcal F)}}))^\vee)
 \]
 over $(\Phi_\alpha \times \mathrm{AJ}(\mathrm{Sec}(\mathcal X/B, \alpha)))\cap ((U_{k(\alpha)} \setminus F) \times \Pic^{d_{\mathbb P_C(\mathcal F)}}(C))$ where morphisms
 \begin{align*}
 &p: U_{k(\alpha)} \times C \times \Pic^{d_{\mathbb P_C(\mathcal F)}}(C) \to C \times \Pic^{d_{\mathbb P_C(\mathcal F)}}(C),\\
 &q : U_{k(\alpha)} \times C \times \Pic^{d_{\mathbb P_C(\mathcal F)}}(C) \to U_{k(\alpha)} \times C,\\
 &\widetilde{\varpi}_{d_{\mathbb P_C(\mathcal F)}} : U_{k(\alpha)} \times C \times \Pic^{d_{\mathbb P_C(\mathcal F)}}(C) \to U_{k(\alpha)} \times \Pic^{d_{\mathbb P_C(\mathcal F)}}(C),
 \end{align*} 
 are the projections.
 
 In particular, $\mathrm{Sec}(\mathcal X/B, \alpha)$ is geometrically irreducible
 if $h(\alpha) - k(\alpha)$ is sufficiently large.
 \end{theo}
 \begin{proof}
 This follows from Proposition~\ref{prop:spaceofsectionsforQ} and Proposition~\ref{prop:dominantsmooth}.
 
 \end{proof}
 
 When $k = \mathbb C$, the intermediate Jacobian of $\mathcal X$ is given by
 \[
 \mathrm{IJ}(\mathcal X) = \Pic^0(\mathcal Z) \times \Pic^0(C).
 \]
 Using this, we have the following corollary, which is analogous to \cite[Theorem 4.2]{Tani26}:
 \begin{coro}
 \label{coro:MRC}
 Let $k = \mathbb C$. 
 Let $\alpha$ be a nef class of sections such that 
$
h(\alpha) - k(\alpha)
$
is sufficiently large.
 Assume further that $k(\alpha) \geq 2g(\mathcal Z)-1$.
 Then, the Abel--Jacobi mapping
 \[
\mathrm{AJ}_{\mathcal X} : \mathrm{Sec}(\mathcal X/B, \alpha) \to  \mathrm{IJ}(\mathcal X)
 \]
 is an MRC fibration.
 \end{coro}
 \begin{proof}
   Note that we have a morphism
   \[
   \rho_\alpha : U_{k(\alpha)} \times \Pic^{d_{\mathbb P_C(\mathcal F)}}(C) \to \Pic^0(\mathcal Z) \times \Pic^0(C),
   \]
   and one can prove that the composition $\rho_\alpha \circ (\Phi_\alpha \times \mathrm{AJ})$ is the Abel--Jacobi mapping $\mathrm{AJ}_{\mathcal X}$. 
   (See the proofs of \cite[Lemma 6.1]{CTT25} and \cite[Example 10.4]{LT19} for a proof of this fact.)
   Thus, our assertion follows from  Theorem~\ref{theo:spaceofsections_structure}.
 \end{proof}

 \subsubsection*{Dimension estimates}

Next let us fix a (possibly singular) multisection $\mathcal W \subset \mathbb P_C(\mathcal F)$ of degree $4$ such that $\mathcal W$ does not contain a section.
Let $\mathcal M^{\mathsf k}$ be the parameter space of length $\mathsf k$ admissible jets on $\mathcal W$ over $C$.
Assume that $2d_{\mathbb P_C(\mathcal F)/C} - \mathsf k$ is sufficiently large.
Under this assumption for any $\tau \in \mathrm{Sec}(\mathbb P_C(\mathcal F)/C, d_{\mathbb P_C(\mathcal F)})$, the intersection number of $\tau(C)$ and $\mathcal W$ is greater than $\mathsf k$.
In particular, there are finitely many length $\mathsf k$ admissible subschemes on $\tau(C)\cap \mathcal W$ over $C$.
Let $\rho(\mathrm{Sec}(\mathbb P_C(\mathcal F)/C, d_{\mathbb P_C(\mathcal F)})) \subset \mathcal M^{\mathsf k}$ be the image of this correspondence.
Then we have
\begin{lemm}
\label{lemm:dimensionestimates3}
The image $\rho(\mathrm{Sec}(\mathbb P_C(\mathcal F)/C, d_{\mathbb P_C(\mathcal F)}))$ is a dense constructible set. Its complement has dimension at most $\max\{\mathsf k  - (d_{\mathbb P_C(\mathcal F)}-\deg(\mathcal F)+1 - 5g(C) + \mathrm{neg}(\mathbb P(\mathcal F)/C)), \mathsf k - (2d_{\mathbb P_C(\mathcal F)}-\mathsf k) + \deg(\mathcal F)-1 + 8g(C)\}.$
\end{lemm}

\section{Poschemes and stratifications}
\label{sec:poschemes}

In this section, we set up poschemes and stratifications parametrized by poschemes on the space of sections so that we can execute the inclusion-exclusion principle using the bar complex developed in \cite{DT24} and \cite{DLTT25}. We mainly follow the discussion of \cite[Section 5]{DLTT25}.

\subsection{Setup}
\label{subsec:pos_setup}


Let $k$ be a perfect field of characteristic $\neq 2, 3$.
Let us work in the setup in Section~\ref{subsec:spaceofsectionsonblowup}. We fix
\[
\pi_{\mathcal Q} : \mathcal Q \to B,\quad  \mathcal Z \subset \mathcal Q, \quad \pi : \mathcal X \to B, \quad  \phi : \mathcal X \to \mathcal Q,
\]
as before.
In this section we work under the following assumptions:
\begin{itemize}
\item $\mathcal Q, \mathcal Z, \mathcal X$ are smooth;
\item the generic fiber $\mathcal X_\eta$ is a smooth quartic del Pezzo surface; and
\item $\mathcal X$ is a wonderful model of $\mathcal X_\eta$.
\end{itemize}
As before, we consider the relative Fano variety of lines
\[
F_1(\mathcal Q) \to C \to B,
\]
and its universal family 
\[
\mathcal U \hookrightarrow \mathcal Q\times_B F_1(\mathcal Q).
\]
We consider its base change
\[
\xymatrix{
\mathcal U_{\mathcal Z} \ar[r] \ar[d] &   \mathcal Z \times_B  F_1(\mathcal Q)\ar[d] \\
\mathcal U \ar[r] &   \mathcal Q \times_B F_1(\mathcal Q).
}
\]
Then this base change admits a natural morphism
\[
\mathcal U_{\mathcal Z} \to \mathcal Z \times_B C.
\]
Under our assumption, the above morphism is a flat, finite, and birational morphism because $\mathcal Z$ avoids vertices of singular fibers of $\pi_{\mathcal Q}$.
In particular, this morphism is an isomorphism.
We denote the induced morphism $$\mathcal W \to F_1(\mathcal Q)$$ by $\iota$
where $\mathcal W = \mathcal Z\times_B C$.
This morphism is birational to the image, but not a closed embedding in general.
Let $z$ be any $0$-dimensional subscheme of $\mathcal W$.
We denote the scheme-theoretic image of $z$ via $\iota$ by
\[
\iota_\diamond z.
\]
We also denote projections
\[
\mathcal W \to \mathcal Z, \quad \mathcal W \to C, \quad \mathcal W \to B,
\]
by $p_{\mathcal Z}, p_C, p_B$ respectively.

\subsection{Poschemes and stratifications}

\subsubsection*{Poschemes}

For a smooth geometrically integral projective curve $B$ defined over $k$, we let
\[
\mathrm{Hilb}^{\mathrm{fin}}(B) = \bigsqcup_{n = 0}^\infty \mathrm{Hilb}^{[n]}(B)
\]
be the Hilbert scheme parametrizing all $0$-dimensional subschemes of $B$.
Then, we define the poscheme
\[
\mathfrak H =\{x = (x_B, x_{\mathcal Z}, x_C, \hat{1}) \in \mathrm{Hilb}^{\mathrm{fin}}(B) \times \mathrm{Hilb}^{\mathrm{fin}}(\mathcal Z)\times \mathrm{Hilb}^{\mathrm{fin}}(C)  \times\{ \hat{1}\} \, | \, \text{$x$ satisfies $(*)$}\},   
\]
where $(*)$ means that 
\[
p_{\mathcal Z \to B}^*x_B \subset x_\mathcal Z, \quad p_{C \to B}^*x_B \subset x_C, 
\]
where $p_{\mathcal Z \to B} : \mathcal Z \to B$ and $p_{C \to B} : C \to B$ are projections and the poset structure is given by the component-wise inclusion.

Also let $\widetilde{\mathfrak Q} = B  \sqcup \mathcal Z \sqcup C \to B$ and we view this scheme as a poscheme over $B$ in the following way: for any connected scheme $T \to B$, we define the poset relation on $\widetilde{\mathfrak Q}(T)$ by
\begin{align*}
 p_{\mathcal Z \to  B}\circ t < t : T  \to \mathcal Z, \quad p_{C \to  B}\circ u < u : T  \to C.
\end{align*}
We let $$\mathfrak Q = \widetilde{\mathfrak Q} \sqcup (\{\hat{1}\} \times B) \to B,$$ where $\hat{1}$ is the largest element.

Suppose that we fix an element $x \in \mathfrak H(\overline{k})$. For each $b \in B(\overline{k})$, we denote its fiber of $\widetilde{\mathfrak Q} \to B$ by $\widetilde{\mathfrak Q}_b$.
Then $x$ induces a function $\widetilde{g}_b : \widetilde{\mathfrak Q}_b(\overline{k}) \to \mathbb N$ by assigning for each $q \in \widetilde{\mathfrak Q}_b(\overline{k})$ the length of $x$ at $q$.
This satisfies the following relations: $(\clubsuit)$ for any $p < q \in \widetilde{\mathfrak Q}_b(\overline{k})$, 
\[
\begin{cases}
\mathrm{lengh}_q(p^*_{\mathcal Z \to B}(p))\widetilde{g}_b(p)\leq \widetilde{g}_b(q)& \text{ if } p \in B, q \in \mathcal Z,\\
 \mathrm{lengh}_q(p^*_{C \to B}(p))\widetilde{g}_b(p) \leq \widetilde{g}_b(q)& \text{ if } p \in B, q \in C.
\end{cases}
\]
We say $\widetilde{g}_b$ is trivial if $\widetilde{g}_b(q) = 0$ for any $q \in  \widetilde{\mathfrak Q}_b(\overline{k})$.
Then $x$ induces
\[
\widetilde{g} : B(\overline{k}) \to \bigoplus_{b \in B(\overline{k})}\mathrm{Hom}_{(\clubsuit)}(\widetilde{\mathfrak Q}_b(\overline{k}), \mathbb N),
\]
where $(\clubsuit)$ indicates that the corresponding function $\widetilde{g}_b$ satisfies $(\clubsuit)$.
Conversely, if we have such a function $\widetilde{g}$ then one can recover an element $x \in \mathfrak H(\overline{k})$.  Also sometimes we may interpret $\widetilde{g}$ as
\[
g : B(\overline{k}) \to \bigoplus_{b \in B(\overline{k})}\mathrm{Hom}_{(\clubsuit)}(\mathfrak Q_b(\overline{k}), \mathbb N \cup \{ \infty\}),
\]
by assigning $g_b(\hat{1}) = \infty$.

\subsubsection*{Saturated elements}

Next, we define saturated elements:

\begin{defi}
\label{defi:saturated}
Let $x = (x_B, x_{\mathcal Z}, x_C, \hat{1}) \in \mathfrak H(\overline{k})$.
We say $x$ is saturated if the following properties hold:
\begin{enumerate}
\item the divisor $(x_{C} - p_{C \to B}^*x_B)$ induces admissible jets on $C$ over $B$; and 
\item  the $0$-dimensional subscheme $\iota_\diamond(p_{\mathcal Z}^*x_{\mathcal Z} - p_{\mathcal Z}^*x_{\mathcal Z} \wedge p_C^*x_C)$ induces admissible jets on $F_1(\mathcal Q)$ over $C$; and
\item for any $x_{\mathcal Z}' \geq x_{\mathcal Z}$, if we have
\[
\iota^*\iota_\diamond(p_{\mathcal Z}^*x_{\mathcal Z} - p_{\mathcal Z}^*x_{\mathcal Z} \wedge p_C^*x_C) \geq p_{\mathcal Z}^*x_{\mathcal Z}' - p_{\mathcal Z}^*x_{\mathcal Z} \wedge p_C^*x_C \geq p_{\mathcal Z}^*x_{\mathcal Z} - p_{\mathcal Z}^*x_{\mathcal Z} \wedge p_C^*x_C,
\]
then $x_{\mathcal Z}' = x_{\mathcal Z}$.

\end{enumerate}
Here $\wedge$ means the schematic intersection and the difference of $0$-dimensional subschemes is defined using the ideal quotient.

Let $\mathfrak H^\circ(\overline{k})$ be the set of saturated elements.  It is invariant under the Galois action. Thus it makes sense to discuss the set of $k$-rational points $\mathfrak H^\circ(k)$. Note that in \cite{DLTT25}, $\mathfrak H^\circ(\overline{k})$ was a Zariski open subset in $\mathfrak H(\overline{k})$, and hence we treated this as an open subscheme. In the current paper, this is not the case. However, it turns out that this is not a serious issue for running the homological sieve method in \cite{DLTT25}.

\end{defi}

\begin{rema}
If the divisor $(x_{\mathcal Z} - p_{\mathcal Z \to B}^*x_B)$ induces admissible jets on $\mathcal Z$ over $B$, then $\iota_\diamond(p^*_{\mathcal Z}x_{\mathcal Z} - p_B^*x_B)$ induces admissible jets over $C$ so that $\iota_\diamond(p_{\mathcal Z}^*x_{\mathcal Z} - p_{\mathcal Z}^*x_{\mathcal Z} \wedge p_C^*x_C)$ does too. However, the converse fails in general. If we assume that $-K_{\mathcal X/B}$ is relatively ample, then for any $b \in B(\overline{k})$ there is no line on $\mathcal Q_b$ which cuts out a length $2$ subscheme of $\mathcal Z_b$. This implies that the morphism
\[
\mathcal Z\times_BC \to F_1(\mathcal Q)
\]
is a closed embedding so that $\iota^*(\iota_\diamond(p_{\mathcal Z}^*x_{\mathcal Z} - p_{\mathcal Z}^*x_{\mathcal Z} \wedge p_C^*x_C)) = p_{\mathcal Z}^*x_{\mathcal Z}- p_{\mathcal Z}^*x_{\mathcal Z} \wedge p_C^*x_C$ is always true and the converse holds. However, such smooth ample models are rather rare.
Also we should note that when $-K_{\mathcal X/B}$ is relatively ample, $\mathfrak H^\circ(\overline{k}) \subset \mathfrak H(\overline{k})$ is a Zariski open subset.
\end{rema}

In terms of the function $g$, the conditions for saturated elements are equivalent to the following conditions: 

\begin{class}
\label{class1}
Let $b \in B(\overline{k})$. 
The condition (1) is equivalent to:
\begin{itemize}
\item if $p_{C \to B}^*(b) = c_1 + c_2$ where $c_i$'s are distinct, then $g_b(b)$ is equal to one of $$ g_b(c_1), g_b(c_2)$$ and the other may be greater than or equal to $g_b(b)$; and
\item if $p_{C \to B}^*(b) = 2c_1$, then $g_b(b)$ is equal to $\lfloor g_b(c_1)/2\rfloor$.
\end{itemize}
\end{class}

\begin{class}
\label{class2}
Next we describe the conditions (2) and (3) in terms of $g$.
First we assume that $p_{C \to B}^*(b) = c_1 + c_2$ where $c_i$'s are distinct.
\begin{itemize}
\item Assume that $p_{\mathcal Z \to B}^*(b) = z_1 + z_2 + z_3 + z_4$ where $z_i$'s are distinct and that $g_b(c_1) = g_b(b)$.
We may assume that
\[
g_b(z_1) \geq g_b(z_2) \geq g_b(z_3) \geq g_b(z_4).
\]
First, assume that even after relabeling while keeping the above inequalities, there is no line in the ruling corresponding to $c_1$ on $\mathcal Q_b$ which cuts out a length $2$ subscheme $z_1 + z_2$ of $\mathcal Z_b$. Then the conditions (2) and (3) are equivalent to the fact that $x_{\mathcal Z} - p_{\mathcal Z \to B}^*x_B$ induces admissible jets on $\mathcal Z$ over $B$. This is equivalent to the following conditions:
$g_b(b)$ is equal to $g_b(z_i)$ for all $i = 2, 3, 4$ and $g_b(z_1)$ is greater than or equal to $g_b(b)$.

Next, we may assume that $z_1, z_2$ are on a line in the ruling corresponding to $c_1$ on $\mathcal Q_b$.
Then $\Spec \, \mathcal O_{\mathcal Z, z_i}$ for $i = 1, 2$ induces jets supported at the same point on the fiber $F_1(\mathcal Q)_{c_1}$, and let $n$ be the maximal length for which two jets coincide. (Note that $n$ could be $\infty$.)
In this case, the conditions (2) and (3) are
\[
g_b(z_3) = g_b(z_4)  = g_b(b),  \quad g_b(b)\leq g_b(z_2) = \min\{ g_b(z_1), g_b(c_2), n + g_b(b)\}.
\]

\item Assume that $p_{\mathcal Z \to B}^*(b) = 2z_1 + z_2 + z_3$ where $z_i$'s are distinct and that $g_b(c_1) = g_b(b)$.
We may assume that
\[
g_b(z_2) \geq g_b(z_3).
\]
First assume that even after relabeling $z_2, z_3$ while keeping the above inequalities, there is no line in the ruling corresponding to $c_1$ on $\mathcal Q_b$ which cuts out length $2$ subschemes $z_1 + z_2$ or $2z_1$ of $\mathcal Z_b$.
Then the conditions (2) and (3) are equivalent to the fact that $x_{\mathcal Z} - p_{\mathcal Z \to B}^*x_B$ induces admissible jets on $\mathcal Z$ over $B$.
This is described as $g_b(b)$ is equal to $g_b(z_1)/2, g_b(z_2), g_b(z_3)$ except one which is greater than or equal to $g_b(b)$ and $2g_b(b) \leq g_b(z_1) \leq 1 + 2g_b(b)$.

Next let us assume that $g_b(z_1)/2 \geq g_b(z_2)$ 
and that $2z_1$ is cut out by a line in the ruling corresponding to $c_1$ on $\mathcal Q_b$. The conditions (2) and (3) are that 
\begin{itemize}
\item $g_b(z_1) = 2g_b(b), 2 + 2g_b(b)$,
$
g_b(z_2) = g_b(z_3)  = g_b(b),   2g_b(b) \leq
g_b(z_1) \leq  2g_b(c_2)
$; or 
\item $g_b(z_2) = g_b(z_3) = g_b(c_2) = g_b(b)$ and $2g_b(b)\leq g_b(z_1)\leq 1+2g_b(b)$.
\end{itemize}

Next let us assume that $g_b(z_1)/2 \geq g_b(z_3)$ 
and that $z_1 + z_2$ is cut out by a line in the ruling corresponding to $c_1$ on $\mathcal Q_b$.
Then $\Spec \, \mathcal O_{\mathcal Z, z_i}$ for $i = 1, 2$ induce jets supported at the same point on the fiber $F_1(\mathcal Q)_{c_1}$, and the maximal length to which two jets coincide is $1$.
In this case, we have
\begin{itemize}
\item $
g_b(z_3)  = g_b(b),
 g_b(z_1)=  \min\{2g_b(z_2), 1 + 2g_b(b)\},
$ $g_b(c_2) \geq   \min\{g_b(z_2), g_b(z_1)/2\}$; or 
\item $g_b(z_3)  = g_b(c_2) =g_b(b)$, $g_b(b) = \min\{g_b(z_1)/2, g_b(z_2)\}$ and $g_b(z_1) = 2g_b(b)$ or $2g_b(b)+1$.
\end{itemize}

Finally let us assume that $g_b(z_3) \geq g_b(z_1)/2$ and there is a line in the ruling corresponding to $c_1$ on $\mathcal Q_b$ which cuts out a length $2$ subscheme $z_2 + z_3$ of $\mathcal Z_b$. 
Then $\Spec \, \mathcal O_{\mathcal Z, z_i}$ for $i = 2, 3$ induce jets supported at the same point on the fiber $F_1(\mathcal Q)_{c_1}$, and let $n$ be the maximal length to which two jets coincide.
In this case, the conditions are 
$
g_b(z_1) = 2g_b(c_1) = 2g_b(b),  g_b(z_2)  \geq g_b(z_3) \geq g_b(b), 
 g_b(z_3) = \min\{ g_b(z_2), g_b(c_2), n + g_b(b)\}.
$

\item Assume that $p_{\mathcal Z \to B}^*(b) = 3z_1 + z_2$ where $z_i$'s are distinct and that $g_b(c_1) = g_b(b)$.
First assume that there is no line in the ruling corresponding to $c_1$ on $\mathcal Q_b$ which cuts out length $2$ subschemes $z_1 + z_2$ or $2z_1$ of $\mathcal Z_b$.
 Then $g_b(b)$ is equal to one of $g_b(z_1)/3, g_b(z_2)$ and another may be greater than or equal to $g_b(b)$. Moreover when $g_b(z_1)/3$ is higher, we have $g_b(z_1)/3 - g_b(b) = 1/3$.

 Let us assume that $g_b(z_1)/3\geq  g_b(z_2)$ and that $2z_1$ is cut out by a line in the ruling corresponding to $c_1$ on $\mathcal Q_b$.
We have
\begin{itemize}
\item$
g_b(z_2)  = g_b(b),
$
$
g_b(z_1) =3g_b(b)  \text{ or } 2 + 3g_b(b), g_b(z_1) \leq 3g_b(c_2);
$
or 
\item $
g_b(z_2)  = g_b(b),
$ $g_b(c_2) = g_b(b)$, and $g_b(z_1) =3g_b(b)  \text{ or } 1 + 3g_b(b)$.
\end{itemize}

Finally let us assume that $z_1, z_2$ are cut out by a line in the ruling corresponding to $c_1$ on $\mathcal Q_b$.
Then $\Spec \, \mathcal O_{\mathcal Z, z_i}$ for $i = 1, 2$ induce jets supported at the same point on the fiber $F_1(\mathcal Q)_{c_1}$, and the maximal length to which two jets coincide is $1$.
In this case, we have
\begin{itemize}
\item $
g_b(z_1) = 3g_b(b) \text{ or } 1 + 3g_b(b), g_b(z_1) \leq g_b(c_2), g_b(z_1) = \min\{3g_b(z_2), 1 + 3g_b(b)\};
$
or 
\item $g_b(c_2) = g_b(b), g_b(z_1) = 3g_b(b) \text{ or } 1 + 3g_b(b)$, and $g_b(b) = \min\{g_b(z_1)/3, g_b(z_2)\}$.
\end{itemize}

\item Assume that $p_{\mathcal Z \to B}^*(b) = 2z_1 +2z_2$ where $z_i$'s are distinct and that $g_b(c_1) = g_b(b)$.
We may assume that 
\[
g_b(z_1) \geq g_b(z_2).
\]
First assume that even after relabeling, there is no line in the ruling corresponding to $c_1$ on $\mathcal Q_b$ which cuts out length $2$ subschemes $z_1 + z_2$ or $2z_1$ of $\mathcal Z_b$.
Then $g_b(b)$ is equal to $g_b(z_2)/2$ and $g_b(z_1)/2$ may be greater than or equal to $g_b(b)$. Moreover when it is higher, the difference is $1/2$.

Next assume that $2z_1$ is cut out by a line in the ruling corresponding to $c_1$ on $\mathcal Q_b$.
In this case,  we have
\begin{itemize}
\item $g_b(z_1) = 2g_b(b)$ or $2 + 2g_b(b)$, $
g_b(z_2) = 2g_b(b), 2g_b(b) \leq g_b(z_1) \leq 2g_b(c_2);
$
or 
\item $g_b(z_2) = 2g_b(b), g_b(c_2) = g_b(b)$ and $g_b(z_1) = 2g_b(b)$ or $1+ 2g_b(b)$.
\end{itemize}

Next assume that $z_1, z_2$ are cut out by a line in the ruling corresponding to $c_1$ on $\mathcal Q_b$.
In this case, we have
\begin{itemize}
\item
$
g_b(z_1) = g_b(z_2) = 2g_b(b) \text{ or } 1+ 2g_b(b),  g_b(z_1) = g_b(z_2)  \leq 2g_b(c_2);
$
or 
\item $g_b(c_2) = g_b(b)$, $2g_b(b) = g_b(z_2)$, and $g_b(z_1) = 2g_b(b) \text{ or } 1+ 2g_b(b)$.
\end{itemize}

\item Assume that $p_{\mathcal Z \to B}^*(b) = 4z_1$ and that $g_b(c_1) = g_b(b)$.
First assume that there is no line in the ruling corresponding to $c_1$ on $\mathcal Q_b$ which cuts out a length $2$ subscheme $2z_1$ of $\mathcal Z_b$.
Then $g_b(b)$ is equal to $\lfloor g_b(z_1)/4\rfloor$. Moreover, when $g_b(z_1)/4$ is greater than $g_b(b)$, then the difference is $1/4$.

Next assume that $2z_1$ is cut out by a line in the ruling corresponding to $c_1$ on $\mathcal Q_b$. In this case, we have
\begin{itemize}
\item $
g_b(z_1) = 4g_b(b) \text{ or } 2+ 4g_b(b), g_b(z_1) \leq 4g_b(c_2);
$
or 
\item $g_b(c_2)=g_b(b)$ and $g_b(z_1) = 4g_b(b) \text{ or } 1+ 4g_b(b)$.
\end{itemize}
\end{itemize}
\end{class}
\begin{class}
\label{class3}
Finally we assume that $p_{C \to B}^*(b) = 2c_1$.
The conditions (2) and (3) are described as follows:
\begin{itemize}
\item When $p_{\mathcal Z \to B}^*(b) = z_1 + z_2 + z_3 + z_4$ where $z_i$'s are distinct.
We may assume that
\[
g_b(z_1) \geq g_b(z_2) \geq g_b(z_3) \geq g_b(z_4).
\]
First assume that even after relabeling, there is no line on $\mathcal Q_b$ which cuts out a length $2$ subscheme $z_1 + z_2$ of $\mathcal Z_b$. Then the conditions (2) and (3) are equivalent to the following conditions:
$g_b(b)$ is equal to $g_b(z_i)$ for all $i = 2, 3, 4$ and $g_b(z_1)$ is greater than or equal to $g_b(b)$.

Next we may assume that $z_1, z_2$ are on a line on $\mathcal Q_b$.
Then $\Spec \, \mathcal O_{\mathcal Z, z_i}$ for $i = 1, 2$ induce jets of twice the length supported at the same point on the fiber $F_1(\mathcal Q)_{c_1}$, and it follows from Proposition~\ref{prop:lemma18} that the maximal length to which two jets coincide is $1$. 
In this case, we have 
$
g_b(z_3) = g_b(z_4) =  g_b(b), g_b(z_2)  =\min\{ g_b(z_1), \lfloor (1 + g_b(c_1))/2  \rfloor\}.
$

\item Assume that $p_{\mathcal Z \to B}^*(b) = 2z_1 + z_2 + z_3$ where $z_i$'s are distinct. 
We may assume that
\[
g_b(z_2) \geq g_b(z_3).
\]
First assume that even after relabeling $z_2, z_3$ while keeping the above inequalities, there is no line on $\mathcal Q_b$ which cuts out length $2$ subschemes $z_1 + z_2$ or $2z_1$ of $\mathcal Z_b$.
Then the conditions (2) and (3) are equivalent to the following conditions: 
we have
$g_b(b)$ is one of $g_b(z_1)/2, g_b(z_2), g_b(z_3)$ except one which is greater than or equal to $g_b(b)$. Moreover when $g_b(z_1)/2$ is higher, we have $g_b(z_1) = 1 + 2g_b(b)$.

Next let us assume that $g_b(z_1)/2 \geq g_b(z_2)$ 
and that $2z_1$ is cut out by a line on $\mathcal Q_b$.
On $\mathcal W$, there is one point $\widetilde{z}$ mapping to $z_1$. The only jets of length $1$ or $2$ at $z_1$ will induce the same jet of  length $2$ at the image of $\widetilde{z}$.
Thus we have
$
g_b(z_2) = g_b(z_3) =  g_b(b),  
g_b(z_1) = 2g_b(b) \text{ or } 2 + 2g_b(b).
$

Next let us assume that $g_b(z_1)/2 \geq g_b(z_3)$ 
and that $z_1 + z_2$ is cut out by a line  on $\mathcal Q_b$.
In this case, we have
$
g_b(z_3) =  g_b(b), g_b(z_2) \geq g_b(b),
2g_b(b)\leq g_b(z_1)  \leq 1+ 2g_b(b), 
$
  and we have 
  \begin{itemize}
 \item $\lceil g_b(z_1)/2 \rceil = g_b(z_2)$; or  
 \item
 $2g_b(z_2) > 1+2g_b(b)$ and $g_b(z_1) = 1+ 2g_b(b)$.
 \end{itemize}

Finally let us assume that $g_b(z_3) \geq g_b(z_1)/2$ and there is a line on $\mathcal Q_b$ which cuts out a length $2$ subscheme $z_2 + z_3$ of $\mathcal Z_b$.
Then $\Spec \, \mathcal O_{\mathcal Z, z_i}$ for $i = 2, 3$ induce jets of twice the length supported at the same point on the fiber $F_1(\mathcal Q)_{c_1}$, and it follows from Proposition~\ref{prop:lemma18} that the maximal length to which two jets coincide is $1$.
In this case, we have
\begin{itemize}
\item
$
g_b(z_1) = 2g_b(b)
$
and
$
2g_b(b) \leq 2g_b(z_2) = 2g_b(z_3) \leq 1 + 2g_b(b);
$
or 
\item
$
g_b(z_1) = 2g_b(b)
$
and
$2g_b(z_2) > 1+2g_b(b)$ and $g_b(z_3) =1 + 2g_b(b)$.
 \end{itemize}

\item Assume that $p_{\mathcal Z \to B}^*(b) = 3z_1 + z_2$ where $z_i$'s are distinct. First assume that there is no line on $\mathcal Q_b$ which cuts out length $2$ subschemes $z_1 + z_2$ or $2z_1$ of $\mathcal Z_b$.
 Then we have $g_b(z_1) = 3g_b(b)$ or $1 + 3g_b(b)$ and $g_b(b) = \min\{g_b(z_1)/3, g_b(z_2)\}$. 
 
 Let us assume that $g_b(z_1)/3\geq g_b(z_2)$ and  that $2z_1$ is cut out by a line  on $\mathcal Q_b$.
In this case, we have
$
g_b(z_2) = g_b(b),
g_b(z_1) = 3g_b(b) \text{ or } 2+3g_b(b).
$

Finally assume that $z_1, z_2$ are cut out by a line on $\mathcal Q_b$.
Then $\Spec \, \mathcal O_{\mathcal Z, z_i}$ for $i = 1, 2$ induce jets supported at the same point on the fiber $F_1(\mathcal Q)_{c_1}$, and the maximal length for which two jets coincide is $1$.
In this case,  we have
$
g_b(z_1) =  \min\{ 3g_b(z_2), 1+3g_b(b)\}.
$

\item Assume that $p_{\mathcal Z \to B}^*(b) = 2z_1 +2z_2$ where $z_i$'s are distinct.
We may assume that 
\[
g_b(z_1) \geq g_b(z_2).
\]
First assume that even after relabeling, there is no line on $\mathcal Q_b$ which cuts out length $2$ subschemes $z_1 + z_2$ or $2z_1$ of $\mathcal Z_b$.
Then $g_b(c_1)$ is equal to $g_b(z_2)$ and $g_b(z_1)$ may be greater than or equal to $2g_b(b)$. Moreover when it is higher, the difference is $1$.

Next assume that $z_1, z_2$ are cut out by a line on $\mathcal Q_b$.
In this case, 
we have $2g_b(b) \leq g_b(z_1)  = g_b(z_2)\leq 1 + 2g_b(b)$.

Next assume that $2z_1$ is cut out by a line on $\mathcal Q_b$.
$
g_b(z_2) = 2g_b(b), g_b(z_1) = 2g_b(b)
$ 
or $2 + 2g_b(b)$.

\item Assume that $p_{\mathcal Z \to B}^*(b) = 4z_1$.
First assume that there is no line on $\mathcal Q_b$ which cuts out a length $2$ subscheme $2z_1$ of $\mathcal Z_b$.
Then we have $4g_b(b)\leq g_b(z_1) \leq 1 + 4g_b(b)$

Next assume that $2z_1$ is cut out by a line on $\mathcal Q_b$. 
Then we have
$
g_b(z_1) = 4g_b(b)
$
or $2 + 4g_b(b)$.

\end{itemize}

\end{class}

\begin{defi}
Now we fix a saturated element $x \in \mathfrak H^\circ(\overline{k})$, and we consider the corresponding function $g$. Let $b \in B(\overline{k})$. 
We fix $c_1 \in C(\overline{k})$ above $b$, and we define the multi-valued function
\[
f_{b, c_1} : \mathbb N\cup \{ \infty\} \to \mathfrak Q_b\cap (B \sqcup \mathcal Z \sqcup C \sqcup\{\hat{1}\})(\overline{k})
\] in the following way:
\begin{itemize}
\item Assume that $p_{\mathcal Z \to B}^*(b) = z_1 + z_2 + z_3 + z_4$ where $z_i$'s are distinct.
We may assume that
$
g_b(z_1) \geq g_b(z_2) \geq g_b(z_3) \geq g_b(z_4).
$
Suppose that $p_{C \to B}^*(b) = c_1 + c_2$.
Then we define $f_{b, c_1}$ by
\[
f_{b, c_1}(n) = 
\begin{cases}
b & \text{ if $0 \leq n \leq \min\{g_b(c_1), g_b(c_2)\}$}\\
c_1 & \text{ if $\min\{g_b(c_1), g_b(c_2)\}< n \leq g_b(c_1)$}\\
z_1+z_2 & \text{ if $g_b(c_1)+ 1\leq n \leq g_b(z_2)$}\\
z_1 & \text{if $\max\{g_b(c_1), g_b(z_2) \}< n \leq g_b(z_1)$}\\
\hat{1} & \text{otherwise}.
\end{cases}
\] 
Next suppose that $p_{C \to B}^*(b) = 2c_1$.
Then we define $f_{b, c_1}$ by
\[
f_{b, c_1}(n) = 
\begin{cases}
b & \text{ if $0 \leq n \leq 2g_b(b)$}\\
 c_1 & \text{ if $2g_b(b)  < n \leq g_b(c_1)$}\\
z_1+z_2 & \text{ if $g_b(c_1)+ 1\leq n \leq 2g_b(z_2)$}\\
z_1 & \text{if $\max\{g_b(c_1), 2g_b(z_2)\} < n \leq 2g_b(z_1)$}\\
\hat{1} & \text{otherwise}.
\end{cases}
\] 

\item Assume that $p_{\mathcal Z \to B}^*(b) = 2z_1 + z_2 + z_3$ where $z_i$'s are distinct.
We may assume that $g(z_2)\geq g(z_3)$. 
Suppose that $p_{C \to B}^*(b) = c_1 + c_2$.
Assume that $g_b(z_1)/2\geq g_b(z_3)$.
Then we define $f_{b, c_1}$ by
\[
f_{b, c_1}(n) = 
\begin{cases}
b & \text{ if $0 \leq n \leq \min\{g_b(c_1), g_b(c_2)\}$}\\
c_1 & \text{ if $\min\{g_b(c_1), g_b(c_2)\} < n \leq g_b(c_1)$}\\
z_1+z_2 & \text{if $g_b(c_1) < n \leq \min\{\lceil g_b(z_1)/2 \rceil, g_b(z_2)\}$.}\\
z_1 & \text{$\max\{g_b(z_2), g_b(c_1)\} < n \leq \lceil g_b(z_1)/2\rceil$.}\\
z_2 & \text{$\max\{\lceil g_b(z_1)/2 \rceil, g_b(c_1)\} < n \leq g_b(z_2)$.}\\
\hat{1} & \text{otherwise}.
\end{cases}
\]
Finally we assume that $g_b(z_3) \geq g_b(z_1)/2$.
Then we define $f_{b, c_1}$ by
\[
f_{b, c_1}(n) = 
\begin{cases}
b & \text{ if $0 \leq n \leq \min\{g_b(c_1), g_b(c_2)\}$}\\
c_1 & \text{ if $\min\{g_b(c_1), g_b(c_2)\} < n \leq g_b(c_1)$}\\
z_2+z_3 & \text{$g_b(c_1) < n \leq g_b(z_3)$.}\\
z_2 & \text{$\max\{g_b(z_3), g_b(c_1)\} < n \leq g_b(z_2)$.}\\
\hat{1} & \text{otherwise}.
\end{cases}
\]
Next assume that $p_{C \to B}^*(b) = 2c_1$.
Assume that $g_b(z_1)/2\geq g_b(z_3)$.
Then we define $f_{b, c_1}$ by
\[
f_{b, c_1}(n) = 
\begin{cases}
b & \text{ if $0 \leq n \leq 2g_b(b) $}\\
 c_1 & \text{ if $2g_b(b)  < n \leq g_b(c_1)$}\\
z_1+z_2 & \text{$g_b(c_1) < n \leq \min\{2 \lceil g_b(z_1)/2\rceil, 2g_b(z_2)\}$.}\\
z_1 & \text{$\max\{ 2g_b(z_2), g_b(c_1)\} < n \leq g_b(z_1)$.}\\
z_2 & \text{$\max\{ 2 \lceil g_b(z_1)/2\rceil, g_b(c_1)\} < n \leq 2g_b(z_2)$.}\\
\hat{1} & \text{otherwise}.
\end{cases}
\]
Finally we assume that $g_b(z_3) \geq g_b(z_1)/2$.
Then we define $f_{b, c_1}$  by
\[
f_{b, c_1}(n) = 
\begin{cases}
b & \text{ if $0 \leq n \leq 2g_b(b) $}\\
 c_1 & \text{ if $2g_b(b) < n \leq g_b(c_1)$}\\
z_2+z_3 & \text{$g_b(c_1) < n \leq 2g_b(z_3)$.}\\
z_2 & \text{$2g_b(z_3) < n \leq 2g_b(z_2)$.}\\
\hat{1} & \text{otherwise}.
\end{cases}
\]

\item Assume that $p_{\mathcal Z \to B}^*(b) = 3z_1 + z_2$ where $z_i$'s are distinct.
Suppose that $p_{C \to B}^*(b) = c_1 + c_2$.
Then we define $f_{b, c_1}$ by
\[
f_{b, c_1}(n) = 
\begin{cases}
b & \text{ if $0 \leq n \leq \min\{g_b(c_1), g_b(c_2)\}$}\\
c_1 & \text{ if $\min\{g_b(c_1), g_b(c_2)\} < n \leq g_b(c_1)$}\\
z_1+z_2 & \text{if $g_b(c_1) < n \leq \min\{\lceil g_b(z_1)/3 \rceil, g_b(z_2)\}$}\\
z_1 & \text{if $\max\{g_b(z_2), g_b(c_1)\} < n \leq \lceil g_b(z_1)/3 \rceil$}\\
z_2 & \text{if $\max\{\lceil g_b(z_1)/3 \rceil, g_b(c_1)\} < n \leq g_b(z_2)$}\\
\hat{1} & \text{otherwise}.
\end{cases}
\]
Next assume that $p_{C \to B}^*(b) = 2c_1$.
Then we define $f_{b, c_1}$ by
\[
f_{b, c_1}(n) = 
\begin{cases}
b & \text{ if $0 \leq n \leq 2g_b(b) $}\\
 c_1 & \text{ if $2g_b(b)  < n \leq g_b(c_1)$}\\
z_1+z_2 & \text{if $g_b(c_1) < n \leq \min\{2\lceil g_b(z_1)/3 \rceil, 2g_b(z_2)\}$}\\
z_1 & \text{if $\max\{2g_b(z_2), g_b(c_1)\} < n \leq 2\lceil g_b(z_1)/3 \rceil$}\\
z_2 & \text{if $\max\{2\lceil g_b(z_1)/3 \rceil, g_b(c_1)\} < n \leq 2g_b(z_2)$}\\
\hat{1} & \text{otherwise}.
\end{cases}
\]

\item When $p_{\mathcal Z \to B}^*(b) = 2z_1 +2z_2$ where $z_i$'s are distinct.
Suppose that $p_{C \to B}^*(b) = c_1 + c_2$.
Then we define $f_{b, c_1}$ by
\[
f_{b, c_1}(n) = 
\begin{cases}
b & \text{ if $0 \leq n \leq \min\{g_b(c_1), g_b(c_2)\}$}\\
 c_1 & \text{ if $\min\{g_b(c_1), g_b(c_2)\} < n \leq g_b(c_1)$}\\
z_1+z_2 & \text{$g_b(c_1) < n \leq \min\{\lceil g_b(z_1)/2\rceil, \lceil g_b(z_2)/2\rceil\}$.}\\
z_1 & \text{$\max\{ \lceil g_b(z_2)/2\rceil, g_b(c_1)\} < n \leq \lceil g_b(z_1)/2\rceil$.}\\
z_2 & \text{$\max\{ \lceil g_b(z_1)/2\rceil, g_b(c_1)\} < n \leq \lceil g_b(z_2)/2\rceil$.}\\
\hat{1} & \text{otherwise}.
\end{cases}
\]


Next suppose that $p_{C \to B}^*(b) = 2c_1$.
Then we define $f_{b, c_1}$ by
\[
f_{b, c_1}(n) = 
\begin{cases}
b & \text{ if $0 \leq n \leq 2g_b(b)$}\\
 c_1 & \text{ if $2g_b(b) < n \leq g_b(c_1)$}\\
z_1+z_2 & \text{$g_b(c_1) < n \leq \min\{2\lceil g_b(z_1)/2\rceil, 2\lceil g_b(z_2)/2\rceil\}$.}\\
z_1 & \text{$\max\{ 2\lceil g_b(z_2)/2\rceil, g_b(c_1)\} < n \leq 2\lceil g_b(z_1)/2\rceil$.}\\
z_2 & \text{$\max\{ 2\lceil g_b(z_1)/2\rceil, g_b(c_1)\} < n \leq 2\lceil g_b(z_1)/2\rceil$.}\\
\hat{1} & \text{otherwise}.
\end{cases}
\]


\item Assume that $p_{\mathcal Z \to B}^*(b) = 4z_1$. Let us suppose that $p_{C \to B}^*(b) = c_1 + c_2$ and we may also assume that $g_b(c_1) = g_b(b)$. 
Then we define $f_{b, c_1}$ by
\[
f_{b, c_1}(n) = 
\begin{cases}
b & \text{ if $0 \leq n \leq \min\{g_b(c_1), g_b(c_2)\}$}\\
c_1 & \text{ if $\min\{g_b(c_1), g_b(c_2)\} < n \leq g_b(c_1)$}\\
z_1 & \text{if $g_b(c_1) < n \leq \lceil g_b(z_1)/4\rceil$}\\
\hat{1} & \text{otherwise}.
\end{cases}
\]
Let us suppose that $p_{C \to B}^*(b) = 2c_1$.
Then we define $f_{b, c_1}$ by
\[
f_{b, c_1}(n) = 
\begin{cases}
b & \text{ if $0 \leq n \leq 2g_b(b) $}\\
 c_1 & \text{ if $2g_b(b)  < n \leq g_b(c_1)$}\\
z_1 & \text{if $g_b(c_1) < n \leq 2\lceil g_b(z_1)/4\rceil$}\\
\hat{1} & \text{otherwise}.
\end{cases}
\]

\end{itemize}
When $p_{C \to B}^*b = c_1 + c_2$, we define $f_b= (f_{b, c_1}, f_{b, c_2})$. When $p_{C \to B}^*b = 2c_1$, we define $f_b = f_{b, c_1}$. 
We denote the set of these pairs by $\mathrm{ch}(\mathfrak Q_b(\overline{k}))$.
If we denote the subset of saturated elements supported over $b$ by $\mathfrak H^\circ_b(\overline{k})$, then the set $\mathrm{ch}(\mathfrak Q_b(\overline{k}))$ is identified with $\mathfrak H^\circ_b(\overline{k})$.

\end{defi}

Next, we define the saturation function:
\begin{defi} The saturation function is the map
\[
\mathrm{sat} : \mathfrak H(\overline{k}) \to \mathfrak H^\circ(\overline{k}),
\]
which is the left adjoint to the inclusion functor $\mathfrak H^\circ(\overline{k}) \hookrightarrow \mathfrak H(\overline{k})$. In other words, for any $x \in \mathfrak H(\overline{k})$, $\mathrm{sat}(x)$ is defined as the smallest saturated element greater than or equal to $x$. 
\end{defi}

\begin{prop}
\label{prop:saturation}
The saturation function $\mathrm{sat}$ exists. Moreover, it is Galois equivariant so that if $x$ is a $k$-rational point on $\mathfrak H(\overline{k})$, then $\mathrm{sat}(x)$ is also a $k$-rational point on $\mathfrak H^\circ(\overline{k})$.
\end{prop}

\begin{proof}
We fix $x = (x_B, x_{\mathcal Z}, x_C, \hat{1}) \in \mathfrak H(\overline{k})$.
First, note that there exists $x_C' \geq x_C$ such that
\[
\iota_\diamond(p_{\mathcal Z}^*x_{\mathcal Z} - p_{\mathcal Z}^*x_{\mathcal Z}\wedge p_C^*x_C')
\] 
induces admissible jets over $C$. Let $x_C', x_C'' \geq x_C$ be two such divisors.
Then
\[
\iota_\diamond(p_{\mathcal Z}^*x_{\mathcal Z} - p_{\mathcal Z}^*x_{\mathcal Z}\wedge p_C^*x_C'\wedge p_C^*x_C'')
\] 
also induces admissible jets over $C$. So let $x_C^*\geq x_C$ be the smallest element such that 
\[
\iota_\diamond(p_{\mathcal Z}^*x_{\mathcal Z} - p_{\mathcal Z}^*x_{\mathcal Z}\wedge p_C^*x_C^*)
\] 
induces admissible jets over $C$. Next let $x_B^*\geq x_B$ be the maximal element such that
\[
p_{C \to B}^*x_B^* \leq x_C^*.
\]
Then we have
\[
\iota_\diamond((p_{\mathcal Z}^*x_{\mathcal Z}\vee p_B^*x_B^*) - (p_{\mathcal Z}^*x_{\mathcal Z}\vee p_B^*x_B^*)\wedge p_C^*x_C^*) = \iota_\diamond(p_{\mathcal Z}^*x_{\mathcal Z} - p_{\mathcal Z}^*x_{\mathcal Z}\wedge p_C^*x_C^*).
\]
Let $x_{\mathcal Z}' = x_{\mathcal Z}\vee p_{\mathcal Z \to B}^*x_B^*$.
Then let $x_{\mathcal Z}^*$ be the maximal element such that 
\[
\iota^*\iota_\diamond(p_{\mathcal Z}^*x_{\mathcal Z}' - p_{\mathcal Z}^*x_{\mathcal Z}' \wedge p_C^*x_C^*) \geq p_{\mathcal Z}^*x_{\mathcal Z}^*- p_{\mathcal Z}^*x_{\mathcal Z}^* \wedge p_C^*x_C^* \geq p_{\mathcal Z}^*x_{\mathcal Z}' - p_{\mathcal Z}^*x_{\mathcal Z}' \wedge p_C^*x_C^*.
\]
Then $\mathrm{sat}(x)$ is given by
\[
\mathrm{sat}(x) =  (x_B^*, x_{\mathcal Z}^*, x_C^*, \hat{1}).
\]
It is clear from the construction that this is saturated and Galois equivariant.

Next we prove that $\mathrm{sat}$ is the left adjoint functor.
Let $y = (y_B, y_{\mathcal Z}, y_C, \hat{1}) \geq x$ be a saturated element.
Since 
\[
p_{\mathcal Z}^*x_{\mathcal Z} - p_{\mathcal Z}^*x_{\mathcal Z}\wedge p_C^*y_C
\]
is a subscheme of 
\[
p_{\mathcal Z}^*y_{\mathcal Z} - p_{\mathcal Z}^*y_{\mathcal Z}\wedge p_C^*y_C,
\]
the subscheme
\[
\iota_{\diamond}(p_{\mathcal Z}^*x_{\mathcal Z} - p_{\mathcal Z}^*x_{\mathcal Z}\wedge p_C^*y_C)
\]
induces admissible jets over $C$. This implies that $x_C^* \leq y_C$.
Then it is clear that $x_B^*\leq y_B$ and $x_{\mathcal Z}' \leq y_{\mathcal Z}$.

We claim that $x_{\mathcal Z}^* \leq y_{\mathcal Z}$.
Once we prove this, we have $\mathrm{sat}(x) \leq y$, hence our assertion follows.
We have the inclusion
\[
p_{\mathcal Z}^*x_{\mathcal Z}^* - p_{\mathcal Z}^*x_{\mathcal Z}^*\wedge p_C^*x_C^* \subset \iota^*\iota_\diamond(p_{\mathcal Z}^*x_{\mathcal Z}' - p_{\mathcal Z}^*x_{\mathcal Z}'\wedge p_C^*x_C^*).
\]
Since we have
\[
\iota_\diamond(p_{\mathcal Z}^*x_{\mathcal Z}^* - p_{\mathcal Z}^*x_{\mathcal Z}^*\wedge p_C^*x_C^* )= \iota_\diamond(p_{\mathcal Z}^*x_{\mathcal Z}' - p_{\mathcal Z}^*x_{\mathcal Z}'\wedge p_C^*x^*_C)
\] 
and $x_C^* \leq y_C$, we conclude that
\[
\iota_\diamond(p_{\mathcal Z}^*x_{\mathcal Z}^* - p_{\mathcal Z}^*x_{\mathcal Z}^*\wedge p_C^*y_C )= \iota_\diamond(p_{\mathcal Z}^*x_{\mathcal Z}' - p_{\mathcal Z}^*x_{\mathcal Z}'\wedge p_C^*y_C).
\]
Indeed, we check this by going through Classification~\ref{class2} and \ref{class3}.
This implies that 
\[
p_{\mathcal Z}^*x_{\mathcal Z}^* - p_{\mathcal Z}^*x_{\mathcal Z}^*\wedge p_C^*y_C \leq \iota^*\iota_\diamond(p_{\mathcal Z}^*y_{\mathcal Z} - p_{\mathcal Z}^*y_{\mathcal Z}\wedge p_C^*y_C).
\]
Hence the maximality of $y_{\mathcal Z}$ shows that $y_{\mathcal Z} \geq x_{\mathcal Z}^*$. Thus our claim follows.

\end{proof}

Next we explain our notation to denote chains:
\begin{notation}
\label{notation:chains}
Let $b \in B(\overline{k})$ and for simplicity we assume that $p_{\mathcal Z \to B}^*(b) = z_1 + z_2 + z_3 + z_4$ and $p_{C \to B}^*(b) = c_1 + c_2$. Let $g_b$ be a saturated element and $f_b = (f_{b, c_1}, f_{b, c_2})$ be the corresponding pair of chains. Without loss of generality, we may assume that
\[
g_{b}(z_1) \geq g_b(z_2) \geq  g_b(z_3) = g_b(z_4) = g_b(b).
\]
Let $m_b = g_b(b)$, $m_{c_1} = g_{b}(c_1) - g_b(b)$, $m_{z_1, z_2} = \max\{0, g_b(z_2) - g_b(c_1)\}$, and $m_{z_1} = g_b(z_1) - g_b(z_2)$.
Then the chain $f_{b, c_1}$ is displayed as
\begin{align*}
&f_{b, c_1} : (b \leq \cdots \leq b \leq c_1 \leq \cdots \leq c_1 \leq z_1+z_2 \leq \cdots \leq z_1+z_2  \leq z_1 \leq \cdots \leq z_1 \leq \hat{1} \leq \hat{1} \leq \cdots)
\end{align*}
where $b$ appears in $f_{b, c_1}$ $m_b + 1$ times, $c_1$ appears in $f_{b, c_1}$ $m_{c_1}$ times, $(z_1, z_2)$ appears in $f_{b, c_1}$ $m_{z_1, z_2}$ times, and $z_1$ appears in $f_{b, c_1}$ $m_{z_1}$ times. In this situation, the chain $f_{b, c_1}$ is denoted by $m_b[b] +m_{c_1}[c_1] + m_{z_1+z_2}[z_1+z_2] +  m_{z_1}[z_1]$. We employ similar notation in other situations too.
Also we denote the initial element $f_{b, c_1}$ given by
\begin{align*}
&f_{b, c_1} : (b \leq \hat{1} \leq \hat{1} \leq \cdots)
\end{align*}
by $\hat{1}$.
\end{notation}

\subsubsection*{Essential saturated elements}

Next, we introduce essential saturated elements:

\begin{defi}
Let $f_1, f_2 \in \mathrm{ch}(\mathfrak Q_b(\overline{k}))$. As in \cite[Definition 5.6]{DLTT25}, we say $f_1 \leq f_2$ if for any $n \in \mathbb N \cup \{\infty\}$, we have $f_{2, c_i}(n) \leq f_{1, c_i}(n)$. Here we impose the poset structure by imposing $c_i < z_j+ z_{j'} < z_j$. Note that this poset structure on $\mathrm{ch}(\mathfrak Q_b(\overline{k}))$ is the same poset structure which we induce by identifying $\mathrm{ch}(\mathfrak Q_b(\overline{k}))$ with $\mathfrak H^\circ_b(\overline{k})$. (Note that $\mathfrak H^\circ$ is a poscheme.)

Next, we claim that $\mathrm{ch}(\mathfrak Q_b(\overline{k}))$ admits joins. Indeed, let us identify $f_1, f_2 \in \mathrm{ch}(\mathfrak Q_b(\overline{k}))$ with saturated elements $x_1 = (x_{1, B}, x_{1, \mathcal Z}, x_{1, C}), x_2  = (x_{2, B}, x_{2, \mathcal Z}, x_{2, C}) \in \mathfrak H^\circ_b(\overline{k})$. Then the join $f_1\vee f_2$ corresponds to the saturation of the union of $x_1, x_2$ as subschemes, i.e., $$\mathrm{sat}(x_{1, B}\cup x_{2, B} , x_{1, \mathcal Z}\cup x_{2, \mathcal Z}, x_{1, C}\cup x_{2, C}).$$
This definition is the same as the definition of joins in $\mathfrak H^\circ_b(\overline{k})$.
\end{defi}

Here is the definition of essential saturated elements:

\begin{defi}
Let $f_0 \in \mathrm{ch}(\mathfrak Q_b(\overline{k}))$ be a pair of chains. Let $\mathfrak S$ be the set of $f_s \in \mathrm{ch}(\mathfrak Q_b(\overline{k}))$ such that $f_0 \prec f_s$, i.e., $f_0 < f_s$ and there is no element $f' \in \mathrm{ch}(\mathfrak Q_b(\overline{k}))$ satisfying $f_0 < f' < f_s$.
We say the pair $f_0 \leq f$ is essential if $f$ can be obtained from $f_0$ by taking the join of some subset of $\mathfrak S$. 

Let $W \subset \mathfrak H^\circ$ be a locally closed subscheme with the reduced scheme structure. 
We say $(w \leq x) \in (W \leq \mathfrak H^\circ)$ is essential if for any point $b \in \mathrm{Supp}(w\leq x)(\overline{k})$, $f_b^w \leq f_b^x$ is essential.
\end{defi}

Let us classify essential pairs over an algebraically closed field:

\begin{exam}
\label{exam:essentialtypes}
We classify essential pairs based on case by case analysis on ramification behaviors of $\mathcal Z \to B$ and $C \to B$ as well as configurations of lines on $\mathcal Q_b$.
We demonstrate these in a few certain cases. Other cases are similar.

\noindent
{\bf Case 1} Assume that $p_{\mathcal Z \to B}^*(b) = z_1 + z_2 + z_3 + z_4$ and $p_{C \to B}^*(b) = c_1 + c_2$.

There are additional informations we need to consider, i.e., the configuration of lines on $\mathcal Q_b$ which cuts out a length $2$ subscheme of $\mathcal Z_b$. They are classified into $6$ cases as follows:
\begin{itemize}
\item there is no line;
\item there is one line containing $z_1$ and $z_2$;
\item there are two lines such that one contains $z_1$ and $z_2$ and the other contains $z_3$ and $z_4$;
\item there are two lines such that one contains $z_1$ and $z_2$ and the other contains $z_1$ and $z_3$;
\item there are three lines such that the first one contains $z_1$ and $z_2$, the second one contain $z_3$ and $z_4$, and the third line contains $z_1$ and $z_3$; and
\item there are four lines such that the first one contains $z_1$ and $z_2$, the second one contain $z_3$ and $z_4$, the third line contains $z_1$ and $z_3$, and the fourth line contains $z_2$ and $z_4$.
\end{itemize}
We demonstrate the classification of essential saturated pairs in certain cases. 

Suppose that there is no line.
First let us assume that $f_0 = \hat{1}$. Then $f \in \mathrm{ch}(\mathfrak Q_b(\overline{k}))$ such that $f_0 \prec f$ is listed as follows
\[
([z_i], [z_i]) \, (i = 1, \cdots, 4), \quad ([c_1], \hat{1}), \quad (\hat{1}, [c_2]).
\]
Then the list of $f$ such that $f_0 \leq f$ is essential additionally consists of
\[
\hat{1}, \quad ([b], [b]), \quad ([c_1], [z_i]), \quad ([z_i], [c_2]) \, (i = 1, \cdots, 4).
\]

Next let us assume that $f_0 = (m_{z_i}[z_i], m_{z_i}[z_i])$ with $m_{z_i} \geq 1$.
Then $f \in \mathrm{ch}(\mathfrak Q_b(\overline{k}))$ such that $f_0 \prec f$ is listed as follows
\[
((m_{z_i}+1)[z_i], (m_{z_i}+1)[z_i]), \quad ([c_1] + (m_{z_i}-1)[z_i], m_{z_i}[z_i]), \quad (m_{z_i}[z_i], [c_2] + (m_{z_i}-1)[z_i]).
\]
Then the list of $f$ such that $f_0 \leq f$ is essential additionally consists of
\begin{align*}
&\hat{1}, \quad ([b] + (m_{z_i}-1)[z_i], [b]+ (m_{z_i}-1)[z_i]),\\
&([c_1] + m_{z_i}[z_i], (m_{z_i} + 1)[z_i]), \quad ((m_{z_i}+1)[z_i], [c_2]+ m_{z_i}[z_i]), \quad ([b] + m_{z_i}[z_i], [b]+ m_{z_i}[z_i]).
\end{align*}

Next let us analyze in the extreme case.
Suppose that there are four lines.
First let us assume that $f_0 = \hat{1}$. Then $f \in \mathrm{ch}(\mathfrak Q_b(\overline{k}))$ such that $f_0 \prec f$ is listed as follows
\[
([z_i], [z_i]) \, (i = 1, \cdots, 4), \quad ([c_1], \hat{1}), \quad (\hat{1}, [c_2]).
\]
Then the list of $f$ such that $f_0 \leq f$ is essential additionally consists of
\[
\hat{1}, \quad ([b], [b]), \quad ([c_1], [z_1+z_3]), \quad ([c_1], [z_2+z_4]) \quad ([z_1+ z_2], [c_2]), \quad [z_3+ z_4, [c_2]).
\]
Next let us assume that $f_0 = (m_{z_1}[z_1], m_{z_1}[z_1])$ with $m_{z_1} \geq 1$.
Then $f \in \mathrm{ch}(\mathfrak Q_b(\overline{k}))$ such that $f_0 \prec f$ is listed as follows
\begin{align*}
&((m_{z_1}+1)[z_1], (m_{z_1}+1)[z_1]), \quad ([c_1] + (m_{z_1}-1)[z_1], [z_1+ z_3]+ (m_{z_1}-1)[z_1]), \\ &([z_1+ z_2]+(m_{z_1}-1)[z_1], [c_2] + (m_{z_1}-1)[z_1]).
\end{align*}
Then the list of $f$ such that $f_0 \leq f$ is essential additionally consists of
\begin{align*}
&\hat{1}, \quad ([b] + (m_{z_1}-1)[z_1], [b]+ (m_{z_1}-1)[z_1]),\\
&([c_1] + m_{z_1}[z_1], [z_1+ z_3]+m_{z_1}[z_1]), \quad ([z_1+ z_2]+m_{z_1}[z_1], [c_2]+ m_{z_1}[z_1]), \quad ([b] + m_{z_1}[z_1], [b]+ m_{z_1}[z_1]).
\end{align*}

\noindent
{\bf Case 2} Assume that $p_{\mathcal Z \to B}^*(b) = 2z_1 + z_2 + z_3$ and $p_{C \to B}^*(b) = c_1 + c_2$.

First we give the classification of configurations of lines $\mathcal Q_b$ which cuts out a length $2$ subscheme of $\mathcal Z_b$:
\begin{itemize}
\item There is no line;
\item There is one line which cuts out $2z_1$;
\item There is one line which cuts out $z_1 + z_2$;
\item There is one line which cuts out $z_2 + z_3$;
\item There are two lines which cuts out $2z_1$ and $z_1 + z_2$ respectively;
\item There are two lines which cuts out $2z_1$ and $z_2 + z_3$ respectively;
\item There are two lines which cuts out $z_1 + z_2$ and $z_2 + z_3$ respectively;
\item There are two lines which cuts out $z_1 + z_2$ and $z_1 + z_3$ respectively; and
\item There are three lines which cuts out $2z_1$, $z_1 + z_2$, and $z_1 + z_3$.
\end{itemize}

Again we demonstrate the classification of essential saturated pairs in certain cases. 

Suppose that there is no line.
First let us assume that $f_0 = \hat{1}$. Then $f \in \mathrm{ch}(\mathfrak Q_b(\overline{k}))$ such that $f_0 \prec f$ is listed as follows
\[
([z_i],[z_i]) \, (i = 1, \cdots, 3), \quad ([c_1], \hat{1}), \quad (\hat{1}, [c_2]).
\]
Then the list of $f$ such that $f_0 \leq f$ is essential additionally consists of
\[
\hat{1}, \quad ([b], [b]),  \quad ([c_1], [z_i]) \, (i = 1, \cdots, 3), \quad ([z_i], [c_2]) \, (i = 1, \cdots, 3).
\]

Next let us assume that $f_0 = (m_{z_i}[z_i], m_{z_i}[z_i])$ with $m_{z_i} \geq 1$ for $i = 2, 3$.
Then $f \in \mathrm{ch}(\mathfrak Q_b(\overline{k}))$ such that $f_0 \prec f$ is listed as follows
\[
((m_{z_i}+1)[z_i], (m_{z_i}+1)[z_i]), \quad  ([c_1], m_{z_i}[z_i]), \quad  (m_{z_i}[z_i], [c_2]).
\]
Then the list of $f$ such that $f_0 \leq f$ is essential additionally consists of
\begin{align*}
&\hat{1}, \quad ([b] + (m_{z_i}-1)[z_i], [b]+ (m_{z_i}-1)[z_i]),\\
&([c_1]+ m_{z_i}[z_i], (m_{z_i}+1)[z_i]), \quad ((m_{z_i}+1)[z_i], [c_2]+ m_{z_i}[z_i]), \quad ([b] + m_{z_i}[z_i], [b]+ m_{z_i}[z_i]).
\end{align*}

Next let us assume that $f_0 = ([z_1], [z_1])$.
Then $f \in \mathrm{ch}(\mathfrak Q_b(\overline{k}))$ such that $f_0 \prec f$ is listed as follows
\[
  ([c_1], [z_1]), \quad ([z_1], [c_2]).
\]
Then the list of $f$ such that $f_0 \leq f$ is essential additionally consists of
\[
\hat{1}, \quad ([b], [b]).
\]

Suppose that there is one line which cuts out $2z_1$.
In this case, the poset of essential saturated pairs becomes completely the same as the case of no lines. Other cases are similar so we omit them.

\noindent
{\bf Case 3} Assume that $p_{\mathcal Z \to B}^*(b) = z_1 + z_2 + z_3 + z_4$ and $p_{C \to B}^*(b) = 2c_1$.
The configurations of lines which cut out a length $2$ subscheme of $\mathcal Z_b$ is the following:
\begin{itemize}
\item There is no line;
\item There is one line which cuts out $z_1 + z_2$; and
\item There are two lines which cut out $z_1 + z_2$ and $z_3 + z_4$ respectively.
\end{itemize}

Suppose there is no line.
First let us assume that $f_0 = \hat{1}$. Then $f \in \mathrm{ch}(\mathfrak Q_b(\overline{k}))$ such that $f_0 \prec f$ is listed as follows
\[
2[z_i] \, (i = 1, \cdots, 4), \quad  [c_1].
\]
Then the list of $f$ such that $f_0 \leq f$ is essential additionally consists of
\[
\hat{1}, \quad 2[b], \quad , [c_1] + [z_i] \, (i = 1, \cdots, 4).
\]

Next let us assume that $f_0 = 2m_{z_i}[z_i]$ with $m_{z_i} \geq 1$.
Then $f \in \mathrm{ch}(\mathfrak Q_b(\overline{k}))$ such that $f_0 \prec f$ is listed as follows
\[
(2m_{z_i}+2)[z_i], \quad  [c_1] + (2m_{z_i}-1)[z_i].
\]
Then the list of $f$ such that $f_0 \leq f$ is essential additionally consists of
\[
\hat{1},\quad [c_1] + (2m_{z_i}+1)[z_i].
\]

Next we consider the extreme case, i.e., there are two lines.
First let us assume that $f_0 = \hat{1}$. Then $f \in \mathrm{ch}(\mathfrak Q_b(\overline{k}))$ such that $f_0 \prec f$ is listed as follows
\[
[z_1 + z_2] + [z_1], \quad, [z_1 + z_2] + [z_2], \quad [z_3 + z_4] + [z_3], \quad, [z_2 + z_4] + [z_4] \quad  [c_1].
\]
Then the list of $f$ such that $f_0 \leq f$ is essential additionally consists of
\[
\hat{1}, \quad 2[b], \quad [c_1] + [z_1 + z_2], \quad [c_1] + [z_3 + z_4].
\]

Next let us assume that $f_0 = [z_1 + z_2] + (2m_{z_1}-1)[z_1]$ with $m_{z_1} \geq 1$.
Then $f \in \mathrm{ch}(\mathfrak Q_b(\overline{k}))$ such that $f_0 \prec f$ is listed as follows
\[
[z_1 + z_2] +(2m_{z_i}+1)[z_i], \quad  [c_1] + [z_1 + z_2] +2m_{z_i}[z_i].
\]
Then the list of $f$ such that $f_0 \leq f$ is essential additionally consists of
\[
\hat{1},\quad [c_1] + [z_1 + z_2] + 2m_{z_i}[z_i].
\]

\noindent
{\bf Case 4} Assume that $p_{\mathcal Z \to B}^*(b) = 2z_1 + z_2 + z_3$ and $p_{C \to B}^*(b) = 2c_1$.

Here is the classification of configurations of lines which cut out a length $2$ subscheme:
\begin{itemize}
\item There is no line;
\item There is one line which cuts out $2z_1$;
\item There is one line which cuts out $z_1 + z_2$;
\item There is one line which cuts out $z_2 + z_3$;
\item There are two lines which cut out $2z_1$ and $z_2 + z_3$ respectively.
\end{itemize}

We only consider the case that there is no line.
First let us assume that $f_0 = \hat{1}$. Then $f \in \mathrm{ch}(\mathfrak Q_b(\overline{k}))$ such that $f_0 \prec f$ is listed as follows
\[
2[z_1], \quad 2[z_i] \, (i = 2, 3).
\]
Then the list of $f$ such that $f_0 \leq f$ is essential additionally consists of
\[
\hat{1}, \quad 2[b],  \quad [c_1] + [z_i] \, (i = 1, 2, 3).
\]

Next let us assume that $f_0 = 2m_{z_i}[z_i]$ with $m_{z_i} \geq 1$ for $i = 2, 3$.
Then $f \in \mathrm{ch}(\mathfrak Q_b(\overline{k}))$ such that $f_0 \prec f$ is listed as follows
\[
2(m_{z_i}+1)[z_i], \quad  [c_1] + (2m_{z_i}-1)[z_i].
\]
Then the list of $f$ such that $f_0 \leq f$ is essential additionally consists of
\[
\hat{1}, \quad [c_1] + (2m_{z_i}+1)[z_i].
\]

Next let us assume that $f_0 = 2[z_1]$.
Then $f \in \mathrm{ch}(\mathfrak Q_b(\overline{k}))$ such that $f_0 \prec f$ is listed as follows
\[
 [c_1] + [z_1].
\]
Then the list of $f$ such that $f_0 \leq f$ is essential additionally consists of
\[
\hat{1}.
\]

\end{exam}

\subsubsection*{Combinatorial functions}

Next, we introduce combinatorial functions:

\begin{defi}
Suppose that we have a function $\zeta : \mathfrak Q(\overline{k}) \to \mathbb N$ such that $\zeta(\hat{1}) = 0$. Furthermore, assume that for each $b \in B(\overline{k})$, we extend this function as
\[
\zeta : \mathrm{ch}(\mathfrak Q_b(\overline{k})) \to \mathbb N,
\]
such that $\zeta(\hat{1})=0$.
Let $x \in \mathfrak H^\circ(\overline{k})$ be a saturated element and we consider the corresponding function
\[
B(\overline{k}) \to \bigsqcup_{b} \mathrm{ch}(\mathfrak Q_b(\overline{k})), \quad b \mapsto f_b.
\]
We define $\zeta(x)$ by
\[
\zeta(x) = \sum_{b \in B(\overline{k})} \zeta(f_b).
\]
Again this is well-defined because for all but finitely many $b$, $f_b = \hat{1}$ and $\zeta(\hat{1}) = 0$.
We further extend $\zeta$ to $\zeta: \mathfrak H(\overline{k}) \to \mathbb N$ by composing with $\mathrm{sat} : \mathfrak H(\overline{k}) \to \mathfrak H^\circ(\overline{k})$.

Finally we extend to pairs $(w\leq x)$ by $\zeta(w \leq x) = \zeta(x) - \zeta(w)$.
\end{defi}

\begin{exam}

We define 
\[
m_{B} :  \mathrm{ch}(\widetilde{\mathfrak Q}_b(\overline{k})) \to \mathbb N
\]
in the following way: when $p_{C \to B}^*b = c_1 + c_2$, let $m_B(f_b)$ be $2$ times the number which counts how many times $b$ appears in $f_{b, c_1}$. Note that this does not depend on the choice of $c_1$. When $p_{C \to B}^*b = 2c_1$, let $m_B(f_b)$ be the number which counts how many times $b$ appears in $f_{b, c_1}$.

Next we define 
\[
m_{C} :  \mathrm{ch}(\widetilde{\mathfrak Q}_b(\overline{k})) \to \mathbb N
\]
in the following way: when $p_{C \to B}^*b = c_1 + c_2$, let $m_C(f_b)$ be 
the sum of the number which counts how many times $c_1$ appears in $f_{b, c_1}$ and the number which counts how many times $c_2$ appears in $f_{b, c_2}$.
When $p_{C \to B}^*b = 2c_1$, let $m_C(f_b)$ be the number which counts how many times $c_1$ appears in $f_{b, c_1}$.

Finally we define 
\[
m_{\mathcal Z} :  \mathrm{ch}(\widetilde{\mathfrak Q}_b(\overline{k})) \to \mathbb N
\]
in the following way: when $p_{C \to B}^*b = c_1 + c_2$, let $m_{\mathcal Z}(f_b)$ be 
the sum of the number which counts how many times $z$ (including the form of $z+ z'$ counted with the multiplicity $1$) appears in $f_{b, c_1}$ and the number which counts how many times $z$ appears in $f_{b, c_2}$.
When $p_{C \to B}^*b = 2c_1$, let $m_{\mathcal Z}(f_b)$ be the number which counts how many times $z$ appears in $f_{b, c_1}$.

We extend these functions as
\[
m_{B}, m_{C}, m_{\mathcal Z} : \mathfrak H^\circ(\overline{k}) \to \mathbb N.
\]
Note that when $x = (x_B, x_{\mathcal Z}, x_C)$ is a saturated element, $m_{\mathcal Z}(x)$ is equal to the length of
\[
\iota_{\diamond}(p_{\mathcal Z}^*x_{\mathcal Z}-p_{\mathcal Z}^*x_{\mathcal Z}\wedge p_C^*x_C).
\]
\end{exam}

\begin{exam}
For any $b \in B(\overline{k})$ and $f \in \mathrm{ch}(\mathfrak Q_b(\overline{k}))$, we define
\[
\mathrm{rank}(f)
\]
to be the length of the maximal chain from $\hat{1}$ to $f$. Note that this definition is well-defined and independent of a choice of maximal chains. We extend it to 
\[
\mathrm{rank} : \mathfrak H^\circ(\overline{k}) \to \mathbb N.
\]
For any $x \in \mathfrak H^\circ(\overline{k})$, the value $\mathrm{rank}(x)$ is equal to the cohomological dimension of the nerve $N([\hat{1}, x] \cap \mathfrak H^\circ(\overline{k}))$.
Also note that we have
\[
 \mathrm{rank}(x) = \frac{3}{2}m_B(x) + \frac{1}{2}m_{\mathcal Z}(x) + \frac{3}{2}m_C(x).
\]

\end{exam}

\begin{exam}
For any $x \in \mathfrak H^\circ(\overline{k})$, $|\mathrm{Supp}(x)|$ is the number of $b \in B(\overline{k})$ such that $f_b$ is non-trivial.
\end{exam}

\begin{exam}
\label{exam:codimension}
We define
\[
\gamma :\mathfrak H^\circ(\overline{k}) \to \mathbb N
\]
by
\[
\gamma(f) = 2m_B(f) + m_{\mathcal Z}(f) + 2m_C(f).
\]
This
is interpreted as the expected codimension $\gamma(x)$ of the incidence condition imposed by $x \in \mathfrak H^\circ(\overline{k})$.

\end{exam}

\subsubsection*{Combinatorial stratifications}

Next, we introduce the stratification of $\mathfrak H$.
Let $$B = \bigsqcup_{(\mathsf a, \mathsf b)} B_{\mathsf a, \mathsf b},$$ where $\mathsf a = (1, 1, 1, 1), (2, 1, 1), (3, 1), (2, 2), (4)$ and $\mathsf b = (1, 1), (2)$
and $B_{\mathsf a, \mathsf b}$ is the set of points on $B$ such that the ramification behavior of $\mathcal Z \to B$ over those points is given by $\mathsf a$ and the ramification behavior of $C \to B$ over those points is given by $\mathsf b$. 

For each $\mathsf a, \mathsf b$, the poset  $\mathfrak Q_b(\overline{k})$ with the stratification by $B\sqcup \mathcal Z \sqcup C \sqcup \{\hat{1}\}$ does not depend on $b \in B_{\mathsf a, \mathsf b}(\overline{k})$, and we denote this abstract poset with the stratification by $\mathfrak Q_{\mathsf a, \mathsf b}$. When $\mathsf a = (2, 1, 1)$ or $(3, 1)$, the unique point with a multiplicity is distinguished.

We add an additional piece of information to this abstract poset. Let $b \in B(\overline{k})$ and we consider the configuration $\Gamma_b$ of lines on $\mathcal Q_{b}$ which cut out length $2$ subschemes of $\mathcal Z_b$.
The configuration $\Gamma_b$ encodes which a length $2$ subscheme of $\mathcal Z_b$ each line contains.
Suppose that we have an automorphism $t \in \mathrm{Aut}(\mathfrak Q_b(\overline{k}))$ of the poset $\mathfrak Q_b(\overline{k})$ with the stratification. Then $t$ induces a map from one configuration to another, and let $\mathfrak R_{\mathsf a, \mathsf b}$ be the set of isomorphism classes of configurations of lines for $\mathfrak Q_{\mathsf a, \mathsf b}$. The configuration without any line will be denoted by $\emptyset$. Let $\mathsf d \in \mathfrak R_{\mathsf a, \mathsf b}$ and let $B_{\mathsf a, \mathsf b, \mathsf d}(\overline{k})$ be the set of points $b \in B_{\overline{k}}$ such that 
\[
(\mathfrak Q_b(\overline{k}), \Gamma_b) \cong (\mathfrak Q_{\mathsf a, \mathsf b}, \Gamma_{\mathsf d}).
\]
Note that $B_{\mathsf a, \mathsf b, \mathsf d}(\overline{k})$ is Galois invariant so it descent as $B_{\mathsf a, \mathsf b, \mathsf d}$ over $k$.
Let $B^\circ$ be $B_{(1, 1, 1, 1), (1, 1), \emptyset}$ and $(\mathsf a, \mathsf b, \mathsf d) \neq ((1, 1, 1, 1), (1, 1), \emptyset)$, we denote the cardinality of $B_{\mathsf a, \mathsf b, \mathsf d}(\overline{k})$ by $r_{\mathsf a, \mathsf b, \mathsf d}$. We denote $\mathfrak Q_{(1, 1, 1, 1), (1, 1), \emptyset}$ by $\mathfrak Q_{\mathrm{gen}}$. 

\begin{defi}
A combinatorial type $T$ is a multiset
\[
\{g_1, \cdots, g_r\} \sqcup \bigsqcup_{(\mathsf a, \mathsf b, \mathsf d) \neq ((1, 1, 1, 1), (1, 1), \emptyset)} \{g^{\mathsf a, \mathsf b, \mathsf d}_1, \cdots, g^{\mathsf a, \mathsf b, \mathsf d}_{r_{\mathsf a, \mathsf b, \mathsf d}}\} ,
\]
where 
\[
g_i \in \mathrm{Hom}_{(\clubsuit)}(\mathfrak Q_{\mathrm{gen}}, \mathbb N \cup \{\infty\}), \quad g_i^{\mathsf a, \mathsf b, \mathsf d} \in \mathrm{Hom}_{(\clubsuit)}(\mathfrak Q_{\mathsf a, \mathsf b, \mathsf d}, \mathbb N \cup \{\infty\}),
\]
and $g_i$'s are non-trivial. However, $g_i^{\mathsf a, \mathsf b, \mathsf d}$ could be trivial.

Let
\begin{align*}
&T = \{g_1, \cdots, g_r\} \sqcup \bigsqcup_{(\mathsf a, \mathsf b, \mathsf d) \neq ((1, 1, 1, 1), (1, 1), \emptyset} \{g^{\mathsf a, \mathsf b, \mathsf d}_1, \cdots, g^{\mathsf a, \mathsf b, \mathsf d}_{r_{\mathsf a, \mathsf b, \mathsf d}}\},\\
&T' = \{h_1, \cdots, h_s\} \sqcup \bigsqcup_{(\mathsf a, \mathsf b, \mathsf d) \neq ((1, 1, 1, 1), (1, 1), \emptyset)} \{h^{\mathsf a, \mathsf b, \mathsf d}_1, \cdots, h^{\mathsf a, \mathsf b, \mathsf d}_{r_{\mathsf a, \mathsf b, \mathsf d}}\}
\end{align*}
be two combinatorial types.
We say two combinatorial types $T$ and $T'$ are equivalent if $r = s$, and there exist permutations $\sigma \in \mathfrak S_r$, $\sigma^{\mathsf a, \mathsf b, \mathsf d} \in \mathfrak S_{r_{\mathsf a, \mathsf b, \mathsf d}}$, and automorphisms of posets with stratifications and configurations of lines
\[
t_i \in \mathrm{Aut}(\mathfrak Q_{\mathrm{gen}}, \Gamma_\emptyset)\, (i = 1, \cdots, r), \quad t_i^{\mathsf a, \mathsf b, \mathsf d} \in \mathrm{Aut}(\mathfrak Q_{\mathsf a, \mathsf b, \mathsf d}, \Gamma_{\mathsf d}) \, (i = 1, \cdots, r_{\mathsf a, \mathsf b, \mathsf d})
\]
such that we have 
\[
g_{\sigma(i)} \circ t_i = h_i, \quad g^{\mathsf a, \mathsf b, \mathsf d}_{\sigma^{\mathsf a, \mathsf b, \mathsf d}(i)}\circ t_i^{\mathsf a, \mathsf b, \mathsf d} = h^{\mathsf a, \mathsf b, \mathsf d}_i.
\]
When $\mathsf a = (2, 1, 1)$ or $(3, 1)$, we insist that $t_i^{\mathsf a, \mathsf b, \mathsf d}$ preserves the distinguished point. We denote $T \equiv T'$ when two are equivalent.

Let $T$ be the above combinatorial type.
Let $x \in \mathfrak H(\overline{k})$.
Let $b_1, \cdots, b_r$ be the support of $x$ on $B^\circ(\overline{k})$ and we also denote $B_{\mathsf a, \mathsf b, \mathsf d}(\overline{k}) = \{ b_1^{\mathsf a, \mathsf b, \mathsf d}, \cdots, b_{r_{\mathsf a, \mathsf b, \mathsf d}}^{\mathsf a, \mathsf b, \mathsf d}\}$. Let 
\[
g_{b_1}, \cdots, g_{b_r}, \quad g_{b_{1}^{\mathsf a, \mathsf b, \mathsf d}}, \cdots, g_{b_{r_{\mathsf a, \mathsf b, \mathsf d}}^{\mathsf a, \mathsf b, \mathsf d}},
\]
be the functions corresponding to $x$. We say $x$ has combinatorial type $T$ if there are permutations $\sigma \in \mathfrak S_r, \sigma^{\mathsf a, \mathsf b, \mathsf d} \in \mathfrak S_{r_{\mathsf a, \mathsf b, \mathsf d}}$ and isomorphisms of posets with stratifications
\[
t_i : (\mathfrak Q_{b_i}(\overline{k}), \Gamma_{b_i}) \cong (\mathfrak Q_{\mathrm{gen}}, \Gamma_\emptyset) \quad t_i^{\mathsf a, \mathsf b, \mathsf d} : (\mathfrak Q_{b_i}(\overline{k}), \Gamma_{b_i}) \cong (\mathfrak Q_{\mathsf a, \mathsf b, \mathsf d}, \Gamma_{\mathsf d}),
\]
such that we have
\[
g_{b_i} = g_{\sigma(i)}\circ t_i, \quad g_{b_i^{\mathsf a, \mathsf b, \mathsf d}} = g^{\mathsf a, \mathsf b, \mathsf d}_{\sigma^{\mathsf a, \mathsf b, \mathsf d}(i)}\circ t^{\mathsf a, \mathsf b, \mathsf d}_i.
\]
When $\mathsf a = (2, 1, 1)$ or $(3, 1)$, we insist that $t_i^{\mathsf a, \mathsf b, \mathsf d}$ preserves the distinguished point.

We say that a combinatorial type $T$ is saturated if every $g_i$ and $g_i^{\mathsf a, \mathsf b, \mathsf d}$ are saturated in the sense of Definition~\ref{defi:saturated}.
\end{defi}

\begin{defi}
Let $T$ be a combinatorial type. Let $\mathcal N_{T, \overline{k}} \subset \mathfrak H_{\overline{k}}$ be a Zariski locally closed subset such that $\mathcal N_{T, \overline{k}}(\overline{k})$ consists of $x \in \mathfrak H(\overline{k})$ of combinatorial type $T$.
Since this set is Galois-invariant, it descends to $\mathcal N_T \subset \mathfrak H$ defined over $k$.
Note that we have 
\[
\mathcal N_T \cap \mathcal N_{T'} \neq \emptyset \iff \mathcal N_T = \mathcal N_{T'}\iff T \equiv T'.
\]
Thus $\mathcal N_T$ depends only on the equivalence class of $T$.
We have a stratification into locally closed subsets:
\[
\mathfrak H = \bigsqcup_T \mathcal N_T,
\]
where $T$ runs over representatives of equivalence classes of combinatorial types.
\end{defi}

It is clear from the definition that we have
\[
\mathcal N_T(\overline{k})\cap \mathfrak H^\circ(\overline{k}) \neq \emptyset \iff \mathcal N_T(\overline{k}) \subset \mathfrak H^\circ(\overline{k}) \iff \text{$T$ is saturated.}
\]
In particular, we have also a stratification into locally closed subsets:
\[
\mathfrak H^\circ(\overline{k}) = \bigsqcup_{T: \text{saturated}} \mathcal N_T(\overline{k}).
\]
\begin{defi}
Let
\begin{align*}
&T = \{g_1, \cdots, g_r\} \sqcup \bigsqcup_{(\mathsf a, \mathsf b, \mathsf d) \neq ((1, 1, 1, 1), (1, 1), \emptyset)} \{g^{\mathsf a, \mathsf b, \mathsf d}_1, \cdots, g^{\mathsf a, \mathsf b, \mathsf d}_{r_{\mathsf a, \mathsf b, \mathsf d}}\},\\
&T' = \{h_1, \cdots, h_s\} \sqcup \bigsqcup_{(\mathsf a, \mathsf b, \mathsf d) \neq ((1, 1, 1, 1), (1, 1), \emptyset)} \{h^{\mathsf a, \mathsf b, \mathsf d}_1, \cdots, h^{\mathsf a, \mathsf b, \mathsf d}_{r_{\mathsf a, \mathsf b, \mathsf d}}\}
\end{align*}
be two combinatorial types. We say $T \geq T'$ if there exist an injective map $\sigma : \{1, \cdots, s\} \to \{1, \cdots, r\}$, permutations $\sigma^{\mathsf a, \mathsf b, \mathsf d} \in \mathfrak S_{r_{\mathsf a, \mathsf b, \mathsf d}}$, and automorphisms
\[
t_i \in \mathrm{Aut}(\mathfrak Q_{\mathrm{gen}}, \Gamma_\emptyset)\, (i = 1, \cdots, s), \quad t_i^{\mathsf a, \mathsf b, \mathsf d} \in \mathrm{Aut}(\mathfrak Q_{\mathsf a, \mathsf b, \mathsf d}, \Gamma_{\mathsf d}) \, (i = 1, \cdots, r_{\mathsf a, \mathsf b, \mathsf d})
\]
such that we have
\[
g_{\sigma(i)} \circ t_i \geq h_i, \quad g^{\mathsf a, \mathsf b, \mathsf d}_{\sigma^{\mathsf a, \mathsf b, \mathsf d}(i)} \circ t^{\mathsf a, \mathsf b, \mathsf d}_i \geq h^{\mathsf a, \mathsf b, \mathsf d}_i.
\]
Note that this defines a poset relation on the set of equivalence classes of combinatorial types.

We say $T \geq_+ T'$ if there exist a map
\[
\sigma : \{1, \cdots, r\} \to \{1, \cdots, s\} \sqcup \bigsqcup_{\mathsf a, \mathsf b, \mathsf d} \{1^{\mathsf a, \mathsf b, \mathsf d}, \cdots, r_{\mathsf a, \mathsf b, \mathsf d}^{\mathsf a, \mathsf b, \mathsf d}\},
\]
and permutations $\sigma_{\mathsf a, \mathsf b, \mathsf d}\in \mathfrak S_{r_{\mathsf a, \mathsf b, \mathsf d}}$ such that 
\begin{enumerate}
\item the image of $\sigma$ contains $\{1, \cdots, s\}$; 
\item for any $j \in \{1, \cdots, s\}$ and $i \in \sigma^{-1}(j)$, we have an automorphism
\[
t_{i, j} \in \mathrm{Aut}(\mathfrak Q_{\mathrm{gen}}, \Gamma_\emptyset);
\]
\item for $j \in \{1^{\mathsf a, \mathsf b, \mathsf d}, \cdots, r_{\mathsf a, \mathsf b, \mathsf d}^{\mathsf a, \mathsf b, \mathsf d}\}$, we have an automorphism
\[
t_{j}^{\mathsf a, \mathsf b, \mathsf d} \in \mathrm{Aut}(\mathfrak Q_{\mathsf a, \mathsf b, \mathsf d}, \Gamma_{\mathsf d});
\]
\item for $j \in \{1^{\mathsf a, \mathsf b, \mathsf d}, \cdots, r_{\mathsf a, \mathsf b, \mathsf d}^{\mathsf a, \mathsf b, \mathsf d}\}$ and $i \in \sigma^{-1}(j)$, we have a surjective homomorphism of posets with stratifications
\[
t_{i, j}^{\mathsf a, \mathsf b, \mathsf d} : \mathfrak Q_{\mathrm{gen}} \to \mathfrak Q_{\mathsf a, \mathsf b, \mathsf d},
\]
\end{enumerate}
such that the following properties hold:
\begin{itemize}
\item when $j \in \{1^{\mathsf a, \mathsf b, \mathsf d}, \cdots, r_{\mathsf a, \mathsf b, \mathsf d}^{\mathsf a, \mathsf b, \mathsf d}\}$ and $i \in \sigma^{-1}(j)$, $t_{i, j}^{\mathsf a, \mathsf b, \mathsf d}$ is bijective counted with multiplicities, i.e., if $q \in \mathfrak Q_{\mathsf a, \mathsf b, \mathsf d}$ is an element with a multiplicity $e$, then $(t_{i, j}^{\mathsf a, \mathsf b, \mathsf d})^{-1}(q)$ consists of $e$ elements;
\item when $j \in \{1, \cdots, s\}$, we have
\[
\sum_{i \in \sigma^{-1}(j)}g_{i} \circ t_{i, j} \geq h_j;
\]
and
\item when $j \in \{1^{\mathsf a, \mathsf b, \mathsf d}, \cdots, r_{\mathsf a, \mathsf b, \mathsf d}^{\mathsf a, \mathsf b, \mathsf d}\}$, we have for any $q \in \mathfrak Q_{\mathsf a, \mathsf b, \mathsf d}$
\[
\sum_{i \in \sigma^{-1}(j)}\sum_{p \in (t_{i, j}^{\mathsf a, \mathsf b, \mathsf d})^{-1}(q)}g_{i}(p) + h_{\sigma_{\mathsf a, \mathsf b, \mathsf d}^{-1}(j)} \circ t_j^{\mathsf a, \mathsf b, \mathsf d}(q) \geq h_j (q).
\]
\end{itemize}

\end{defi}

With this definition, the set
\[
\bigcup_{T' \leq_+T }\mathcal N_{T'}
\]
is closed in $\mathfrak H$ and is downward-closed with respect to the poset structure of $\mathfrak H$.

\begin{defi}
Let $T$ be a saturated combinatorial type. We define
\[
\mathcal S_T = \bigcup_{T', \, \mathrm{sat}(T') \equiv T} \mathcal N_{T'}.
\]
\end{defi}
Note that as a topological space, we have
\[
\mathcal S_T = \bigsqcup_{T', \, \mathrm{sat}(T') \equiv T} \mathcal N_{T'},
\]
where $T'$ runs over representatives of equivalence classes of combinatorial types. Indeed, this is because if $T_1, T_2$ satisfies $\mathrm{sat}(T_1) \equiv \mathrm{sat}(T_2)$, then we have
\[
|\mathrm{Supp}(T_1)\cap B^\circ| = |\mathrm{Supp}(T_2)\cap B^\circ|,
\]
so we have 
\[
\mathcal N_{T_1} \cap \overline{\mathcal N}_{T_2} = \emptyset.
\]
We define the relation $\leq_{+, \mathrm{sat}}$ as the transitive closure of the following relation on the set of saturated combinatorial types:
\[
T_1 \prec_{+, \mathrm{sat}} T_2 \text{ if there exist types $T_1', T_2'$ such that $T_1'\leq_+T_2'$ and $\mathrm{sat}(T_i') = T_i$ for $i = 1, 2$.}
\]
This defines a poset relation on the set of equivalence classes of saturated combinatorial types and the saturation function from the set of combinatorial types to the set of saturated combinatorial types is a homomorphism with respect to $\leq_+$ and $\leq_{+, \mathrm{sat}}$. See \cite[Proposition 4.6]{DT24} for more details.

Let $T$ be a saturated combinatorial type. We define
\[
\mathcal Z_T = \bigcup_{T': \text{ saturated},\,  T' \leq_{+, \mathrm{sat}} T} \mathcal S_{T'} = \bigcup_{\widetilde{T}, \, \mathrm{sat}(\widetilde{T}) \leq_{+, \mathrm{sat}} T} \mathcal N_{\widetilde{T}}.
\]
Then $\mathcal Z_T$ is closed and downward-closed in $\mathfrak H$. We also have
\[
\mathcal Z_{T_1}\subset \mathcal Z_{T_2} \iff T_1 \leq_{+, \mathrm{sat}} T_2,
\]
and 
\[
\mathcal Z_T \setminus \bigcup_{T' <_{+, \mathrm{sat}} T} \mathcal Z_{T'} = \mathcal S_T.
\]

\subsubsection*{Relative variants}

We extend the definitions in this section to the poscheme
\[
(\mathfrak H \leq \mathfrak H)
\]
consisting of pairs $(w\leq x)$ where $w, x \in \mathfrak H$.
A combinatorial type $T$ is a multiset
\[
\{g_1^w\leq g_1^x, \cdots, g_r^w\leq g_r^x\} \sqcup \bigsqcup_{(\mathsf a, \mathsf b, \mathsf d) \neq ((1, 1, 1, 1), (1, 1), \emptyset)} \{g^{w, \mathsf a, \mathsf b, \mathsf d}_1\leq g^{x, \mathsf a, \mathsf b, \mathsf d}_1, \cdots, g^{w, \mathsf a, \mathsf b, \mathsf d}_{r_{\mathsf a, \mathsf b, \mathsf d}} \leq g^{x, \mathsf a, \mathsf b, \mathsf d}_{r_{\mathsf a, \mathsf b, \mathsf d}}\} .
\]
Two combinatorial types
\begin{align*}
&T =  \{g_1^w\leq g_1^x, \cdots, g_r^w\leq g_r^x\} \sqcup \bigsqcup_{(\mathsf a, \mathsf b, \mathsf d) \neq ((1, 1, 1, 1), (1, 1), \emptyset)} \{g^{w, \mathsf a, \mathsf b, \mathsf d}_1\leq g^{x, \mathsf a, \mathsf b, \mathsf d}_1, \cdots, g^{w, \mathsf a, \mathsf b, \mathsf d}_{r_{\mathsf a, \mathsf b, \mathsf d}} \leq g^{x, \mathsf a, \mathsf b, \mathsf d}_{r_{\mathsf a, \mathsf b, \mathsf d}}\},\\
&T' =  \{h_1^w\leq h_1^x, \cdots, h_s^w\leq h_s^x\} \sqcup \bigsqcup_{(\mathsf a, \mathsf b, \mathsf d) \neq ((1, 1, 1, 1), (1, 1), \emptyset)} \{h^{w, \mathsf a, \mathsf b, \mathsf d}_1\leq h^{x, \mathsf a, \mathsf b, \mathsf d}_1, \cdots, h^{w, \mathsf a, \mathsf b, \mathsf d}_{r_{\mathsf a, \mathsf b, \mathsf d}} \leq h^{x, \mathsf a, \mathsf b, \mathsf d}_{r_{\mathsf a, \mathsf b, \mathsf d}}\}
\end{align*}
are equivalent if $r = s$, and there exist permutations $\sigma \in \mathfrak S_r$, $\sigma^{\mathsf a, \mathsf b, \mathsf d} \in \mathfrak S_{r_{\mathsf a, \mathsf b, \mathsf d}}$, and automorphisms of posets with stratifications
\[
t_i \in \mathrm{Aut}(\mathfrak Q_{\mathrm{gen}}, \Gamma_\emptyset)\, (i = 1, \cdots, r), \quad t_i^{\mathsf a, \mathsf b, \mathsf d} \in \mathrm{Aut}(\mathfrak Q_{\mathsf a, \mathsf b, \mathsf d}, \Gamma_{\mathsf d}) \, (i = 1, \cdots, r_{\mathsf a, \mathsf b, \mathsf d}),
\]
such that we have 
\[
g_{\sigma(i)}^* \circ t_i = h_i^*, \quad g^{*, \mathsf a, \mathsf b, \mathsf d}_{\sigma^{\mathsf a, \mathsf b, \mathsf d}(i)}\circ t_i^{\mathsf a, \mathsf b, \mathsf d} = h^{*, \mathsf a, \mathsf b, \mathsf d}_i,
\]
where $* = w, x$.

For each equivalence class of a combinatorial type $T$, we have a Zariski locally closed subset $\mathcal N_T \subset (\mathfrak H \leq \mathfrak H)$. A combinatorial type $T$ is saturated if each component is saturated. We also have $\leq, \leq_{+}, \leq_{+, \mathrm{sat}}$ well-defined.

Let $W \subset \mathfrak H$ be a Zariski locally closed subset which is a finite union of $\mathcal N_{T''}$ where $T''$ is a combinatorial type of $\mathfrak H$.
Then
\[
\mathcal Z_T = \left(\bigcup_{T', \, \mathrm{sat}(T')\leq_{+, \mathrm{sat}} T} \mathcal N_{T'}\right)\cap (W\leq \mathfrak H)
\]
is closed and downward closed in $(W\leq \mathfrak H)$.

\subsection{Stratifications for the spaces of sections}
\label{subsec:stratificationforspaceofsections}
Let $k$ be an algebraically closed field or a finite field.
We work in the setup established in Section~\ref{subsec:pos_setup}: we fix
\[
\pi_{\mathcal Q} : \mathcal Q \to B,\quad  \mathcal Z \subset \mathcal Q, \quad \pi : \mathcal X \to B, \quad  \phi : \mathcal X \to \mathcal Q,
\]
as before.

Let $\alpha \in \Nef_1(\mathcal X)_{\mathrm{sec}, \mathbb Z}$ be an integral nef class and we consider the space of sections
\[
M_\alpha = \Sec(\mathcal X/B, \alpha).
\]
In Section~\ref{subsec:spaceofsectionsonblowup}, we realize $M_\alpha$ as a Zariski open subset
\[
M_\alpha \subset \mathrm{Proj}\left(\mathrm{Sym}\left(((\widetilde{\varpi}_{d_{\mathbb P_C(\mathcal F)}})_*(q^*\widetilde{\varpi}^*\mathcal F^\vee(-\widetilde{\mathfrak W}) \otimes p^*\mathcal L_{d_{\mathbb P_C(\mathcal F)}}))^\vee\right)\right).
\]
We define $E \to U_{k(\alpha)} \times \Pic^{d_{\mathbb P_C(\mathcal F)}}(C)$ to be
\[
E = \Spec\left(\mathrm{Sym}\left(((\widetilde{\varpi}_{d_{\mathbb P_C(\mathcal F)}})_*(q^*\widetilde{\varpi}^*\mathcal F^\vee(-\widetilde{\mathfrak W}) \otimes p^*\mathcal L_{d_{\mathbb P_C(\mathcal F)}}))^\vee\right)\right) \to U_{k(\alpha)} \times \Pic^{d_{\mathbb P_C(\mathcal F)}}(C).
\]
Note that the fiber $E_{w, L}$ at a general $(w, L) \in U_{k(\alpha)}(\overline{k}) \times \Pic^{d_{\mathbb P_C(\mathcal F)}}(C)(\overline{k})$ is identified with the vector space
\[
\rH^0(C_{\overline{k}},  \mathcal F^\vee(-\Sigma_C(w)) \otimes L).
\]
We have a $\mathbb G_m$-torsor
\[
E \setminus (\text{zero section}) \to \mathrm{Proj}\left(\mathrm{Sym}\left(((\widetilde{\varpi}_{d_{\mathbb P_C(\mathcal F)}})_*(q^*\widetilde{\varpi}^*\mathcal F^\vee(-\widetilde{\mathfrak W}) \otimes p^*\mathcal L_{d_{\mathbb P_C(\mathcal F)}}))^\vee\right)\right),
\]
and we consider the induced $\mathbb G_m$-torsor
\[
\widetilde{M}_\alpha \to M_\alpha,
\]
where $\widetilde{M}_\alpha \subset E$ is a Zariski open subset.

Our goal is now to apply the inclusion-exclusion principle to the complement
\[
E \setminus \widetilde{M}_\alpha.
\]
To this end, we stratify this complement using the poscheme $\mathfrak H$.
First we embed $U_{k(\alpha)}$ to $\mathfrak H$ by
\[
U_{k(\alpha)} \hookrightarrow \mathfrak H, \, w \mapsto x = (x_B = \emptyset, x_{\mathcal Z} = w, x_C = \emptyset, \hat{1}).
\]
We denote the base change of $E \to U_{k(\alpha)} \times \Pic^{d_{\mathbb P_C(\mathcal F)}}(C)$ via $$(U_{k(\alpha)} < \mathfrak H) \times\Pic^{d_{\mathbb P_C(\mathcal F)}}(C) \to  U_{k(\alpha)} \times \Pic^{d_{\mathbb P_C(\mathcal F)}}(C)$$ by $E'$.
We consider universal divisors by 
\begin{align*}
&\mathcal D_B \subset B \times (U_{k(\alpha)} < \mathfrak H)\times\Pic^{d_{\mathbb P_C(\mathcal F)}}(C) ,\\
& \mathcal D_{\mathcal Z} \subset \mathcal Z \times(U_{k(\alpha)} < \mathfrak H)\times\Pic^{d_{\mathbb P_C(\mathcal F)}}(C), \\
& \mathcal D_{C} \subset C \times (U_{k(\alpha)} < \mathfrak H)\times\Pic^{d_{\mathbb P_C(\mathcal F)}}(C),
\end{align*}
which are the pullbacks of the following universal families of divisors:
\begin{align*}
&\mathcal D_B' \subset B \times \mathrm{Hilb}^{\mathrm{fin}}(B),\quad \mathcal D_{\mathcal Z}' \subset \mathcal Z \times \mathrm{Hilb}^{\mathrm{fin}}(\mathcal Z), \quad \mathcal D_{C}' \subset C \times \mathrm{Hilb}^{\mathrm{fin}}(C).
\end{align*}
They satisfy 
\[
p_{\mathcal Z \to B}^*\mathcal D_B \subset \mathcal D_{\mathcal Z}, \quad p_{C \to B}^*\mathcal D_B \subset \mathcal D_{C}.
\]
We also denote the pullback of $q^*\widetilde{\varpi}^*\mathcal F^\vee(-\widetilde{\mathfrak W}) \otimes p^*\mathcal L_{d_{\mathbb P_C(\mathcal F)}}$ over $C \times \Pic^{d_{\mathbb P_C(\mathcal F)}}(C)$ via $$C \times (U_{k(\alpha)} < \mathfrak H)\times\Pic^{d_{\mathbb P_C(\mathcal F)}}(C) \to C \times \Pic^{d_{\mathbb P_C(\mathcal F)}}(C),$$ by $\mathfrak F$.
Then we have the restriction maps
\[
E' \to \Gamma(p_{C \to B}^*\mathcal D_B, \mathfrak F) \times \Gamma(p_{\mathcal Z}^*\mathcal D_Z, p_{C}^*\mathfrak F) \times \Gamma(\mathcal D_C, \mathfrak F),
\]
where the middle map is induced by the composition of the pullback via $p_C : \mathcal W \to C$ and the restriction to $p_{\mathcal Z}^*\mathcal D_{\mathcal Z}$.
Then the stratification $\mathsf Z \subset (U_{k(\alpha)} < \mathfrak H)\times_{U_{k(\alpha)}} E$ is defined as the fiber product
\[
\xymatrix{
\mathsf Z \ar[r] \ar[d]& \Gamma(p_{C \to B}^*\mathcal D_B, (0 \text{-section})) \times \Gamma(p_{\mathcal Z}^*\mathcal D_{\mathcal Z}, \mathfrak{Z}) \times \Gamma(\mathcal D_C,  (0 \text{-section})) \ar[d] 
\\
E' \ar[r] & \Gamma(p_{C \to B}^*\mathcal D_B, \mathfrak F) \times \Gamma(p_{\mathcal Z}^*\mathcal D_Z, p_{C}^*\mathfrak F) \times \Gamma(\mathcal D_C, \mathfrak F),
}
\]
where the line bundle $\mathfrak Z \subset p_{C}^*\mathfrak F$ is induced by the morphism
\[
\iota : \mathcal W = \mathcal Z \times_B C \to \mathbb P_C(\mathcal F).
\]
For each $(w < x, L) \in ((U_{k(\alpha)} < \mathfrak H) \times \Pic^{d_{\mathbb P_C(\mathcal F)}}(C))(\overline{k})$, the fiber $\mathsf Z_{w < x, L}$ is a linear subspace of $E_{w, L}$.
It is easy to verify the following lemma:
\begin{lemm}
\label{lemm:propforstratification}
Let $(w < x, L), (w < x_1, L), (w < x_2, L), \in ((U_{k(\alpha)} < \mathfrak H) \times \Pic^{d_{\mathbb P_C(\mathcal F)}}(C))(\overline{k})$. Then the following properties hold:
\begin{itemize}
\item if $x_1\leq x_2$, then we have $\mathsf Z_{w < x_1, L} \supset \mathsf Z_{w < x_2, L}$;
\item we have $\mathsf Z_{w < x, L} = \mathsf Z_{w < \mathrm{sat}(x), L}$; and
\item we have $\mathsf Z_{w < x_1\vee x_2, L} = \mathsf Z_{w < x_1, L}\cap \mathsf Z_{w < x_2, L}$.
\end{itemize}
\end{lemm}

\begin{proof}
The only non-trivial statement is that $\mathsf Z_{w < x, L} = \mathsf Z_{w < \mathrm{sat}(x), L}$.
We prove this using the notation from the proof of Proposition~\ref{prop:saturation}.
If 
\[
\iota_\diamond(p_{\mathcal Z}^*x_{\mathcal Z} - p_{\mathcal Z}^*x_{\mathcal Z}\wedge p_C^*x_C)
\]
does not induce an admissible jet at $c_1 \in C(\overline{k})$, then any section satisfying the incidence condition imposed by $\iota_\diamond(p_{\mathcal Z}^*x_{\mathcal Z} - p_{\mathcal Z}^*x_{\mathcal Z}\wedge p_C^*x_C)$ will vanish at $c_1$. This observation shows that we have
\[
\mathsf Z_{w < x, L} = \mathsf Z_{w < (x_B, x_{\mathcal Z}, x_C^*), L}. 
\]
Then it is easy to show that 
\[
\mathsf Z_{w < (x_B, x_{\mathcal Z}, x_C^*), L} = \mathsf Z_{w < (x_B^*, x_{\mathcal Z}, x_C^*), L}
= \mathsf Z_{w < (x_B^*, x_{\mathcal Z}', x_C^*), L} = \mathsf Z_{w < \mathrm{sat}(x), L}. 
\]
Thus our assertion follows.
\end{proof}

\section{Homological sieve}

In this section, we introduce the bar complexes and discuss the homological sieve to compare the cohomology of the space over $U_{k(\alpha)}$ with the cohomology of the bar complex. Those concepts in the context of the space of maps were originally introduced over $\mathbb C$ in \cite{DT24} and extended to arbitrary perfect fields in \cite{DLTT25}. Our discussion runs parallel to \cite[Sections 7 and 8]{DLTT25}, and in many cases, the proofs are identical. In that situation, we only state the theorems and omit the proofs.

\subsection{The bar complex}
\label{subsec:barcomplex}

In this section, we closely follow the exposition of \cite[Section 7]{DLTT25}.
Let $k$ be a perfect field.
We work in the setup established in Section~\ref{sec:poschemes} and use the notation established there.
Note that $U_{k(\alpha)} \subset \mathfrak H$ is a locally closed subset which is a finite union of $\mathcal N_{T'}$ where $T'$ is a saturated combinatorial type of $\mathfrak H$.
Let $P \subset (U_{k(\alpha)} < \mathfrak H)$ be a closed and downward closed union of locally closed subsets $\mathcal S_T$ for finitely many saturated combinatorial types $T$ of $(U_{k(\alpha)} < \mathfrak H)$. Fix a Zariski open subset 
\[
U \subset U_{k(\alpha)} \times \Pic^{d_{\mathbb P_C(\mathcal F)}}(C).
\]
Let $(U < \mathfrak H)$ and $P_U\to U$ be the base changes of 
\[
(U_{k(\alpha)} < \mathfrak H) \to U_{k(\alpha)}, \quad P \subset (U_{k(\alpha)} < \mathfrak H) \to U_{k(\alpha)}
\]
via 
\[
U \subset U_{k(\alpha)} \times \Pic^{d_{\mathbb P_C(\mathcal F)}}(C) \to U_{k(\alpha)}.
\]
Suppose that we have a flat separated morphism $E \to U$ of finite type with geometrically irreducible fibers. Assume that this admits a stratification
\[
\mathsf Z \subset (U < \mathfrak H)\times_{U} E.
\]
We assume that this satisfies the properties of Lemma~\ref{lemm:propforstratification}.

Let $\Delta$ be the category of non-empty finite subsets of $\mathbb N$ with non-decreasing maps.

\begin{defi}
The bar complex $B(P_U, \mathsf Z_{P_U})$ is a simplicial scheme which is a contravariant functor from $\Delta$ to the category of separated schemes and assigns $[n] = \{ 0, \cdots, n\}$ to the scheme
\[
\{ ((w, L), w < x_0 \leq \cdots \leq x_n \in \mathfrak H, z \in \mathsf Z_{w < x_n, L}) \, | \, (w < x_i) \in P\},
\]
with the augmentation morphism $\epsilon : B(P_U, \mathsf Z_{P_U}) \to E$ given by
\[
((w, L), w < x_0 \leq \cdots \leq x_n, z) \mapsto (w, L, z) \in E_{w, L}.
\]
See \cite[Section 2.3]{DLTT25} and references therein for simplicial schemes and related terminology.
\end{defi}

\subsubsection*{Spectral sequence}

Let $T$ be a saturated type of $P$ and we consider $\mathsf Z|_{U \times_{U_{k(\alpha)}} \mathcal N_T}$. We further consider the associated nerve:
\[
\epsilon_T : (N(-\infty, \mathsf Z|_{U \times_{U_{k(\alpha)}} \mathcal S_T}) < \mathsf Z|_{U \times_{U_{k(\alpha)}} \mathcal N_T}) \to \mathsf Z|_{U \times_{U_{k(\alpha)}} \mathcal N_T},
\]
which is defined as the simplicial scheme over $\mathsf Z|_{U \times_{U_{k(\alpha)}} \mathcal N_T}$ whose $n$-th simplex is given by
\[
\{(L, w < x_0 \leq \cdots \leq x_n < y, z) \, |\, (w, L) \in U, (w<y) \in \mathcal N_T, (w < x_n) \in \mathcal Z_T \setminus \mathcal S_T, z \in Z_{w < y, L}\},
\]
and the augmentation morphism $\epsilon_T$ is given by
\[
(L, w < x_0 \leq \cdots \leq x_n < y, z) \mapsto (L, w < y, z).
\]
Let $D^b(\mathsf Z|_{U \times_{U_{k(\alpha)}} \mathcal N_T})$ be the bounded derived category of $\mathbb Z_\ell$-constructible sheaves in the pro-\'etale topology.
Let $\mu'(T)[2] \in D^b(\mathsf Z|_{U \times_{U_{k(\alpha)}} \mathcal N_T})$ be the cone to form the distinguished triangle
\[
\underline{\mathbb Z}_\ell \to R\epsilon_{T*}\epsilon_T^*\underline{\mathbb Z}_\ell \to \mu'(T)[2] \to \underline{\mathbb Z}_\ell[1].
\]
Then we have the following spectral sequence:

\begin{theo}[{\cite[Theorem 7.3]{DLTT25}}]
\label{theo:Theorem7.3DLTT}
There is a spectral sequence:
\[
\bigoplus_{i + j = n}E_1^{i, j} = \bigoplus_{T: \textnormal{ essential types of $P$}}\rH^n_{\textnormal{\'et}, c}((\mathsf Z|_{U \times_{U_{k(\alpha)}} \mathcal N_T})_{\overline{k}}, \mu'(T)[1]) \implies\rH^n_{\textnormal{\'et}, c}(B(P_U, \mathsf Z_{P_U})_{\overline{k}}, \mathbb Z_\ell).
\]
\end{theo}

\begin{proof}
The proof is identical to that of \cite[Theorem 7.3]{DLTT25}, so we omit it.
\end{proof}

\subsubsection*{The main theorem on the bar complex}

We closely follow the exposition of \cite[Section 7.3]{DLTT25}. Suppose that we have a codimension function 
\[
\gamma : (U_{k(\alpha)} < \mathfrak H) \to \mathbb N.
\]
(For our applications, this $\gamma$ is the one defined in Example~\ref{exam:codimension}.)
We say that $\mathsf Z_{w < x, L}$ is unobstructed if its codimension in $E_{w, L}$ is given by $\gamma(w < x) = \gamma(x) - \gamma(w)$.
We define $\mathrm{Supp}(w < x) = \mathrm{Supp}(x-w)$, and we also consider $\mathrm{Supp}(w < x) \cap B^\circ$ consisting of divisors supported over $B^\circ$. For any saturated combinatorial type $T$, we let $\kappa(T)$ to be
\[
\kappa(T) = 2\gamma(T) -\mathrm{rank} (T) -2 |\mathrm{Supp}(T)\cap B^\circ|.
\]
This is viewed as the expected cohomological codimension of the contribution of $\mathsf Z|_{U\times_{U_{k(\alpha)}} \mathcal N_T}$.
We consider the evaluation map $\mathsf Z_{P_U} \to E$ and denote its image by $\mathrm{im}(\mathsf Z_{P_U} \to E)$.

The following theorem was proved over $\mathbb C$ in the language of simplicial spaces in \cite[Theorem 5.9]{DT24} and it was extended to arbitrary perfect fields in \cite[Theorem 7.5]{DLTT25}:

\begin{theo}[{\cite[Theorem 7.5]{DLTT25}}]
\label{theo:Theorem7.5DLTT}
Let $I \in \mathbb Q_{\geq 0}$. Assume that 
\begin{enumerate}
\item $P \to U_{k(\alpha)}$ is proper and $P$ is downward closed;
\item for every $(w < x, L) \in P_U(\overline{k})$ with $x$ being saturated and every $y \in \mathfrak H^\circ(\overline{k})$ such that $x \prec y$, $\mathsf Z_{w < x, L}$ and $\mathsf Z_{w < y, L}$ are unobstructed; and
\item $P$ contains all $\mathcal S_T$ such that $\kappa(T) \leq I$.
\end{enumerate}
Then the map on compactly supported \'etale cohomology induced by
\[
B(P_U, \mathsf Z_{P_U}) \to \mathrm{im}(\mathsf Z_{P_U} \to E),
\]
is cohomology connected in codimension $\lfloor I \rfloor+2$, i.e., for any $i > 2 \dim \, E -\lfloor I \rfloor -2$, the homomorphism
\[
\rH^i_{\textnormal{\'et}, c}( \mathrm{im}(\mathsf Z_{P_U} \to E)_{\overline{k}}, \mathbb Z_\ell) \to \rH^i_{\textnormal{\'et}, c}( B(P_U, \mathsf Z_{P_U})_{\overline{k}}, \mathbb Z_\ell)
\]
is an isomorphism, and for $i =2 \dim \, E -\lfloor I \rfloor -2$, it is a surjection.

\end{theo}

\begin{proof}
The proof of \cite[Theorem 7.5]{DLTT25} applies without any modification.
\end{proof}

\subsubsection*{The cohomological virtual bar complex}
We follow the exposition of \cite[Section 7.4]{DLTT25}.
Let us work in the setup in Section~\ref{subsec:stratificationforspaceofsections}.
We have $E \to U_{k(\alpha)} \times \Pic^d(C)$ which is generically a vector bundle and $\mathsf Z_{w < x, L}$ is a subspace of $E_{w, L}$. In this section, we construct the cohomological virtual bar complex which encodes how the bar complex behaves if every linear subspace has expected dimension.

Let $U \subset U_{k(\alpha)} \times \Pic^d(C)$ be a Zariski open subset.
Let $T$ be a saturated combinatorial type of $(U_k < \mathfrak H)$.
We consider the following nerve:
\[
(N(-\infty, U\times_{U_{k(\alpha)}} \mathcal S_T), U\times_{U_{k(\alpha)}} \mathcal N_T)
\]
is the simplicial scheme over $U\times_{U_{k(\alpha)}} \mathcal N_T$,
whose $n$-th simplex is given by
\[
\{(L, w < x_0 \leq \cdots \leq x_n <y) \,|\, (w, L) \in U, (w< y) \in \mathcal N_T, (w < x_n) \not\in \mathcal S_T\},
\]
with the augmentation morphism $\hat{\epsilon}_T : (N(-\infty, U\times_{U_k} \mathcal S_T), U\times_{U_{k(\alpha)}} \mathcal N_T) \to U\times_{U_{k(\alpha)}} \mathcal N_T$ given by
\[
(L, w < x_0 \leq \cdots \leq x_n <y) \mapsto (L, w< y).
\]
We define the complex $\mu(T)[2] \in D^b(U\times_{U_{k(\alpha)}} \mathcal N_T)$ by the distinguished triangle
\[
\underline{\mathbb Z}_\ell \to R\hat{\epsilon}_{T*}\hat{\epsilon}_T^*\underline{\mathbb Z}_\ell \to \mu(T)[2] \to \underline{\mathbb Z}_\ell[1].
\]
We denote the operation $(d)[2d]$ by $\langle d\rangle$.
\begin{defi}
We define the cohomological virtual bar complex $\mathcal A_{h_{\mathcal Q}, k(\alpha)}(U)$ by
\begin{align*}
\mathcal A_{h_{\mathcal Q}, k(\alpha)}(U) & = \bigoplus_{T: \text{ essential type}}\mathcal A_{h_{\mathcal Q}, k(\alpha)}(U, T) \\
&= \bigoplus_T\bigoplus_i \rH^i_{\text{\'et}, c}((U\times_{U_{k(\alpha)}}\mathcal N_T)_{\overline{k}}, \mu(T)[1]\langle -n_{h_{\mathcal Q}, k(\alpha), T}\rangle \otimes \mathbb Q_\ell),
\end{align*}
where $n_{h_{\mathcal Q}, k(\alpha), T} = h_{\mathcal Q} + 3-2g(B)-g(C) -2k(\alpha) - \gamma(T)$.
When $U=  U_{k(\alpha)} \times \Pic^d(C)$, we write $\mathcal A_{h_{\mathcal Q}, k(\alpha)}= \mathcal A_{h_{\mathcal Q}, k(\alpha)}(U)$.
\end{defi}

Let us assume that $k = \mathbb F_q$ and $U = U_{k(\alpha)} \times \Pic^d(C)$. Then we have
\begin{prop}[{\cite[Proposition 7.8]{DLTT25}}]
\label{prop:Proposition7.8DLTT}
Let $T$ be an essential type. Then we have
\[
\sum_i (-1)^i \mathrm{Tr}(\mathrm{Frob}\curvearrowright \mathcal A^i_{h_{\mathcal Q}, k(\alpha)}(T)) = - \#\Pic^0(C)(k)q^{h_{\mathcal Q} +3-2g(B)-g(C)} \sum_{(w < x) \in \mathcal N(T)(k)} \mu_k(w, x)q^{-\gamma(x)},
\]
where $\mu_k$ is the M\"obius function for the poset $\mathfrak H^\circ(k)$. See \cite[Definition 7.6]{DLTT25} for its definition.
\end{prop}

\begin{proof}
The proof is similar to that of \cite[Proposition 7.8]{DLTT25}.
\end{proof}

\subsubsection*{Convergence.}
In this section, we work over an algebraic closure $\overline{k}$ of a perfect field $k$.
First let us describe $\mathcal N_T$ when $T$ is an essential saturated type of $(U_{k(\alpha)} \leq \mathfrak H)$. Recall that for $\mathsf a = (1, 1, 1, 1), (2, 1, 1), (3, 1), (2, 2), (4)$, $\mathsf b = (1, 1), (2)$, and $\mathsf d \in \mathfrak R_{\mathsf a, \mathsf b}$
we introduce the abstract poset with the stratification and the configuration of lines
\[
(\mathfrak Q_{\mathsf a, \mathsf b, \mathsf d}, \Gamma_{\mathsf d}),
\]
and we denote $\mathfrak Q_{\mathrm{gen}} = \mathfrak Q_{(1, 1, 1, 1), (1, 1), \emptyset}$.
Note that when $\mathsf a = (2, 1, 1)$ or $(3, 1)$, the element with a multiplicity is distinguished. 

Also recall that in Example~\ref{exam:essentialtypes}, we list possible essential pairs.
For each $\mathsf a, \mathsf b, \mathsf d$ with $(\mathsf a, \mathsf b, \mathsf d) \neq ((1, 1, 1, 1), (1, 1), \emptyset)$, we list representatives of equivalence classes of essential pairs as
\begin{align*}
&\mathsf g_1^{\mathsf a, \mathsf b, \mathsf d} = (\hat{1}\leq g_1^{x, \mathsf a, \mathsf b, \mathsf d}), \quad \cdots, \quad \mathsf g_{s_{\mathsf a, \mathsf b, \mathsf d}}^{\mathsf a, \mathsf b, \mathsf d} = (\hat{1}\leq g_{s_{\mathsf a, \mathsf b, \mathsf d}}^{x, \mathsf a, \mathsf b, \mathsf d})\\
&\mathsf g_1^{\mathsf a, \mathsf b, \mathsf d, m} = (g_{\mathsf b, z, m}^w\leq g_1^{x, \mathsf a, \mathsf b, \mathsf d, m}), \quad \cdots, \quad \mathsf g_{s_{\mathsf a, \mathsf b, \mathsf d,  m}}^{\mathsf a, \mathsf b, \mathsf d, m}  = (g_{\mathsf b, z, m}^w\leq g_{s_{\mathsf a, \mathsf b, \mathsf d,  m}}^{x, \mathsf a, \mathsf b, \mathsf d, m})\\
&\mathsf g_1^{\mathsf a, \mathsf b, \mathsf d, \mathrm{ram}} = (g_{\mathsf b, z', \mathrm{ram}}^w\leq g_1^{x, \mathsf a, \mathsf b, \mathsf d, \mathrm{ram}}), \quad \cdots, \quad \mathsf g_{s_{\mathsf a, \mathsf b, \mathsf d,  \mathrm{ram}}}^{\mathsf a, \mathsf b, \mathsf d, \mathrm{ram}}  = (g_{\mathsf b, z', \mathrm{ram}}^w\leq g_{s_{\mathsf a, \mathsf b, \mathsf d,  \mathrm{ram}}}^{x, \mathsf a, \mathsf b, \mathsf d, \mathrm{ram}})
\end{align*}
where $m$ is a positive integer, $z \in \mathcal Z$ is unramified over $B$, $z' \in \mathcal Z$ is ramified over $B$, $g_{\mathsf b, z, m}^w$ is induced by $m[z] \in U_{m}$, and $g_{\mathsf b, z', \mathrm{ram}}^w$ is induced by $z' \in U_1$.
When $\mathsf a = (1, 1, 1, 1), \mathsf b = (1, 1), \mathsf d = \emptyset$, we list them as
\begin{align*}
&\mathsf g_1 = (\hat{1} < ([z], [z])), \quad \mathsf g_2  = (\hat{1} <  (\hat{1}, [c])), \quad \mathsf g_3 = (\hat{1} < ([z], [c])), \quad \mathsf g_4 = (\hat{1} < ([b], [b]))\\
& \mathsf g_1^m = ((m[z], m[z]) \leq (m[z], m[z])), \quad \mathsf g_2^m = ((m[z], m[z]) < ((m+1)[z], (m+1)[z])),\\ &\mathsf g_3^m = ((m[z], m[z]) < (m[z], [c] + (m-1)[z])), \quad
 \mathsf g_4^m = ((m[z], m[z]) < ([b] + (m-1)[z], [b])),\\ & \mathsf g_5^m = ((m[z], m[z]) < ((m+1)[z], [c] + m[z])),
\quad\mathsf g_6^m = ((m[z], m[z]) < ([b] + m[z], [b]+ m[z])),
\end{align*}
where $m$ is a positive integer. (Here we intentionally exclude $\mathsf g_0 = (\hat{1}\leq\hat{1})$.)
Let $\mathcal Z^\circ, C^\circ, \mathcal W^\circ$ be the preimages of $B^\circ$.
Then let $\mathcal D_i^\circ, \mathcal D_i^{m\circ}$ be
\begin{align*}
&\mathcal D_1^\circ = \mathcal Z^\circ_{\overline{k}}, \quad   \mathcal D_2^\circ = C^\circ_{\overline{k}}, \quad  \mathcal D_3^\circ = \mathcal W_{\overline{k}}^\circ, \quad  \mathcal D_4^\circ = B^\circ_{\overline{k}} \\
&  \mathcal D_1^{m\circ} = \mathcal Z^\circ_{\overline{k}}, \quad \mathcal D_2^{m\circ} = \mathcal Z^\circ_{\overline{k}}, \quad \mathcal D_3^{m\circ} = \mathcal W_{\overline{k}}^\circ, \quad \mathcal D_4^{m\circ} = \mathcal Z^\circ_{\overline{k}}, \quad \mathcal D_5^{m\circ} = \mathcal W_{\overline{k}}^\circ,\quad \mathcal D_6^{m\circ} = \mathcal Z^\circ_{\overline{k}}.
\end{align*}

Let us fix an essential combinatorial type $T$ of $(U_k < \mathfrak H)$.
For $\mathsf a, \mathsf b, \mathsf d$, $m \geq 0$, and $i$, let $n_{\mathsf a, \mathsf b, \mathsf d, m, i}$ (resp. $n_{\mathsf a, \mathsf b, \mathsf d, \mathrm{ram}, i}$) be the multiplicity of $\mathsf g_i^{\mathsf a, \mathsf b, \mathsf d, m}$ (resp. $\mathsf g_i^{\mathsf a, \mathsf b, \mathsf d, \mathrm{ram}}$) in $T$. When $\mathsf a = (1, 1, 1, 1), \mathsf b =(1, 1), \mathsf d = \emptyset$, we simply denote them by $n_{m, i}$. Then $\mathcal N_T$ is the disjoint union of the open subset of the form of
\[
\mathrm{Conf}^\circ\left(\prod_{m = 0}^{k(\alpha)} \prod_i (\mathcal D_i^{m\circ})^{n_{m, i}} \right)/\left(\prod_m \prod_i \mathfrak S_{n_{m, i}}\right),
\]
where $\mathrm{Conf}^\circ$ denotes the big Zariski open subset such that every coordinates are distinct when we take the pushforwards to $B^\circ$.
Note that when $(\mathsf a, \mathsf b, \mathsf d) \neq ((1, 1, 1, 1), (1, 1), \emptyset)$, we have
\begin{equation}
\label{equation:particition}
r_{\mathsf a, \mathsf b, \mathsf d} = \sum_{m = 0}^{k(\alpha)} \sum_i n_{\mathsf a, \mathsf b, \mathsf d, m, i} + \sum_i n_{\mathsf a, \mathsf b, \mathsf d, \mathrm{ram}, i}.
\end{equation}
Let $p(r_{\mathsf a, \mathsf b, \mathsf d}| (n_{\mathsf a, \mathsf b, \mathsf d, m, i})_{m, i} \cup (n_{\mathsf a, \mathsf b, \mathsf d, \mathrm{ram}, i})_i)$ be the number of partitions of $\{1, \cdots, r_{\mathsf a, \mathsf b, \mathsf d}\}$ into subsets of cardinalities given by (\ref{equation:particition}).
Note that since $r_{\mathsf a, \mathsf b, \mathsf d}$ is fixed, such a number is uniformly bounded depending only on $\mathcal X_{\overline{k}}$.
Then the number of the above components in $\mathcal N_T$ is given by
\begin{align*}
\mathsf C(T) : =&\prod_{(\mathsf a, \mathsf b, \mathsf d) \neq ((1, 1, 1, 1), (1, 1), \emptyset)} \Bigg( p(r_{\mathsf a, \mathsf b, \mathsf d}| (n_{\mathsf a, \mathsf b, \mathsf d, m, i})_{m, i}) \cup (n_{\mathsf a, \mathsf b, \mathsf d, \mathrm{ram}, i})_i) \\
&\times \prod_{m = 0}^{k(\alpha)} \prod_i \left(\frac{\#\mathrm{Aut}(\mathfrak Q_{\mathsf a, \mathsf b, \mathsf d}, \Gamma_{\mathsf a, \mathsf b, \mathsf d})}{\# \mathrm{Stab}(\mathsf g_i^{\mathsf a, \mathsf b,\mathsf d, m})} \right)^{n_{\mathsf a, \mathsf b,\mathsf d, m, i}} \times\prod_i \left(\frac{\#\mathrm{Aut}(\mathfrak Q_{\mathsf a, \mathsf b, \mathsf d}, \Gamma_{\mathsf a, \mathsf b, \mathsf d})}{\# \mathrm{Stab}(\mathsf g_i^{\mathsf a, \mathsf b,\mathsf d, \mathrm{ram}})} \right)^{n_{\mathsf a, \mathsf b,\mathsf d, \mathrm{ram}, i}}\Bigg),
\end{align*}
where $\mathrm{Stab}(\mathsf g_i^{\mathsf a, \mathsf b,\mathsf d, m})$ is the stabilizer of $\mathsf g_i^{\mathsf a, \mathsf b,\mathsf d, m}$.
In particular, there is a uniform constant $\mathsf C_1$ depending only on $\mathcal X_{\overline{k}}$ such that the above number is bounded by $\mathsf C_1$.

From now on, we follow the discussion of \cite[Section 7.5]{DLTT25}.
Let $k = \mathbb F_q$.
Let $\alpha = (h_{\mathcal Q}, k(\alpha))$ and recall that we defined
\[
\mathcal A_\alpha = \bigoplus_{T: \text{ essential types}} \bigoplus_i \rH^i_{\text{\'et}, c}((\mathcal N_T)_{\overline{k}}, \mu(T)[1]\langle -n_{h_{\mathcal Q}, k(\alpha), T}\rangle \otimes \mathbb Q_\ell).
\]
We would like to bound the $\mathsf L^1$-trace of the Frobenius of the truncation into cohomological degrees $\leq 2h_{\mathcal Q} - 2k(\alpha) + 6-4g(B)-2g(C)-I$. Our goal is to show that the $\mathsf L^1$-trace of this truncation is $o(q^{h_{\mathcal Q}-k(\alpha)+ 3-2g(B)-g(C)})$.
To this end, we consider the following twist
\[
\mathcal A_\alpha' = \bigoplus_{T: \text{ essential types}} \bigoplus_i \rH^i_{\text{\'et}, c}((\mathcal N_T)_{\overline{k}}, \mu(T)[|\mathrm{Supp}(T)| + r]\langle k(T) + \gamma (T) \rangle \otimes \mathbb Q_\ell),
\]
where $k(T) = k(\alpha)$ and $r = \sum_{(\mathsf a, \mathsf b, \mathsf d) \neq ((1, 1, 1, 1), (1, 1), \emptyset)}r_{\mathsf a, \mathsf b, \mathsf d}$. It therefore suffices to show that the $\mathsf L^1$-trace of the Frobenius of the truncation of $\mathcal A_\alpha'$ to cohomological degrees $\leq -I -|\mathrm{Supp}(T)\cap B^\circ|- r +1$ is bounded as $o(1)$.

Following \cite[Section 7.5]{DLTT25}, we rename the $T$-component as $\mathcal C_{T}$ and we consider the bigraded vector space
\[
\mathcal C = \bigoplus_T \mathcal C_T,
\]
where $T$ runs over all saturated types of $(\mathfrak H \leq \mathfrak H)$. (When $T$ is not essential or does not take the form of combinatorial types of $(U_{k(\alpha)} < \mathfrak H)$, we simply assign zero.) Here the first grading is given by $k(T) + |\mathrm{Supp}(T)\cap B^\circ| + r$ and the second grading is given by the homological grading.

We have the locally constant pro-\'etale sheaf on $\mathcal N_T$:
\[
\mathbb H^*(\mu(T)[|\mathrm{Supp}(T)\cap B^\circ| + r]):= \bigoplus_i\mathbb H^i(\mu(T)[|\mathrm{Supp}(T)\cap B^\circ| + r])
\]
and this pullbacks to the constant local system
\begin{align*}
&\left(\bigotimes_{m = 0}^{k(T)} \bigotimes_i \rH^*(\mu(\mathsf g_i^m)[1])^{\otimes n_{m, i}} \right) \\&\otimes \left(\bigotimes_{(\mathsf a, \mathsf b, \mathsf d) \neq ((1, 1, 1, 1), (1, 1), \emptyset)} \left(\bigotimes_{m = 0}^{k(T)}\bigotimes_i \rH^*(\mu(\mathsf g_i^{\mathsf a, \mathsf b, \mathsf d, m})[1])^{\otimes n_{\mathsf a, \mathsf b, \mathsf d, m, i}}  \otimes \bigotimes_{i} \rH^*(\mu(\mathsf g_i^{\mathsf a, \mathsf b,\mathsf d, \mathrm{ram}})[1])^{\otimes n_{\mathsf a, \mathsf b,\mathsf d, \mathrm{ram}, i}}\right)\right)
\end{align*}
on $\mathrm{Conf}^\circ\left(\prod_{m = 0}^{k(\alpha)} \prod_i (\mathcal D_i^{m\circ})^{n_{m, i}} \right)$, where the first term is $\prod_m \prod_i \mathfrak S_{n_{m, i}}$-equivariant with respect to the Koszul sign rule.
Hence by Galois descent, 
\begin{align*}
&\rH^*_{\text{\'et}, c}((\mathcal N_T)_{\overline{k}}, \mu(T)[|\mathrm{Supp}(T)\cap B^\circ| + r]\langle k(T) + \gamma (T) \rangle \otimes \mathbb Q_\ell)
\end{align*}
is the direct sum of $\mathsf C(T)$ copies of
\begin{align*}
&= \left(\rH^*_{\text{\'et}, c}\left(\mathrm{Conf}^\circ\left(\prod_{m = 0}^{k(\alpha)} \prod_i (\mathcal D_i^{m\circ})^{n_{m, i}} \right)\right)\otimes\left(\bigotimes_{m = 0}^{k(T)} \bigotimes_i \rH^*(\mu(\mathsf g_i^m)[1]\langle \gamma(\mathsf g_i^m) + k(\mathsf g_i^m) \rangle)^{\otimes n_{m, i}} \right)\right)^{\prod_m \prod_i \mathfrak S_{n_{m, i}} }\\
&\qquad \qquad \otimes \Bigg(\bigotimes_{(\mathsf a, \mathsf b, \mathsf d) \neq ((1, 1, 1, 1), (1, 1), \emptyset)} \Bigg(\bigotimes_{m = 0}^{k(T)}\bigotimes_i \rH^*(\mu(\mathsf g_i^{\mathsf a, \mathsf b,\mathsf d, m})[1]\langle \gamma(\mathsf g_i^{\mathsf a, \mathsf b,\mathsf d,  m}) + k(\mathsf g_i^{\mathsf a, \mathsf b, \mathsf d, m})\rangle)^{\otimes n_{\mathsf a, \mathsf b, \mathsf d, m, i}}  \\
&\qquad \qquad\qquad \qquad\otimes \bigotimes_{i} \rH^*(\mu(\mathsf g_i^{\mathsf a, \mathsf b, \mathsf d, \mathrm{ram}})[1]\langle \gamma(\mathsf g_i^{\mathsf a, \mathsf b, \mathsf d, \mathrm{ram}}) + k(\mathsf g_i^{\mathsf a, \mathsf b, \mathsf d, \mathrm{ram}})\rangle)^{\otimes n_{\mathsf a, \mathsf b, \mathsf d, \mathrm{ram}, i}}\Bigg)\Bigg).
\end{align*}
For an essential pair $(w \leq x)$, we define
\[
L(w\leq x) :=\rH^*(\mu(w\leq x)[1]\langle \gamma(w\leq x) + k(w) \rangle).
\]
For each $(\mathsf a, \mathsf b, \mathsf d)$, let $\mathcal J_{\mathsf a, \mathsf b, \mathsf d}$ be the set of representatives of equivalence classes of essential types listed before. We let $\mathcal J = \mathcal J_{(1, 1, 1, 1), (1, 1), \emptyset}$ and $\mathcal J' = \sqcup_{(\mathsf a, \mathsf b, \mathsf d) \neq ((1, 1, 1, 1), (1, 1), \emptyset)}\mathcal J_{\mathsf a, \mathsf b, \mathsf d}$. Recall that we exclude $(\hat{1}\leq\hat{1})$ from $\mathcal J$.
When $(w \leq x) \in \mathcal J\sqcup \mathcal J'$, the first grading of $L(w \leq x)$ is in degree $k(w) + 1$.

Then the $\mathsf L^1$-trace of $\mathcal C$ is bounded by $\mathsf C_1 \times $ the $\mathsf L^1$-trace of  
\begin{align*}
&\Bigg (\bigoplus_{n : \mathcal J \to \mathbb N: \text{ fin. supported}}\Bigg(\rH^*_{\text{\'et}, c}\Bigg(\mathrm{Conf}^\circ\Bigg(\prod_{j \in \mathcal J}(\mathcal D_j^{\circ})^{n(j)}\Bigg), \mathbb Q_\ell\Bigg)\otimes \bigotimes_{j \in \mathcal J} L(j)^{\otimes n(j)} \Bigg)_{\prod_{j \in \mathcal J} \mathfrak S_{n(j)}}\\
&\otimes \Bigg(\bigoplus_{n' : \mathcal J' \to \mathbb N: \text{ fin. supported}}\bigotimes_{j \in \mathcal J'} L'(j)^{\otimes n'(j)} \Bigg),
\end{align*}
where $n'$ runs over functions $n' : \mathcal J' \to \mathbb N$ such that $\sum_{j \in \mathcal J_{\mathsf a, \mathsf b}} n'(j) = r_{\mathsf a, \mathsf b}$.
Let $\rho_j : \mathcal D_j^\circ \to B^\circ$ be the projection and we define the sheaf
\[
\mathfrak L = \bigoplus_{j \in \mathcal J} \rho_{j*}\underline{\mathbb Q}_\ell \otimes L(j), \quad \mathfrak L_{\mathsf a, \mathsf b, \mathsf d}=  \bigoplus_{j' \in \mathcal J_{\mathsf a, \mathsf b, \mathsf d}} L'(j').
\]
We also let
\[
\mathfrak C  =  \bigoplus_{m \geq 0} \mathfrak C_{\bullet, \bullet, m} = \bigoplus_{ m \geq 0}\rH^*_{\text{\'et}, c}(\mathrm{Conf}_m(B^\circ), \mathfrak L^{\boxtimes m})_{\mathfrak S_m}.
\]
Then
$\mathcal C$
 is isomorphic to
 \[
 \mathfrak C \otimes \bigotimes_{(\mathsf a, \mathsf b, \mathsf d) \neq ((1, 1,1, 1), (1, 1), \emptyset)} \mathfrak L_{\mathsf a, \mathsf b, \mathsf d}^{\otimes r_{\mathsf a, \mathsf b, \mathsf d}}.
 \]
Thus it suffices to verify the assumptions of Proposition~\ref{prop:keyestimate}.
To this end, we need to understand 
\[
|\mathrm{Frob}, \rH_{\textnormal{\'et}, c}^j(B^\circ_{\overline{k}},  \mathfrak L_{\overline{k}, n, i})|.
\]
First of all, $\mathfrak L_{n, i}$ is concentrated in the first grading $\geq 1$.
Then because of twisting $\langle \gamma(w\leq x) + k(w) \rangle$, $L(w\leq x)$ has the maximal cohomological degree $-2$ and the minimal cohomological degree $-2(\gamma(w\leq x) + k(w))$.
Due to finite possibilities of poset structures of essential pairs as Example~\ref{exam:essentialtypes} and twisting $\langle \gamma(w\leq x) + k(w) \rangle$, there is some uniform constant $E$ depending only on $B^\circ_{\overline{k}}$ such that 
\[
|\mathrm{Frob}, \rH_{\textnormal{\'et}, c}^j(B^\circ_{\overline{k}},  \mathfrak L_{\overline{k}, n, i})|\leq Eq^{j/2+i/2}.
\]
Similarly, there is a uniform constant $E_2$ such that we have
\[
|\mathrm{Frob}, \bigotimes_{(\mathsf a, \mathsf b, \mathsf d) \neq ((1, 1,1, 1), (1, 1), \emptyset)} \mathfrak L_{\mathsf a, \mathsf b,\mathsf d, n, i} | \leq E_2q^{i/2}.
\]
Hence Proposition~\ref{prop:keyestimate} shows that we have the following theorem:
\begin{theo}
\label{theo:keyestimate}
Let $\alpha = (h_{\mathcal Q}, k(\alpha))$.
There exists a uniform constant $E$ depending only on $\mathcal X_{\overline{k}}$ such that assuming $q > 2^{22}E^2$, we have
\[
\left|\mathrm{Frob}, \bigoplus_{i \leq  2h_{\mathcal Q} - 2k(\alpha) + 6-4g(B)-2g(C) -I} \mathcal A_\alpha^i\right| = O(q^{h_{\mathcal Q} - k(\alpha) + 3-2g(B)-g(C)}(k(\alpha) + 1)(4E)^{k(\alpha)}4^Iq^{-I/2 + 1/2}),
\]
where the implied constant is depending only on $\mathcal X_{\overline{k}}$.
\end{theo}
\subsection{Homological sieve}
\label{subsec:sieve}

Here we closely follow the exposition of \cite[Section 8.3]{DLTT25}.
We work in the setup established in Section~\ref{subsec:pos_setup}.
Let $\alpha$ be a nef class of sections on $\mathcal X$ and we define
\[
h(\alpha) = -K_{\mathcal X/B}.\alpha, \quad h_{\mathcal Q}(\alpha) = -\phi^*K_{\mathcal Q/B}.\alpha, \quad k(\alpha) = E.\alpha.
\]
Note that we have the relation $h(\alpha) = h_{\mathcal Q}(\alpha) - k(\alpha)$.
We also define $I \in \mathbb Q_{\geq 0}$ by 
\[
I = \frac{1}{10}(h(\alpha) - k(\alpha)  - \mathsf c_3(\mathcal F)),
\]
where $\mathsf c_3(\mathcal F)= \max\{\deg(\mathcal F) - \frac{\Delta}{2}-2 + 14g(C) - 2\mathrm{neg}(\mathbb P(\mathcal F)/C),  - \frac{\Delta}{2} +12g(C)- 1\}$.
We assume that $I$ is sufficiently large compared to a fixed constant only depending on $\mathcal X_{\overline{k}}$.

For $(w < x) \in (U_{k(\alpha)} <  \mathfrak H)(\overline{k})$, we have
\begin{align*}
&\gamma(w < x) = 2m_B(x) + m_{\mathcal Z}(x) + 2m_C(x)  - 2k(\alpha).
\end{align*}
We let $P \subset (U_{k(\alpha)} < \mathfrak H)(\overline{k})$ be the subposet consisting of $(w < x)$ such that $\gamma(w < x) \leq 2I$.
It follows from the following lemma that $P$ is proper over $U_{k(\alpha)}$:
\begin{lemm}
For any positive constant $\mathsf t$, the set of $x \in \mathfrak H$ such that $\gamma(x) \leq \mathsf t$ is closed and downward closed in $\mathfrak H$.
\end{lemm}

\begin{proof}
 A proof is similar to that of \cite[Lemma 8.4]{DLTT25}. We omit it.

\end{proof}

Thus $P$ is proper over $U_{k(\alpha)}$. We constructed the structure morphism
\[
\Phi_\alpha \times \mathrm{AJ}: M_\alpha \to U_{k(\alpha)} \times \Pic^{d_{\mathbb P_C(\mathcal F)}}(C).
\]
Let $F_1$ be the closed subset of Proposition~\ref{prop:dominantsmooth}.
It follows from Proposition~\ref{prop:dominantsmooth} that the codimension of $F_1$ is greater than or equal to $5I$.

Let $(w < x ) \in P(\overline{k})$ be a saturated element and $(w < y)$ be a saturated element such that $x \prec y$. We consider
\[
\mathsf Z_{w < y, L}.
\]
We define
\[
d_y = d_{\mathbb P_C(\mathcal F)} - m_B(y) - m_C(y), \quad h_{\mathcal Q, y} = 2d_y  - \mathrm{deg}(\mathcal F) - \frac{\Delta}{2}, \quad k_y = m_{\mathcal Z}(y).
\]
We also define
\[
v = \iota_{\diamond}(p_{\mathcal Z}^*y_{\mathcal Z} - p_{\mathcal Z}^*y_{\mathcal Z}\wedge p_{C}^*y_{C})
\]
which induces admissible jets over $C$ of length $k_y$.
We also define
\[
\mathsf L_{y, L} = L(-y_C).
\]
Let $E_{v, \mathsf L_{y, L}}$ be the subspace of elements in $E_{\mathsf L_{y, L}}$ satisfying the incidence condition imposed by $v$.
The expected dimension of $E_{v, \mathsf L_{y, L}}$ is given by 
\begin{align*}
&h_{\mathcal Q, y} +\frac{\Delta}{2}- k_y + 2(1-g(C)) = h_{\mathcal Q}(\alpha) + \frac{\Delta}{2} - \gamma(y)+ 2(1-g(C))  = \dim E_{w, L} - \gamma(w < y) \\
&\geq 8I  + \mathsf c_3(\mathcal F)+\frac{\Delta}{2} -2g(C).
\end{align*}
Moreover, $\mathsf Z_{w < y, L}$ is identified with $E_{v, \mathsf L_{y, L}}$.
When $y$ varies in $\mathcal N_{T'}$,
the dimension of the locus of $y$ is at most $m_B(y)/2  + k_y/2 + m_C(y)$. Using the argument of Lemmas~\ref{lemm:dimensionestimates}, Proposition~\ref{prop:dominantsmooth}, and Lemma~\ref{lemm:dimensionestimates3} the dimension of $y$ such that $\mathsf Z_{w < y, L}$ does not have expected dimension for some $L$ is at most
\begin{align*}
&\max\{m_B(y)/2  + k_y + m_C(y) - (d_y-\deg(\mathcal F)+1 - 7g(C) + \mathrm{neg}(\mathbb P(\mathcal F)/C)),\\
& m_B(y)/2 + m_C(y) + k_y/2 -  (h_{\mathcal Q, y} - k_y  - \mathsf c_3(\mathcal F))\}\\
&\leq \max\{\gamma(y)/2 - ((h_{\mathcal Q, y} - k_y - \mathsf c_3(\mathcal F))/2), \gamma(y)/2 - 8I-2\}\\
&\leq \max\{\gamma(y)/2 - 4I+1), \gamma(y)/2 - 8I+2\}  \leq k(\alpha)   -3I+3 .
\end{align*}
Thus the codimension of the loci of $w$ such that there exists $w < y$ and $L$ whose corresponding $\mathsf Z_{w < y, L}$ does not have expected dimension is at least $3I-3$.

Let $F_2$ be the union of Galois orbits of the closures of such loci over $\overline{k}$ where the combinatorial type $T$ of $x$ runs over all combinatorial types of $P$. Then the codimension of $F_2$ in $U_{k(\alpha)}$ is greater than or equal to $3I-3$. 
Let $F = F_1 \cup F_2$.

Next, we analyze $\kappa(T)$ as defined in Section~\ref{subsec:barcomplex}.
\begin{lemm}
We have $\gamma(T) \leq 2 \mathrm{rank}(T) \leq 2\kappa(T)$.
\end{lemm}
\begin{proof}
We prove $\mathrm{rank}(T) \leq \kappa(T)$.
We fix $b \in B(\overline{k})$. It suffices to show that for any $(f_b < f_b')$, we have
\[
\mathrm{rank}(f_b < f_b') \leq \kappa(f_b < f_b').
\]
As mentioned before, we have
\[
 \mathrm{rank}(f_b < f_b') = \frac{3}{2}m_B(f_b < f_b') + \frac{1}{2}m_{\mathcal Z}(f_b < f_b') + \frac{3}{2}m_C(f_b < f_b').
\]
Hence we have
\[
  \kappa (f_b < f_b') = \frac{5}{2}m_B(f_b < f_b') + \frac{3}{2}m_{\mathcal Z}(f_b < f_b') +  \frac{5}{2}m_C(f_b < f_b')  - 2|\mathrm{Supp}(f_b < f_b')\cap B^\circ|.
\]
Using this one can show that if we have $x_1 \prec x_2$, then we have $\kappa(x_2)\geq \kappa(x_1) +1$.
Thus our assertion follows by induction.
The other inequality is clear.
\end{proof}

We conclude that
\[
\gamma(T)  \leq 2\kappa(T).
\]
In particular, if $\kappa(T) \leq I$, then $\gamma(T) \leq 2I$. Thus $T$ is a combinatorial type of $P$.

We assume that our ground field $k$ is $\mathbb F_q$.
Let $\widetilde{M}_\alpha \to M_\alpha$ be the $\mathbb G_m$-torsor as before.
Then we have
\[
\#\widetilde{M}_\alpha(k) = (q-1)\#M_{\alpha}(k).
\]
Thus our goal is to understand $\#\widetilde{M}_\alpha(k)$.
By the Grothendieck--Lefschetz trace formula, we have
\[
\#\widetilde{M}_\alpha(k) = \sum_i (-1)^i\mathrm{Tr}(\mathrm{Frob}\curvearrowright \rH^i_{\text{\'et}, c}(\widetilde{M}_{\alpha, \overline{k}}, \mathbb Q_\ell)).
\]
By combining Corollary~\ref{coro:FengHaseLiu}, Deligne's estimates, and Leray's spectral sequence,
we have
\begin{equation}
\label{equation:error1}
\left| \sum_{i < 2h(\alpha) +6 -4g(B)  - I }(-1)^i\mathrm{Tr}(\mathrm{Frob}\curvearrowright \rH^i_{\text{\'et}, c}(\widetilde{M}_{\alpha, \overline{k}}, \mathbb Q_\ell))\right| = O(q^{h(\alpha)+3-2g(B)  -I/2 }\mathsf C_1^{h(\alpha)}),
\end{equation}
where $\mathsf C_1$ is a constant depending only on $\mathcal X_{\overline{k}}$ and the implied constant is independent of $q, \alpha, I$.

Let $U = (U_{k(\alpha)}\setminus F) \times \Pic^{d_{\mathbb P_C(\mathcal F)}}(C)$. The codimension of $F$ in $U_{k(\alpha)}$ is greater than or equal to $3I-3$, and $\widetilde{M}_\alpha$ is flat over $U_{k(\alpha)}$ assuming $I$ is sufficiently large, so the inclusion
\[
\widetilde{M}_\alpha|_U \hookrightarrow \widetilde{M}_\alpha
\]
induces isomorphisms on \'etale cohomology with compact support in codimension $\leq 5I$. This implies that assuming $I$ is sufficiently large, we have
\begin{align*}
&\sum_{i \geq 2h(\alpha) +6 -4g(B) - I }(-1)^i\mathrm{Tr}(\mathrm{Frob}\curvearrowright \rH^i_{\text{\'et}, c}(\widetilde{M}_{\alpha, \overline{k}}, \mathbb Q_\ell))\\&= \sum_{i \geq 2h(\alpha) +6 -4g(B) - I }(-1)^i\mathrm{Tr}(\mathrm{Frob}\curvearrowright \rH^i_{\text{\'et}, c}((\widetilde{M}_{\alpha}|_U)_{\overline{k}}, \mathbb Q_\ell)).
\end{align*}
Also we have
\[
E|_U = \widetilde{M}_\alpha|_U \sqcup \mathrm{im}(\mathsf Z_{P_U} \to E|_U).
\]
It follows from (\ref{equation:error1}) that we have
\begin{align*}
&\sum_{i \geq 2h(\alpha) +6 -4g(B) - I}(-1)^i\mathrm{Tr}(\mathrm{Frob}\curvearrowright \rH^i_{\text{\'et}, c}((\widetilde{M}_{\alpha}|_U)_{\overline{k}}, \mathbb Q_\ell))\\
& = \sum_{i \geq 2h(\alpha)+6 -4g(B)  - I }(-1)^i\mathrm{Tr}(\mathrm{Frob}\curvearrowright \rH^i_{\text{\'et}, c}((E|_U)_{\overline{k}}, \mathbb Q_\ell))\\
&- \sum_{i \geq 2h(\alpha)+6 -4g(B)- I }(-1)^i\mathrm{Tr}(\mathrm{Frob}\curvearrowright \rH^i_{\text{\'et}, c}(\mathrm{im}(\mathsf Z_{P_U} \to E|_U)_{\overline{k}}, \mathbb Q_\ell))+ O(q^{h(\alpha) + 3 -2g(B)  -I/2 +1}\mathsf C_1^{h(\alpha)}).
\end{align*}
We analyze 
\[
\sum_{i \geq 2h(\alpha) +6 -4g(B) - I }(-1)^i\mathrm{Tr}(\mathrm{Frob}\curvearrowright \rH^i_{\text{\'et}, c}(\mathrm{im}(\mathsf Z_{P_U} \to E|_U)_{\overline{k}}, \mathbb Q_\ell)).
\]
By Theorem~\ref{theo:Theorem7.5DLTT}, this is equal to
\[
\sum_{i \geq 2h(\alpha) +6 -4g(B)- I }(-1)^i\mathrm{Tr}(\mathrm{Frob}\curvearrowright \rH^i_{\text{\'et}, c}(B(P_U, \mathsf Z_{P_U})_{\overline{k}}, \mathbb Q_\ell)).
\]
It follows from Theorem~\ref{theo:Theorem7.3DLTT} and \cite[Lemma 2.3]{DLTT25} that we have
\begin{align*}
&\Big|\sum_{i \geq 2h(\alpha) +6 -4g(B)- I }(-1)^i\mathrm{Tr}(\mathrm{Frob}\curvearrowright \rH^i_{\text{\'et}, c}(B(P_U, \mathsf Z_{P_U})_{\overline{k}}, \mathbb Q_\ell)) -\\
&\sum_{T : \textnormal{ess. type of $P$}} \sum_{i \geq 2h(\alpha) +6 -4g(B) - I }(-1)^i\mathrm{Tr}(\mathrm{Frob}\curvearrowright \rH^i_{\text{\'et}, c}((\mathsf Z|_{U\times_{U_{k(\alpha)}} \mathcal N_T})_{\overline{k}}, \mu'(T)[1]\otimes \mathbb Q_\ell))\Big| \leq N,
\end{align*}
where $N$ is the $\mathsf L^1$-trace of
\[
\bigoplus_{T : \textnormal{ess. type of $P$}} \rH^{\lceil2h(\alpha)+6 -4g(B)  - I \rceil}_{\text{\'et}, c}((\mathsf Z|_{U\times_{U_{k(\alpha)}} \mathcal N_T})_{\overline{k}}, \mu'(T)[1]\otimes \mathbb Q_\ell).
\]
Next we bound this $N$. Since $\mathsf Z$ has expected codimension over $U$, for $$i \geq 2h(\alpha) +6 -4g(B) - I ,$$ we have
\[
\rH^{i}_{\text{\'et}, c}((\mathsf Z|_{U\times_{U_{k(\alpha)}} \mathcal N_T})_{\overline{k}}, \mu'(T)[1]\otimes \mathbb Q_\ell) \cong \rH^{i}_{\text{\'et}, c}((U\times_{U_{k(\alpha)}} \mathcal N_T)_{\overline{k}}, \mu(T)[1]\langle -n_{h(\alpha), k(\alpha), T}\rangle\otimes \mathbb Q_\ell)
\]
where $n_{h(\alpha), k(\alpha), T} = h(\alpha)+ 3-2g(B)-g(C) -k(\alpha) - \gamma(T)$.
The cohomological dimension of 
\[
\rH^{i}_{\text{\'et}, c}((F\times_{U_{k(\alpha)}} \mathcal N_T)_{\overline{k}}, \mu(T)[1]\langle -n_{h(\alpha), k(\alpha), T}\rangle\otimes \mathbb Q_\ell)
\]
is at most
\begin{align*}
&2k(\alpha) + 2g(C) -4I  + 2|\mathrm{Supp}(T)\cap B^\circ| + \mathrm{rank}(T) + 2n_{h(\alpha), k(\alpha), T}\\
&\leq 2h(\alpha) - \kappa(T) -4I +6-4g(B).
\end{align*}
Thus, assuming $I$ is sufficiently large, for $i \geq 2h(\alpha) +6 -4g(B) - I$, we have
\[
\rH^{i}_{\text{\'et}, c}((U\times_{U_{k(\alpha)}} \mathcal N_T)_{\overline{k}}, \mu(T)[1]\langle -n_{h(\alpha), k(\alpha), T}\rangle\otimes \mathbb Q_\ell) \cong \rH^{i}_{\text{\'et}, c}((\mathcal N_T)_{\overline{k}}, \mu(T)[1]\langle -n_{h(\alpha), k(\alpha), T}\rangle\otimes \mathbb Q_\ell).
\]
Thus using Theorem~\ref{theo:keyestimate}, we have
\begin{prop}
There exists some uniform constant $\mathsf c_4$ depending only on $\mathcal F$ such that assuming $q > 2^{22}E^2$,
for any constant $\mathsf C_2 > 4E$, 
\[
N = O(q^{h(\alpha) + \mathsf c_4}(q/16)^{-I/2}\mathsf C_2^{h(\alpha)}).
\]
\end{prop}
Hence, we conclude that
\begin{align*}
&\sum_{i \geq 2h(\alpha)  +6 -4g(B)- I }(-1)^i\mathrm{Tr}(\mathrm{Frob}\curvearrowright \rH^i_{\text{\'et}, c}(B(P_U, \mathsf Z_{P_U})_{\overline{k}}, \mathbb Q_\ell))\\
&=\sum_{T : \textnormal{ess. type of $P$}} \sum_{i \geq 2h(\alpha) +6 -4g(B) - I }(-1)^i\mathrm{Tr}(\mathrm{Frob}\curvearrowright \rH^i_{\text{\'et}, c}(( \mathcal N_T)_{\overline{k}}, \mu(T)[1]\langle -n_{h(\alpha), k(\alpha), T}\rangle\otimes \mathbb Q_\ell))\\
&\qquad \qquad\qquad\qquad\qquad \qquad\qquad\qquad + O(q^{h(\alpha)  + \max\{4-2g(B), \mathsf c_4\}}(q/16)^{-I/2}\max\{\mathsf C_1, \mathsf C_2\}^{h(\alpha)}).
\end{align*}
When $T$ is an essential type which is not a type of $P$, we have $\kappa(T) > I$.
This implies that the cohomological dimension of 
\[
\rH^i_{\text{\'et}, c}(( \mathcal N_T)_{\overline{k}}, \mu(T)[1]\langle -n_{h(\alpha), k(\alpha), T}\rangle\otimes \mathbb Q_\ell))
\]
is less than
\[
2h(\alpha)+6 -4g(B)  - I.
\]
Hence we have
\begin{align*}
&\sum_{T : \textnormal{ess. type of $P$}} \sum_{i \geq 2h(\alpha) +6 -4g(B)- I }(-1)^i\mathrm{Tr}(\mathrm{Frob}\curvearrowright \rH^i_{\text{\'et}, c}(( \mathcal N_T)_{\overline{k}}, \mu(T)[1]\langle -n_{h(\alpha), k(\alpha), T}\rangle\otimes \mathbb Q_\ell))\\
&=\sum_{T : \textnormal{ess. type}} \sum_{i \geq 2h(\alpha) +6 -4g(B) - I }(-1)^i\mathrm{Tr}(\mathrm{Frob}\curvearrowright \rH^i_{\text{\'et}, c}(( \mathcal N_T)_{\overline{k}}, \mu(T)[1]\langle -n_{h(\alpha), k(\alpha), T}\rangle\otimes \mathbb Q_\ell)).
\end{align*}
Using Theorem~\ref{theo:keyestimate} again, we obtain
\begin{align*}
&\sum_{T : \textnormal{ess. type}} \sum_{i \geq 2h(\alpha) +6 -4g(B)- I }(-1)^i\mathrm{Tr}(\mathrm{Frob}\curvearrowright \rH^i_{\text{\'et}, c}(( \mathcal N_T)_{\overline{k}}, \mu(T)[1]\langle -n_{h(\alpha), k(\alpha), T}\rangle\otimes \mathbb Q_\ell))\\
&=\sum_{T : \textnormal{ess. type}} \sum_{i}(-1)^i\mathrm{Tr}(\mathrm{Frob}\curvearrowright \rH^i_{\text{\'et}, c}(( \mathcal N_T)_{\overline{k}}, \mu(T)[1]\langle -n_{h(\alpha), k(\alpha), T}\rangle\otimes \mathbb Q_\ell))\\
&\qquad\qquad\qquad\qquad+ O(q^{h(\alpha)  + \max\{4-2g(B), \mathsf c_4\}}(q/16)^{-I/2}\max\{\mathsf C_1, \mathsf C_2\}^{h(\alpha)})\\
& = - \#\Pic^0(C)(k)q^{h_{\mathcal Q}(\alpha) +3-2g(B)-g(C)} \sum_{(w < x) \in (U_{k(\alpha)}(k) < \mathfrak H^\circ(k))} \mu_k(w, x)q^{-\gamma(x)}\\
&\qquad\qquad\qquad\qquad+ O(q^{h(\alpha)    + \max\{4-2g(B), \mathsf c_4\}}(q/16)^{-I/2}\max\{\mathsf C_1, \mathsf C_2\}^{h(\alpha)}),
\end{align*}
where we use Proposition~\ref{prop:Proposition7.8DLTT} for the last equality.

Similarly, one can prove that
\begin{align*}
& \sum_{i \geq 2h(\alpha) +6-4g(B) - I }(-1)^i\mathrm{Tr}(\mathrm{Frob}\curvearrowright \rH^i_{\text{\'et}, c}((E|_U)_{\overline{k}}, \mathbb Q_\ell))\\
& = \#\Pic^0(C)(k)q^{h_{\mathcal Q}(\alpha) +3-2g(B)-g(C)} \sum_{w \in U_{k(\alpha)}(k)} q^{-\gamma(w)}+ O(q^{h(\alpha)   + \max\{4-2g(B), \mathsf c_4\}}(q/16)^{-I/2}\max\{\mathsf C_1, \mathsf C_2\}^{h(\alpha)}).
\end{align*}
Thus we conclude
\begin{theo}
\label{theo:homologicalsieve}
There exist constants $\mathsf C_* > 0$ and $\mathsf c_*$ depending only on $B_{\overline{k}}$ and $\mathcal X_{\overline{k}}$ such that assuming $I$ is sufficiently large and $q > 2^{22}E^2$,
we have
\begin{align*}
&\#\widetilde{M}_\alpha(k) = \\
&\#\Pic^0(C)(k)q^{h_{\mathcal Q}(\alpha) +3-2g(B)-g(C)} \sum_{(w \leq x) \in (U_{k(\alpha)}(k) \leq \mathfrak H^\circ(k))} \mu_k(w, x)q^{-\gamma(x)}
+ O(q^{h(\alpha)  + \mathsf c_*}(q/16)^{-I/2}\mathsf C_*^{h(\alpha)}),
\end{align*}
where the implied constant is independent of $q, \alpha, I$.
\end{theo}
\subsection{The virtual height zeta function}
\label{subsec:heightzeta}

The goal of this section is to establish the asymptotic of $\#\Sec(\mathcal X/B, \alpha)(\mathbb F_q)$.

\subsubsection*{Peyre's constants}
Here we recall the definition of Peyre's constants in our setting. Let $k = \mathbb F_q$ of characteristic $\neq 2$ and $B$ be a geometrically integral smooth projective curve defined over $k$. Let $K(B)$ be the function field of $B$ and let $Q$ be a non-split smooth quadric surface defined over $K(B)$. Let $Z \in Q$ be a degree $4$ closed point such that $Z_{\overline{k}}$ remains integral. Let $\phi : X \to Q$ be the blow-up of $Q$ along $Z$. We assume that $X$ is a smooth quartic del Pezzo surface over $K(B)$.

Let $\pi : \mathcal X \to B$ be a smooth wonderful model of $X$. 
We are interested in $\Pic(X_{K(B)^s})$ where $K(B)^s$ denotes the separable closure of $K(B)$. Note that this decomposes 
\[
\Pic(X_{K(B)^s}) = \mathsf M_e\oplus \mathsf M_p,
\]
as Galois modules where $\mathsf M_e$ is the module spanned by exceptional divisors of $\phi_{K(B)^s}$ and $\mathsf M_p$ is the pullback of $\Pic(Q_{K(B)^s})$.

We define local convergence factors to be for any $b \in |B|$, 
\[
\lambda_b = \det(1 - q_b^{-1}\mathrm{Fr}_{\overline{k(b)}})^{-1}
\]
where $\mathrm{Fr}_{\overline{k(b)}}$ is the geometric Frobenius acting on $\Pic(\mathcal X_{\overline{b}})\otimes \mathbb Q$. We also define
\[
\lambda_{e, b} = 
\det(1 - q_b^{-1}(\mathrm{Fr}_{\overline{k(b)}} \curvearrowright \mathsf M_{e, b}\otimes \mathbb Q))^{-1}
\]
and
\[
\lambda_{p, b} = 
\det(1 - q_b^{-1}(\mathrm{Fr}_{\overline{k(b)}} \curvearrowright \mathsf M_{p, b} \otimes \mathbb Q))^{-1}
\]
where $\Pic(\mathcal X_{\overline{b}}) = \mathsf M_{e, b} \oplus \mathsf M_{p, b}$ is the similar decomposition as before.
By the definition, we have $\lambda_b = \lambda_{e, b}\lambda_{p, b}$.
We also define
\begin{align*}
&L_{ *}(1, \Pic(X_{K(B)^s})\otimes \mathbb Q) = \lim_{t \to 1} (1-t)^2\prod_{b \in |B|}  \det(1 - q_b^{-1}t^{|b|}\mathrm{Fr}_{\overline{k(b)}})^{-1}\\
&L_{*}(1, M_e\otimes \mathbb Q) = \lim_{t \to 1} (1-t)\prod_{b \in |B|} \det(1 - q_b^{-1}t^{|b|}(\mathrm{Fr}_{\overline{k(b)}} \curvearrowright \mathsf M_{e, b}\otimes \mathbb Q))^{-1}\\
&L_{*}(1, M_p\otimes \mathbb Q) = \lim_{t \to 1} (1-t)\prod_{b \in |B|}  \det(1 - q_b^{-1}t^{|b|}(\mathrm{Fr}_{\overline{k(b)}} \curvearrowright \mathsf M_{p, b}\otimes \mathbb Q))^{-1}.
\end{align*}
Again note that we have $$L_{*}(1, \Pic(X_{K(B)^s})\otimes \mathbb Q) = L_{*}(1, M_e\otimes \mathbb Q)L_{*}(1, M_p\otimes \mathbb Q).$$
We also define
\begin{align*}
&L(t, M_e\otimes \mathbb Q) = \prod_{b \in |B|}  \det(1 - q_b^{-1}t^{|b|}(\mathrm{Fr}_{\overline{k(b)}} \curvearrowright \mathsf M_{e, b}\otimes \mathbb Q))^{-1}\\
&L(t, M_p\otimes \mathbb Q) = \prod_{b \in |B|}  \det(1 - q_b^{-1}t^{|b|}(\mathrm{Fr}_{\overline{k(b)}} \curvearrowright \mathsf M_{p, b}\otimes \mathbb Q))^{-1}.
\end{align*}
Note that $L(t, M_p\otimes \mathbb Q)$ coincides with the Hasse--Weil zeta function $\mathsf Z_C(q^{-1}t)$ of $C$, defined by
\[
\mathsf Z_C(q^{-1}t) = \prod_{c \in |C|}(1-(q^{-1}t)^{|c|})^{-1}.
\]
From this, we can conclude that
\begin{equation}
\label{equation:Hasse--Weil}
L_{*}(1, M_p\otimes \mathbb Q) = \frac{\#\Pic^0(C)(k)}{q^{g(C)}(1-q^{-1})}.
\end{equation}

The following constant is the main ingredient of Peyre's constant:

\begin{defi}
We define the Tamagawa number of $\pi : \mathcal X \to B$ by
\[
\tau_{-K_{\mathcal X/B}}(\mathcal X) =q^{2(1-g(B))} L_{*}(1, \Pic(X_{K(B)^s})\otimes \mathbb Q )\prod_{b \in |B|}\lambda_b^{-1}\#\mathcal X_b^{\mathrm{sm}}(k(b)),
\]
where $\mathcal X_b^{\mathrm{sm}}$ is the smooth locus of $\mathcal X_b$.
This Euler product absolutely converges. See \cite{Peyre} or \cite{CLT10} for more details.
\end{defi}

An important observation is
\begin{lemm}
\label{lemm:flopsTamagawa}
Let $\pi : \mathcal X \to B$ and $\pi' : \mathcal X' \to B$ be two smooth models of $X$.
Assume that two models are connected by a sequence of $\pi$-relative flops. Then we have
\[
\tau_{-K_{\mathcal X/B}}(\mathcal X) = \tau_{-K_{\mathcal X'/B}}(\mathcal X').
\]
\end{lemm}
\begin{proof}
As explained in \cite{Peyre} and \cite{CLT10}, $\mathcal O(-K_{\mathcal X/B})$ defines an adelic metric on $\omega_X^{-1}$, and this defines the Tamagawa measure $\tau_X$ on the adelic space $X(\mathbb A_{K(B)})$. Using this, we have
\[
\tau_{-K_{\mathcal X/B}}(\mathcal X)  = \int_{X(\mathbb A_{K(B)})} \mathrm d\tau_X.
\]
Since two models are connected by a sequence of flops, we have a common resolution
\[
\xymatrix{
& \widetilde{\mathcal X} \ar[ld]_f \ar[rd]^g& \\
\mathcal X & & \mathcal X'
}
\]
such that we have $-f^*K_{\mathcal X/B} \sim -g^*K_{\mathcal X'/B}$.
Moreover, one may assume that $\widetilde{\mathcal X}$ is an integral model of $X$.
Thus two divisors define the same adelic metric. Thus our assertion follows.
\end{proof}

For a later application, we will need the following definition as well:

\begin{defi}
    Let $\pi : \mathcal X \to B$ be a smooth wonderful model of $X$.
    To the vector space of real $1$-cycles $N_1(\mathcal X_\eta)$, we assign the Lebesgue measure such that the fundamental domain of the lattice $N_1(\mathcal X_\eta) \cap N_1(\mathcal X)_{\mathbb Z}$
    has volume equal to $1$. Let $\mathsf C \subset \mathrm{Nef}_1(\mathcal X_\eta) \subset N_1(\mathcal X_\eta)$ be a closed cone. We define the alpha constant of $\mathsf C$ as
    \[
    \alpha(\mathcal X_\eta, \mathsf C) = \left( \dim \, N_1(\mathcal X_\eta) \right) . \mathrm{vol}(\{\alpha \in \mathsf C \, | \, -K_{\mathcal X/B}.\alpha \leq 1\}).
    \]
    \end{defi}
    
    We should note that in our setting, Peyre's constant is
    \[
     \alpha(\mathcal X_\eta, \Nef_1(\mathcal X_\eta))\tau_{-K_{\mathcal X/B}}(\mathcal X).
    \]
    
    \subsubsection*{The virtual height zeta function}
We closely follow the exposition of \cite[Section 8.4]{DLTT25}.
By Theorem~\ref{theo:homologicalsieve}, our goal is to understand the limiting behavior of
\[
(\#\Pic^0(C)(k))q^{h_{\mathcal Q}(\alpha) +3-2g(B)-g(C)} \sum_{(w \leq x) \in (U_{k(\alpha)}(k) \leq \mathfrak H^\circ(k))} \mu_k(w, x)q^{-\gamma(x)},
\]
as $h_{\mathcal Q}(\alpha), k(\alpha) \to \infty$. To this end, we consider the following virtual height zeta function:
\[
\mathsf Z(t) = \sum_{\mathsf k = 0}^\infty q^{\mathsf k}\left(\sum_{(w \leq x) \in (U_{\mathsf k}(k) \leq \mathfrak H^\circ(k))} \mu_k(w, x)q^{-\gamma(x)}\right)t^{\mathsf k}.
\]
It follows from \cite[Lemma 7.7(1)]{DLTT25} that the above zeta function becomes the Euler product:
\[
\mathsf Z(t) = \prod_{b \in |B|}\left( \sum_{\mathsf k = 0}^\infty (qt)^{\mathsf k|b|} \sum_{(w \leq x) \in (U_{\mathsf k|b|}(k(b)) \leq \mathfrak H^\circ(k(b)))_b} \mu_{k(b)}(w, x)q^{-\gamma(x)|b|}\right),
\]
where $(U_{\mathsf k|b|, }(k(b)) \leq \mathfrak H^\circ(k(b)))_b$ is the subposet consisting of elements supported over $b \in B$. When $t = 1$,  the $q$-expansion of each Euler factor looks
\[
1 + \sum_{w \in (U_{|b|})_b(k(b))} q^{-|b|} + \cdots .
\]
This shows that for any $\delta > 0$, the above zeta function absolutely converges when $|t| \leq q^{-\delta}$. Moreover we have the following proposition:

\begin{prop}
\label{prop:Tamagawanumber}
There exists $\delta > 0$ such that
the following function
\[
L(t, \mathsf M_e\otimes \mathbb Q)^{-1}\mathsf Z(t)
\]
absolutely converges when $|t|\leq q^\delta$.
Moreover, we have
\[
\lim_{t \to 1}(1-t)\mathsf Z(t) = L_{*}(1, M_e\otimes \mathbb Q) \prod_{b \in |B|}\lambda_b^{-1}\frac{\#\mathcal X_b^{\mathrm{sm}}(k(b))}{q^{2|b|}}.
\]
\end{prop}

\begin{proof}
Since the $q$-expansion of $L(t, \mathsf M_e\otimes \mathbb Q)$ looks
\[
1 + \sum_{w \in \mathcal Z_b(k(b))} q^{-|b|} + \cdots,
\]
and we have
\[
\sum_{w \in \mathcal Z_b(k(b))} q^{-|b|} = \sum_{w \in (U_{|b|})_b(k(b))} q^{-|b|},
\]
the first assertion is clear. 

For the second assertion, note that we have
\begin{align*}
\lim_{t \to 1}(1-t)\mathsf Z(t) & = \lim_{t \to 1}(1-t)L(t, M_e\otimes \mathbb Q) L(t, M_e\otimes \mathbb Q)^{-1}\mathsf Z(t)\\ 
&=L_{*}(1, M_e\otimes \mathbb Q) \prod_{b \in |B|}\lambda_{e, b}^{-1} \left( \sum_{\mathsf k = 0}^\infty q^{\mathsf k|b|} \sum_{(w \leq x) \in (U_{\mathsf k|b|}(k(b)) \leq \mathfrak H^\circ(k(b)))_b} \mu_{k(b)}(w, x)q^{-\gamma(x)|b|}\right).
\end{align*}
Thus it suffices to verify that for any $b \in |B|$, we have
\[
\sum_{\mathsf k = 0}^\infty q^{\mathsf k|b|} \sum_{(w \leq x) \in (U_{\mathsf k|b|}(k(b)) \leq \mathfrak H^\circ(k(b)))_b} \mu_{k(b)}(w, x)q^{-\gamma(x)|b|}= \lambda_{p, b}^{-1}\frac{\#\mathcal X_b^{\mathrm{sm}}(k(b))}{q^{2|b|}}.
\]
Let us demonstrate this computation when $\mathcal Z$ and $C$ are unramified over $b \in B$.
First let us assume that $C$ is unramified and splits over $b$.
Let $r = \#\mathcal Z_b(k(b))$. 
Then we have
\[
\#\mathcal X_b(k(b)) = q^{2|b|}+ (2+r)q^{|b|} + 1.
\]
On the other hand, we have 
\[
\lambda_{p, b}^{-1} = (1-q^{-|b|})^2.
\]
Hence we conclude that
\[
\lambda_{p, b}^{-1} \frac{\#\mathcal X_b(k(b)) }{q^{2|b|}} = 1 + rq^{-|b|} - 2(1 + r)q^{-2|b|} + rq^{-3|b|} + q^{-4|b|}.
\]
On the other hand, by \cite[Lemma 7.7 (2)]{DLTT25}, $\mu_{k(b)}(w, x)$ is only non-zero when $(w \leq x)$ is essential. Using the description in Example~\ref{exam:essentialtypes}, we have
\[
\sum_{\mathsf k = 0}^\infty q^{\mathsf k|b|} \sum_{(w \leq x) \in (U_{\mathsf k|b|}(k(b)) \leq \mathfrak H^\circ(k(b)))_b} \mu_{k(b)}(w, x)q^{-\gamma(x)|b|}= 1 + rq^{-|b|} - 2(1 + r)q^{-2|b|} + rq^{-3|b|} + q^{-4|b|}.
\]

Next assume that $C$ is non-split over $b$.
Then we have
\[
\#\mathcal X_b(k(b)) = q^{2|b|} + rq^{|b|} + 1.
\]
On the other hand, we have 
\[
\lambda_{p, b}^{-1} = 1-q^{-2|b|}.
\]
Hence we conclude that
\[
\lambda_{p, b}^{-1} \frac{\#\mathcal X_b(k(b)) }{q^{2|b|}} = 1 + rq^{-|b|}  - rq^{-3|b|} - q^{-4|b|}.
\]
Using the description in Example~\ref{exam:essentialtypes}, we have
\[
\sum_{\mathsf k = 0}^\infty q^{\mathsf k|b|} \sum_{(w \leq x) \in (U_{\mathsf k|b|}(k(b)) \leq \mathfrak H^\circ(k(b)))_b} \mu_{k(b)}(w, x)q^{-\gamma(x)|b|}= 1 + rq^{-|b|}  - rq^{-3|b|} - q^{-4|b|}.
\]
Other cases are similar. 
\end{proof}

As a corollary, we have the following:

\begin{coro}
\label{coro:tamagawa}
There exists $\eta > 0$ depending only on $\mathcal X_{\overline{k}}$ such that
\[
\frac{\#\Pic^0(C)(k)q^{\mathsf k+2-2g(B)-g(C)}}{1-q^{-1}}\left(\sum_{(w \leq x) \in (U_{\mathsf k}(k) \leq \mathfrak H^\circ(k))} \mu_k(w, x)q^{-\gamma(x)}\right) =  \tau_{-K_{\mathcal X/B}}(\mathcal X) + O(q^{-\eta \mathsf k})
\]
\end{coro}
\begin{proof}
One may argue as \cite[Proposition 8.8]{DLTT25} using (\ref{equation:Hasse--Weil}) and Proposition~\ref{prop:Tamagawanumber}.
\end{proof}

As a conclusion, we have the following theorem
\begin{theo}
\label{theo:allhegiht}
We have
\[
\frac{\#M_{\alpha}(k)}{q^{h(\alpha)}} = \tau_{-K_{\mathcal X/B}}(\mathcal X) + O((q/16)^{-I/2}q^{ \mathsf c_*}\mathsf C_*^{h(\alpha)}) + O(q^{-\eta  k(\alpha)}).
\]
\end{theo}

\begin{proof}
This follows from Theorem~\ref{theo:homologicalsieve} and Corollary~\ref{coro:tamagawa}.
\end{proof}

\section{Main results}

Our goal in this section is to prove Theorem~\ref{theo:main}.
Let $k$ be $\mathbb F_q$ of characteristic $\geq 7$ and $B$ be a smooth geometrically integral projective curve over $k$.

Let $X$ be a smooth quartic del Pezzo surface as defined in the setup of Section~\ref{subsec:wonderful}.
This means that $X$ admits a birational morphism $\phi_1 : X \to Q_1$ to a non-split smooth quadric surface $Q_1$ over $K(B)$.

Let $\pi : \mathcal X \to B$ be a smooth wonderful model of $X$ over $B$.
Let $E_1$ be the flat closure of the exceptional divisor of $\phi$.
Let $\alpha$ be a nef class of sections on $\mathcal X$. We denote
\[
h(\alpha) = -K_{\mathcal X/B}.\alpha, \quad k(\alpha) = E_1.\alpha.
\]
Let $D_1 \subset \mathcal X$ be the flat closure of the pullback of a general hyperplane section on $Q_1$. Then there exists an integer $n_1$ such that we have
\[
-K_{\mathcal X/B} \equiv 2D_1 + n_1[\mathcal X_b] -E_1.
\]

By Proposition~\ref{prop:Dflops}, Corollary~\ref{coro:Dflopsoverk}, and Corollary~\ref{coro:blowupofquadric}, there exists a sequence of flops
\[
\psi_1 : \mathcal X \dashrightarrow \mathcal X_1,
\]
such that $\mathcal X_1$ admits a birational morphism $\phi_1 : \mathcal X_1 \to \mathcal Q_1$ to a smooth quadric bundle over $B$ which is a blow-up of a smooth multisection of degree $4$ $\mathcal Z_1 \subset \mathcal Q_1$. Let $\widetilde{D}_1$ and $\widetilde{E}_1$ be the strict transforms of $D_1$ and $E_1$ on $\mathcal X_1$ respectively.

Let $\sigma : B \to \mathcal X$ be a section of class $\alpha$ and $\sigma_1 : B \to \mathcal X_1$ be its strict transform. We denote invariants by
\[
h_1(\sigma_1), h_{\mathcal Q_1}(\sigma_1), d_{F_1(\mathcal Q_1)}(\sigma_1), k_1(\sigma_1)
\]
as before. Since $\psi_1$ is a sequence of flops, we have $h(\alpha) = h_1(\sigma_1)$. We also have the following lemma:
\begin{lemm}
\label{lemm:flopsdifference}
There exists a constant $\mathsf e_1$ depending only on $\mathcal X_{\overline{k}}$ such that for any nef class $\alpha$ of section, $\sigma : B \to \mathcal X$ a section of class $\alpha$ and $\sigma_1 : B \to \mathcal X_1$ its strict transform, we have
\[
|k(\alpha) - k_1(\sigma_1)| \leq \mathsf e_1.
\]
\end{lemm}
\begin{proof}
It suffices to show that 
\[
|D_1.\alpha - \deg \sigma_1^*\widetilde{D}_1|
\]
is uniformly bounded.
Let
\[
\xymatrix{
& \mathcal W \ar[rd]^q\ar[ld]_p& \\
\mathcal X & & \mathcal X_1
}
\]
be a common smooth resolution.
Then by negativity lemma, there exists an effective exceptional divisor $E'$ such that
\[
p^*D_1 \sim q^*\widetilde{D}_1 + E'.
\]
Our claim follows from this.
\end{proof}

Finally we prove our main theorem:

\begin{proof}[Proof of Theorem~\ref{theo:main}]
We define the functional $\ell : \Nef_1(\mathcal X) \to \mathbb R$ by
\[
\ell(\alpha) =\frac{1}{20} \max\{ \min\{h(\alpha)- k(\alpha), k(\alpha)\}, \min\{2k(\alpha) - 2h(\alpha), 3h(\alpha) - 2k(\alpha) \}\}.
\]
Let $\epsilon$ be a small positive rational number.

Recall that in the introduction we define the counting function as
\[
\mathsf N(\mathbb F_{q}, B, \mathcal X, \ell, \epsilon, d) = \sum_{\substack{\alpha \in \Nef_1(\mathcal X)_{\mathrm{sec}, \epsilon, \mathbb Z}, \\-K_{\mathcal X/B}.\alpha \leq m(\pi) + r(\pi)d}} \#\Sec(\mathcal X/B, \alpha)(\mathbb F_{q}).
\]
To this end, we consider the following counting function:
\begin{align*}
\mathsf N_1(B, \mathcal X, -K_{\mathcal X/B}, \epsilon, m(\pi) + r(\pi)d)= \sum_{\substack{\alpha \in \mathrm{Nef}_1(\mathcal X)_{\mathrm{sec}, \mathbb Z}, \\ h(\alpha) \leq m(\pi) + r(\pi)d, \\ h(\alpha) - k(\alpha)\geq 20\epsilon h(\alpha), \\ k(\alpha) \geq 20\epsilon h(\alpha)}} \#\mathrm{Sec}(\mathcal X/B, \alpha) (k),
\end{align*}
and
\begin{align*}
\mathsf N(B, \mathcal X_1, -K_{\mathcal X_1/B}, \epsilon, \mathsf c, m(\pi) + r(\pi)d)= \sum_{\substack{\alpha \in \mathrm{Nef}_1(\mathcal X_1)_{\mathrm{sec}, \, \mathbb Z}, \\ h_1(\alpha) \leq m(\pi) + r(\pi)d, \\ h_1(\alpha) - k_1(\alpha)\geq  20\epsilon h_1(\alpha) + \mathsf c, \\ k_1(\alpha) \geq 20\epsilon h_1(\alpha) + \mathsf c}} \#\mathrm{Sec}(\mathcal X_1/B, \alpha) (k).
\end{align*}
It follows from Lemma~\ref{lemm:flopsdifference} that we have
\begin{align*}
&\mathsf N(B, \mathcal X_1, -K_{\mathcal X_1/B}, \epsilon,  \mathsf e_1, m(\pi) + r(\pi)d) \\&\leq \mathsf N_1(B, \mathcal X, -K_{\mathcal X/B}, \epsilon, m(\pi) + r(\pi)d) \leq \mathsf N(B, \mathcal X_1, -K_{\mathcal X_1/B}, \epsilon,  -\mathsf e_1, m(\pi) + r(\pi)d).
\end{align*}

Next it follows from Lemma~\ref{lemm:anotherbirational} that there is another birational morphism to a non-split smooth quadric surface $\phi_2 : X \to Q_2$. Let $D_2$ be the flat closure of the pullback of a hyperplane section on $Q_2$. By Lemma~\ref{lemm:anotherbirational}, there exists some integer $m$ that
\[
D_2 \equiv 3D_1 - 2E_1 + m[\mathcal X_b].
\]
Let $E_2$ be the flat closure of the exceptional divisor of $\phi_2$.
Then there exists an integer $n_2$ such that
\[
-K_{\mathcal X/B} \equiv 2D_2 + n_2[\mathcal X_b] - E_2.
\]
These imply that there exists an integer $m'$ such that 
\[
-K_{\mathcal X/B} - E_2 \equiv  E_1 + K_{\mathcal X/B} + m'[\mathcal X_b].
\]
Applying Proposition~\ref{prop:Dflops}, Corollary~\ref{coro:Dflopsoverk}, and Corollary~\ref{coro:blowupofquadric}, we obtain a sequence of flops $\psi_2 : \mathcal X \dashrightarrow \mathcal X_2$ and a birational morphism $\phi_2 : \mathcal X_2 \to \mathcal Q_2$.
Let us define
\begin{align*}
\mathsf N_2(B, \mathcal X, -K_{\mathcal X/B}, \epsilon,  m(\pi) + r(\pi)d)= \sum_{\substack{\alpha \in \mathrm{Nef}_1(\mathcal X)_{\mathrm{sec}, \, \mathbb Z}, \\ h(\alpha) \leq m(\pi) + r(\pi)d, \\  h(\alpha) - k(\alpha)\leq -20\epsilon h(\alpha) , \\3h(\alpha) - 2k(\alpha) \geq 20\epsilon h(\alpha)}} \#\mathrm{Sec}(\mathcal X/B, \alpha) (k).
\end{align*}
Also we define
\begin{align*}
&\mathsf N(B, \mathcal X_2, -K_{\mathcal X_2/B}, \epsilon, \mathsf c, m(\pi) + r(\pi)d)= \sum_{\substack{\alpha \in \mathrm{Nef}_1(\mathcal X_2)_{\mathrm{sec}, \, \mathbb Z}, \\ h_2(\alpha) \leq m(\pi) + r(\pi)d, \\ h_2(\alpha) - k_2(\alpha)\geq 20\epsilon h_2(\alpha) + \mathsf c, \\ k_2(\alpha) \geq 20\epsilon h_2(\alpha) + \mathsf c}} \#\mathrm{Sec}(\mathcal X_2/B, \alpha) (k).
\end{align*}
Then we have
\begin{align*}
&\mathsf N(B, \mathcal X_2, -K_{\mathcal X_2/B}, \epsilon,  \mathsf e_2 + m', m(\pi) + r(\pi)d) \\&\leq \mathsf N_2(B, \mathcal X, -K_{\mathcal X/B}, \epsilon, m(\pi) + r(\pi)d) \leq \mathsf N(B, \mathcal X_2, -K_{\mathcal X_2/B}, \epsilon, -\mathsf e_2 + m', m(\pi) + r(\pi)d),
\end{align*}
and
\[
\mathsf N(\mathbb F_{q}, B, \mathcal X, \ell, \epsilon, d) = \mathsf N_1(B, \mathcal X, -K_{\mathcal X/B}, \epsilon, m(\pi) + r(\pi)d) + \mathsf N_2(B, \mathcal X, -K_{\mathcal X/B}, \epsilon, m(\pi) + r(\pi)d).
\]
Thus it suffices to consider the asymptotic of 
\[
\mathsf N(B, \mathcal X_1, -K_{\mathcal X_1/B}, \epsilon, \mathsf c, m(\pi) + r(\pi)d).
\]
To this end, let
\[
\Nef_1(\mathcal X_1)_{\mathrm{sec}, \epsilon, 1} = \left\{ \alpha \in \Nef_1(\mathcal X_1)_{\mathrm{sec}} \, \Big | \, h_1(\alpha)- k_1(\alpha) \geq 20\epsilon h_1(\alpha) + \mathsf c, \quad k_1(\alpha) \geq \epsilon 20h_1(\alpha) + \mathsf c\right\},
\]
and
\[
\Nef_1(\mathcal X_\eta)_{\epsilon, 1} = \left\{ \beta \in \Nef_1(\mathcal X_\eta) \, \Big | \,\frac{1}{20} \min\{(h(\beta)- k(\beta)), k(\beta)\} \geq \epsilon h(\beta)\right \}.
\]
Let $\mathsf C_*, \mathsf c_*$ be constants coming from Theorem~\ref{theo:homologicalsieve}.
We assume that $\epsilon$ is sufficiently small and $q^\epsilon > 16^\epsilon C_*$.
Using Theorem~\ref{theo:allhegiht}, we have
\[
\frac{\#\mathrm{Sec}(\mathcal X_1/B, \alpha)(k)}{q^{h(\alpha)}} = \tau_{-K_{\mathcal X/B}}(\mathcal X) + O((q/16)^{-\epsilon h(\alpha)}\mathsf C_*^{h(\alpha)}) + O(q^{-20\eta\epsilon h(\alpha)}).
\]
Then combining the proof of Proposition~\ref{prop:conetheorem}, the counting argument of \cite[Theorem 1.2]{DLTT25}, and the above asymptotic, we show that 
\[
\mathsf N(B, \mathcal X_1, -K_{\mathcal X_1/B}, \epsilon, \mathsf c, m(\pi) + r(\pi)d) \sim (1-q^{-1})^{-1} \alpha(\mathcal X_\eta, \Nef_1(\mathcal X_\eta)_{\epsilon, 1}) \tau_{-K_{\mathcal X_1/B}}(\mathcal X_1)q^{m(\pi) + r(\pi)d}(r(\pi)d),
\]
as $d \to \infty$. Thus our assertion follows from Lemma~\ref{lemm:flopsTamagawa}.
\end{proof}

\bigskip

\bibliographystyle{alpha}
\bibliography{stability}

\end{document}